\documentclass{article}

\usepackage{graphicx}
\usepackage{amsmath}
\usepackage{amssymb}
\usepackage{amsfonts}
\usepackage{amsthm}
\usepackage{comment}
\usepackage{bm}
\usepackage{tikz}
\usepackage{url}
\usepackage[subrefformat=parens]{subcaption}
\usepackage{soul}
\usepackage[normalem]{ulem}

\usepackage[left=3cm, right=3cm]{geometry}

\newtheoremstyle{mystyle}
    {3pt}
    {3pt}
    {\normalfont}
    {}
    {\bf}
    {.}
    { }
    {}%

\theoremstyle{mystyle}

\newtheorem{thm}{Theorem}[section]

\newtheorem{defi}{Definition}[section]
\newtheorem{lem}{Lemma}[section]
\newtheorem{cor}{Corollary}[section]
\newtheorem{alg}{Algorithm}[section]
\newtheorem{prob}{Problem}[section]
\newtheorem{ope}{Operation}[section]

\newtheorem{claim}{Claim}

\begin{document}

\title{The Computational Complexity of Classical Knot Recognition}
\author{Kazuhiro Ichihara \and Yuya Nishimura \and Seiichi Tani}
\date{}
\maketitle

\abstract{
    The classical knot recognition problem is the problem of determining whether the virtual knot represented by a given diagram is classical.
    We prove that this problem is in NP, 
    and we give an exponential time algorithm for the problem.
}


\section{Introduction}
A disjoint union of closed loops embedded in the 3-sphere is called a {\it link}.
In particular, a one-component link is called a {\it knot}.
A knot $K$ is said to be {\it trivial} if there is a disk $D \subset \mathbb{S}^3$ which satisfies $\partial D = K$.
In this paper, we assume that all links are locally flat: for any point $x$ in a link $L$, there is an open neighborhood $\mathcal{O}_x \subset \mathbb{S}^3$ such that $(\mathcal{O}_x, \mathcal{O}_x \cap L)$ is homeomorphic to $(\mathbb{R}^3, \mathbb{R})$.

Kauffman defined a {\it virtual link} by extending the notion of a link (\cite{Kau}).
We call a conventional link a {\it classical link} to distinguish it from a virtual link.
A virtual link is defined by a virtual link diagram, which contains real and virtual crossings, up to certain moves.
Note that a classical link admits a virtual link diagram containing virtual crossings.
Given a virtual link diagram, a problem of determining whether the virtual link represented by the diagram is classical is called {\it classical link recognition}.
In particular, {\it classical knot recognition} is the problem restricting inputs of classical link recognition to virtual knot diagram.

The computational complexity of problems in classical link theory has been studied by many researchers.
In particular, Hass, Lagarias and Pippenger proved that some problems in classical link theory such as unknot regcognition is in NP (\cite{HLP}), where unknot recognition is a problem of determining whether a classical knot represented by a given classical knot diagram is trivial.
Furthermore, Lackenby showed that unknot recognition is in co-NP in \cite{Lack}, and Burton and Ozlen gave a fast algorithm for unknot recognition in \cite{B_unknot}.

On the other hand, little is known about the computational complexity of problems in virtual link theory.
We show that classical knot recognition is in NP.
\begin{thm}\label{thm:main}
    Classical knot recognition is in NP.
\end{thm}
Any virtual knot $K$ is represented as a knot $\hat{K}$ in a thickened orientable closed surface $\mathcal{S} \times I$.
Kuperberg showed that $K$ is classical if and only if the genus of $\mathcal{S}$ is reduced to zero by repeating cutting the exterior $E = \text{cl}(\mathcal{S} \times I - N(\hat{K}))$ along a vertical annulus in $E$ (\cite{K}).
In the proof of Theorem \ref{thm:main}, normal vertical annuli in the exterior $E$ are used as a witness of classical knot recognition.
The key of the proof of Theorem \ref{thm:main} is to reduce the running time of cutting a triangulation $\mathcal{T}$ of the exterior $E$ along a normal vertical annulus $A$ in $E$ to polynomial time.
It takes exponential time of the number of tetrahedra  in $\mathcal{T}$ to cut $\mathcal{T}$ along $A$ because the number of normal disks in A may grow exponentially.
Therefore, we use the crushing procedure along $A$ instead of cutting $\mathcal{T}$ along $A$.
Jaco and Rubinstein defined the crushing procedure on a triangulation along a normal surface 
and analyzed its effects on the underlying $3$-manifold in the case where the normal surface is a disk or a $2$-sphere in an orientable compact $3$-manifold.
In Section $4$, we generalize this result to the setting of a normal vertical annulus in the exterior of a link in a thickened orientable closed surface.


It is known that if a problem is in NP, then there is an exponential time algorithm for the problem.
In this paper, we also give a specific exponential time algorithm for classical knot recognition,
and we prove that the time complexity of the algorithm is $\phi^{\mathcal{O}(c^4)}$,
where $c$ is the number of real crossings of a given diagram and $\phi = \frac{1+\sqrt{5}}{2}$.
However, it is expected that this bound is not sharp.
Therefore, in order to estimate a better bound, we conduct a computer experiment to measure the running time of our algorithm.
As a result, we see that the average of the running times is $2^{2.250c - 0.887}$.
From this result, it is expected that the time complexity of our algorithm is bounded by $2^{\mathcal{O}(c)}$.

This paper is organized as follows.
In Section 2, we review the definition of a virtual link.
We also define the canonical exterior of a virtual link diagram, which allows us to study virtual links using 3-manifold theory.
Then in Section 3, we give a method of constructing a triangulation of the canonical exterior when a virtual link diagram is given.
We give a brief overview of normal surface theory in Section 4.
In Section 5, we prove Theorem \ref{thm:main}, and in Section 6, we give an exponential time algorithm for classical knot recognition.
Furthermore, we conducted a computational experiment to estimate a better bound of the running time of the algorithm.
The experimental results are shown in Section 7.
In section 8, we summarize the paper and state future work.

\section{Virtual links}
\subsection{The definition of virtual links}
A {\it virtual link diagram} is a 4-regular plane graph with over/under or virtual information at each vertex.
A real crossing is expressed by cutting the string passing under as depicted in Figure \ref{fig:crossings}\subref{fig:realcrossing}, 
and a virtual crossing is expressed by drawing a small circle as depicted in Figure \ref{fig:crossings}\subref{fig:virtualcrossing}.
Let $\mathcal{VD}$ be the entire set of virtual link diagrams.
The {\it virtual Reidemeister equivalence} $\simeq_{R}$ is an equivalence relation on $\mathcal{VD}$
generated by virtual Reidemeister moves depicted in Figure \ref{fig:VRmoves}.
A {\it virtual link} is an equivalence class obtained as the quotient of $\mathcal{VD}$ by the virtual Reidemeister equivalence.
A virtual link $L$ is {\it classical} if and only if there is a diagram of $L$ which has no virtual crossings.

\begin{figure}[htbp]
    \begin{minipage}{0.49\textwidth}
        \centering
        \tikzset{every picture/.style={line width=0.75pt}} 

\begin{tikzpicture}[x=0.75pt,y=0.75pt,yscale=-1,xscale=1]

\draw    (9,8) -- (97.5,98) ;
\draw    (8,98) -- (47.5,58) ;
\draw    (58,48) -- (97.5,8) ;

\end{tikzpicture}
        \subcaption{Real crossing}
        \label{fig:realcrossing}
    \end{minipage}
    \begin{minipage}{0.49\textwidth}
        \centering
        \tikzset{every picture/.style={line width=0.75pt}} 

\begin{tikzpicture}[x=0.75pt,y=0.75pt,yscale=-1,xscale=1]

\draw    (7,7) -- (96.5,97) ;
\draw    (97.5,7) -- (7.5,97) ;
\draw   (43.75,52) .. controls (43.75,47.58) and (47.33,44) .. (51.75,44) .. controls (56.17,44) and (59.75,47.58) .. (59.75,52) .. controls (59.75,56.42) and (56.17,60) .. (51.75,60) .. controls (47.33,60) and (43.75,56.42) .. (43.75,52) -- cycle ;

\end{tikzpicture}
        \subcaption{Virtual crossing}
        \label{fig:virtualcrossing}
    \end{minipage}
    \caption{Crossings in a virtual link diagram}
    \label{fig:crossings}
\end{figure}
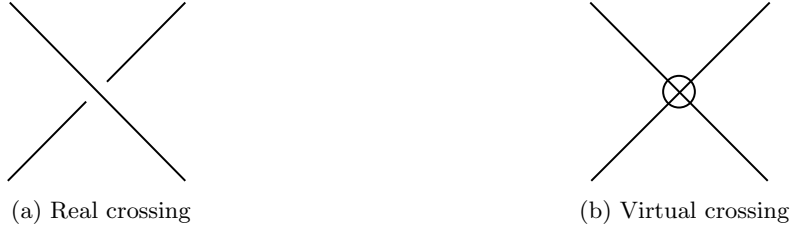

\begin{figure}[htbp]
    \centering
    \input{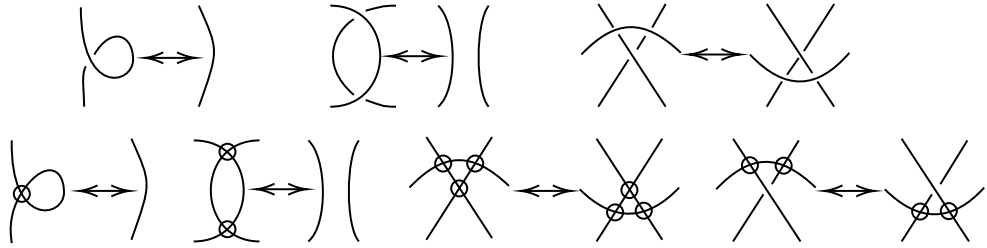}
    \caption{Virtual Reidemeister moves}
    \label{fig:VRmoves}
\end{figure}

Next, we encode a virtual link diagram to handle it on a computer.
There are several ways to encode virtual links.
Here we use oriented Gauss code.

\begin{defi}\label{def:gauss}
    Let $L$ be a virtual link and $D$ be a virtual link diagram of $L$.
    The character string obtained by the following operations is called an {\it oriented Gauss code} of $D$.
    \begin{enumerate}
        \item Assign natural numbers to all real crossings.
        \item For each component, choose an orientation and a starting point.
        \item Choose a component of $D$, then travel along the component from the starting point in the chosen orientation.
        \item Each time you come through a real crossing, write the following three symbols:
            \begin{itemize}
                \item ``$+$'' if you go over the crossing, otherwise ``$-$'',
                \item ``\verb|>|'' if you see the arc you intersect going from left to right, otherwise ``\verb|<|'',
                \item the index number of the crossing.
            \end{itemize}
        \item Write ``;'' as a components separator.
    \end{enumerate}
\end{defi}

\begin{figure}[htbp]
    \centering
    \tikzset{every picture/.style={line width=0.75pt}} 

\begin{tikzpicture}[x=0.75pt,y=0.75pt,yscale=-1,xscale=1]

\draw [color={rgb, 255:red, 0; green, 0; blue, 0 }  ,draw opacity=1 ][line width=1.5]    (55.93,99.96) -- (55.93,13) ;
\draw [shift={(55.93,10)}, rotate = 450] [color={rgb, 255:red, 0; green, 0; blue, 0 }  ,draw opacity=1 ][line width=1.5]    (14.21,-4.28) .. controls (9.04,-1.82) and (4.3,-0.39) .. (0,0) .. controls (4.3,0.39) and (9.04,1.82) .. (14.21,4.28)   ;
\draw    (62.3,54.53) -- (99.37,54.53) ;
\draw [shift={(101.37,54.53)}, rotate = 180] [color={rgb, 255:red, 0; green, 0; blue, 0 }  ][line width=0.75]    (10.93,-3.29) .. controls (6.95,-1.4) and (3.31,-0.3) .. (0,0) .. controls (3.31,0.3) and (6.95,1.4) .. (10.93,3.29)   ;
\draw    (10.5,54.53) -- (49.57,54.53) ;
\draw [color={rgb, 255:red, 0; green, 0; blue, 0 }  ,draw opacity=1 ][line width=1.5]    (174.07,99.96) -- (174.07,13) ;
\draw [shift={(174.07,10)}, rotate = 450] [color={rgb, 255:red, 0; green, 0; blue, 0 }  ,draw opacity=1 ][line width=1.5]    (14.21,-4.28) .. controls (9.04,-1.82) and (4.3,-0.39) .. (0,0) .. controls (4.3,0.39) and (9.04,1.82) .. (14.21,4.28)   ;
\draw    (180.43,54.53) -- (219.5,54.53) ;
\draw    (130.63,54.53) -- (167.7,54.53) ;
\draw [shift={(128.63,54.53)}, rotate = 0] [color={rgb, 255:red, 0; green, 0; blue, 0 }  ][line width=0.75]    (10.93,-3.29) .. controls (6.95,-1.4) and (3.31,-0.3) .. (0,0) .. controls (3.31,0.3) and (6.95,1.4) .. (10.93,3.29)   ;

\draw (61.29,35.11) node [anchor=north west][inner sep=0.75pt]    {$i$};
\draw (179.42,36.02) node [anchor=north west][inner sep=0.75pt]    {$i$};
\draw (35.48,108.72) node [anchor=north west][inner sep=0.75pt]    {$+>i$};
\draw (153.61,109.63) node [anchor=north west][inner sep=0.75pt]    {$+<i$};

\end{tikzpicture}
    \caption{The cording method of each crossing}
    \label{fig:Rmoves}
\end{figure}
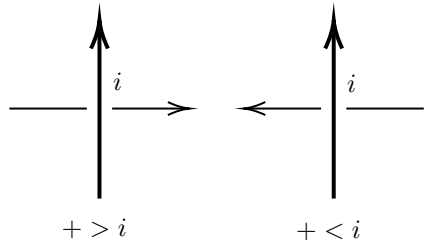

Let $D$ and $D'$ are virtual link diagrams.
$D$ and $D'$ can have the same oriented Gauss code even if $D$ and $D'$ are different virtual link diagrams.
However, $D$ and $D'$ represent the same virtual link if $D$ and $D'$ have the same oriented Gauss code (\cite{Kau}).
Therefore, we can encode virtual links with oriented Gauss codes.

When a virtual link is encoded with an oriented Gauss code, we meet each real crossing twice.
Thus, a virtual link diagram $D$ can be encoded with the length of $\mathcal{O}(c)$, where $c$ is the number of real crossings.
Therefore, the computational complexity of classical link recognition is measured by a function of the number of real crossings.

Here, we give a formal definition of classical link recognition.
\begin{prob}[Classical link recognition]
    Let $L$ be a virtual link, and let $D$ be a virtual link diagram of $L$.
    \begin{description}
        \item[Input] An oriented Gauss code of $D$
        \item[Output] $
                \left \{
                    \begin{array}{ll}
                        \mbox{yes} & \mbox{if } L \mbox{ is classical}\\
                        \mbox{no}  & \mbox{otherwise}
                    \end{array}
                \right .
                $
    \end{description}
\end{prob}



Let $L$ be a virtual link and $D$ be a virtual link diagram of $L$.
If $D$ is disconnected, then $L$ is classical if and only if each link represented by a component of $D$ is classical.
In addition, if $D$ has no real crossings, then the link represented by $D$ is a trivial link, i.e. a classical link.
For these reasons, we assume that an input of classical link recognition is a connected virtual link diagram which has at least one real crossing.


Kauffman and Manturov showed that this problem can be solved by using a problem of determining whether given two Haken manifolds are homeomorphic.
\begin{thm}[Kauffman and Manturov \cite{KM}]
        Classical link recognition is computable.
\end{thm}

\subsection{Virtual links in thickened surfaces}
Because the definition of a virtual link is based on a diagram, it is easy to calculate polynomial invariants of virtual links.
By contrast, with this definition, it is not straightforward to discuss using 3-manifold theory like classical link theory.

The study of virtual links using 3-manifold theory was based on \cite{KK, CKS}.
The method is as follows.
Let $D$ be a virtual link diagram.
First, we place each real crossing on a surface as shown in the left of Figure \ref{fig:absCrossing} 
and place each virtual crossing on a surface as shown in the right of Figure \ref{fig:absCrossing}.
By gluing adjacent surfaces depicted in Figure \ref{fig:abstract}, we obtain an orientable surface $N(\tilde{D})$ and a diagram $\tilde{D}$ which has no virtual crossings on $N(\tilde{D})$.
The pair $(N(\tilde{D}), \tilde{D})$ is called an {\it abstract link diagram} of $D$.
Note that $N(\tilde{D})$ may be disconnected even though $D$ is connected.
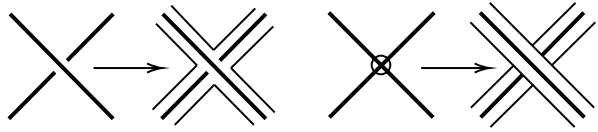
\begin{figure}[htbp]
    \centering
    \tikzset{every picture/.style={line width=0.75pt}} 

\begin{tikzpicture}[x=0.75pt,y=0.75pt,yscale=-0.7,xscale=0.7]

\draw [color={rgb, 255:red, 0; green, 0; blue, 0 }  ,draw opacity=1 ][line width=1.5]    (8.84,18) -- (83.15,93.57) ;
\draw [color={rgb, 255:red, 0; green, 0; blue, 0 }  ,draw opacity=1 ][line width=1.5]    (8,93.57) -- (41.17,59.98) ;
\draw [color={rgb, 255:red, 0; green, 0; blue, 0 }  ,draw opacity=1 ][line width=1.5]    (49.98,51.59) -- (83.15,18) ;
\draw [color={rgb, 255:red, 0; green, 0; blue, 0 }  ,draw opacity=1 ][line width=1.5]    (238.51,17) -- (313.66,92.57) ;
\draw [color={rgb, 255:red, 0; green, 0; blue, 0 }  ,draw opacity=1 ][line width=1.5]    (314.5,17) -- (238.93,92.57) ;
\draw  [color={rgb, 255:red, 0; green, 0; blue, 0 }  ,draw opacity=1 ] (269.37,54.79) .. controls (269.37,51.08) and (272.37,48.07) .. (276.08,48.07) .. controls (279.79,48.07) and (282.8,51.08) .. (282.8,54.79) .. controls (282.8,58.5) and (279.79,61.5) .. (276.08,61.5) .. controls (272.37,61.5) and (269.37,58.5) .. (269.37,54.79) -- cycle ;
\draw [color={rgb, 255:red, 0; green, 0; blue, 0 }  ,draw opacity=1 ][line width=1.5]    (118.84,18) -- (193.15,93.57) ;
\draw [color={rgb, 255:red, 0; green, 0; blue, 0 }  ,draw opacity=1 ][line width=1.5]    (118,93.57) -- (151.17,59.98) ;
\draw [color={rgb, 255:red, 0; green, 0; blue, 0 }  ,draw opacity=1 ][line width=1.5]    (159.98,51.59) -- (193.15,18) ;
\draw [color={rgb, 255:red, 0; green, 0; blue, 0 }  ,draw opacity=1 ]   (125,12) -- (155.5,44) ;
\draw [color={rgb, 255:red, 0; green, 0; blue, 0 }  ,draw opacity=1 ]   (114,25) -- (143.5,55) ;
\draw [color={rgb, 255:red, 0; green, 0; blue, 0 }  ,draw opacity=1 ]   (168,57) -- (199.5,89) ;
\draw [color={rgb, 255:red, 0; green, 0; blue, 0 }  ,draw opacity=1 ]   (155.5,44) -- (185.5,14) ;
\draw [color={rgb, 255:red, 0; green, 0; blue, 0 }  ,draw opacity=1 ]   (143.5,55) -- (111.5,87) ;
\draw [color={rgb, 255:red, 0; green, 0; blue, 0 }  ,draw opacity=1 ]   (168,57) -- (198.5,27) ;
\draw [color={rgb, 255:red, 0; green, 0; blue, 0 }  ,draw opacity=1 ]   (156,69) -- (184.5,98) ;
\draw [color={rgb, 255:red, 0; green, 0; blue, 0 }  ,draw opacity=1 ]   (156,69) -- (127.5,98) ;
\draw    (69,57) -- (116.5,57) ;
\draw [shift={(118.5,57)}, rotate = 180] [color={rgb, 255:red, 0; green, 0; blue, 0 }  ][line width=0.75]    (10.93,-3.29) .. controls (6.95,-1.4) and (3.31,-0.3) .. (0,0) .. controls (3.31,0.3) and (6.95,1.4) .. (10.93,3.29)   ;
\draw [color={rgb, 255:red, 0; green, 0; blue, 0 }  ,draw opacity=1 ][line width=1.5]    (347.51,17) -- (422.66,92.57) ;
\draw [color={rgb, 255:red, 0; green, 0; blue, 0 }  ,draw opacity=1 ][line width=1.5]    (378.08,61.79) -- (347.93,92.57) ;
\draw    (354.51,10) -- (429.66,85.57) ;
\draw    (340.51,24) -- (415.66,99.57) ;
\draw    (386.08,68.79) -- (354.93,99.57) ;
\draw    (372.08,54.79) -- (340.93,85.57) ;
\draw    (430.5,24) -- (399.34,54.79) ;
\draw    (416.24,10) -- (385.08,40.79) ;
\draw [color={rgb, 255:red, 0; green, 0; blue, 0 }  ,draw opacity=1 ][line width=1.5]    (422.24,17) -- (392.08,47.79) ;
\draw    (305,57) -- (352.5,57) ;
\draw [shift={(354.5,57)}, rotate = 180] [color={rgb, 255:red, 0; green, 0; blue, 0 }  ][line width=0.75]    (10.93,-3.29) .. controls (6.95,-1.4) and (3.31,-0.3) .. (0,0) .. controls (3.31,0.3) and (6.95,1.4) .. (10.93,3.29)   ;

\end{tikzpicture}
    \caption{Each crossing of an abstract link diagram}
    \label{fig:absCrossing}
\end{figure}
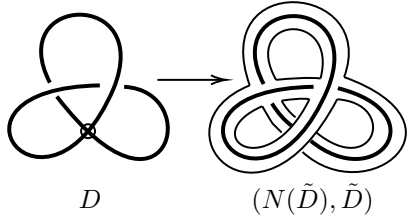
\begin{figure}[htbp]
    \centering
    \tikzset{every picture/.style={line width=0.75pt}} 

\begin{tikzpicture}[x=0.75pt,y=0.75pt,yscale=-0.7,xscale=0.7]

\draw [color={rgb, 255:red, 0; green, 0; blue, 0 }  ,draw opacity=1 ][line width=1.5]    (36.5,67) .. controls (3.5,3) and (111.5,-1) .. (82.5,71) .. controls (34.5,183) and (-54.5,71) .. (72.5,69) ;
\draw [color={rgb, 255:red, 0; green, 0; blue, 0 }  ,draw opacity=1 ][line width=1.5]    (90,70) .. controls (148.5,73) and (121.5,181) .. (41.5,77) ;
\draw  [color={rgb, 255:red, 0; green, 0; blue, 0 }  ,draw opacity=1 ] (59,102) .. controls (59,99.24) and (61.24,97) .. (64,97) .. controls (66.76,97) and (69,99.24) .. (69,102) .. controls (69,104.76) and (66.76,107) .. (64,107) .. controls (61.24,107) and (59,104.76) .. (59,102) -- cycle ;
\draw [color={rgb, 255:red, 0; green, 0; blue, 0 }  ,draw opacity=1 ][line width=1.5]    (190.5,68) .. controls (157.5,4) and (265.5,0) .. (236.5,72) .. controls (188.5,184) and (99.5,72) .. (226.5,70) ;
\draw [color={rgb, 255:red, 0; green, 0; blue, 0 }  ,draw opacity=1 ][line width=1.5]    (244,71) .. controls (315.5,89) and (272.5,148) .. (224.5,109) ;
\draw    (199.5,65) .. controls (173.5,19) and (244.5,14) .. (231.5,64) ;
\draw    (199.5,65) .. controls (209.5,64) and (223.5,63) .. (231.5,64) ;
\draw    (180.5,69) .. controls (150.5,-15) and (273.5,-5) .. (248.5,63) ;
\draw    (180.5,69) .. controls (108.5,98) and (185.5,189) .. (243.5,79) ;
\draw    (223.5,79) .. controls (192.5,142) and (141.5,105) .. (189.5,84) ;
\draw [color={rgb, 255:red, 0; green, 0; blue, 0 }  ,draw opacity=1 ][line width=1.5]    (210.5,98) .. controls (205.5,94) and (201.5,90) .. (196.5,82) ;
\draw    (203.5,103) .. controls (193.5,93) and (198.5,99) .. (189.5,84) ;
\draw    (213.5,92) .. controls (209.5,89) and (205.5,84) .. (203.5,80) ;
\draw    (223.5,79) .. controls (218.5,78) and (210.5,79) .. (203.5,80) ;
\draw    (248.5,63) .. controls (342.5,97) and (265.5,160) .. (218.5,117) ;
\draw    (243.5,79) .. controls (287.5,82) and (270.5,137) .. (228.5,103) ;
\draw    (114,65) -- (161.5,65) ;
\draw [shift={(163.5,65)}, rotate = 180] [color={rgb, 255:red, 0; green, 0; blue, 0 }  ][line width=0.75]    (10.93,-3.29) .. controls (6.95,-1.4) and (3.31,-0.3) .. (0,0) .. controls (3.31,0.3) and (6.95,1.4) .. (10.93,3.29)   ;

\draw (65.61,151) node    {$D$};
\draw (225.76,150.69) node    {$(N(\tilde{D}),  \tilde{D} )$};

\end{tikzpicture}
    \caption{Abstract link diagram}
    \label{fig:abstract}
\end{figure}

Next, we obtain an orientable closed surface $\mathcal{S}$ by attaching disks to each component of $\partial N(\tilde{D})$.
The pair $(\mathcal{S}, \tilde{D})$ is called the {\it canonical surface realization} of $D$ (Figure \ref{fig:canonical}), 
and $\mathcal{S}$ is called the {\it supporting surface} of $D$.
Note that the canonical surface realization of $D$ is unique for a virtual link diagram and $\mathcal{S}$ may also be disconnected even though $D$ is connected.
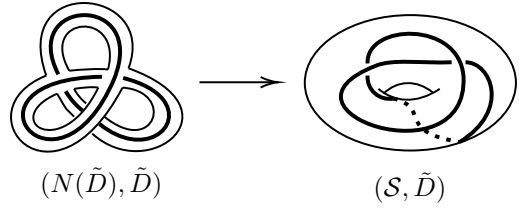
\begin{figure}[htbp]
    \centering
    \tikzset{every picture/.style={line width=0.75pt}} 

\begin{tikzpicture}[x=0.75pt,y=0.75pt,yscale=-0.8,xscale=0.8]

\draw [color={rgb, 255:red, 0; green, 0; blue, 0 }  ,draw opacity=1 ][line width=1.5]    (169.04,84.01) .. controls (143.86,35.18) and (226.25,32.13) .. (204.13,87.06) .. controls (167.51,172.5) and (99.61,87.06) .. (196.5,85.53) ;
\draw [color={rgb, 255:red, 0; green, 0; blue, 0 }  ,draw opacity=1 ][line width=1.5]    (209.85,86.3) .. controls (222.52,89.49) and (230.48,94.36) .. (234.77,99.6) .. controls (248.95,116.91) and (223.09,138.13) .. (194.97,115.29) ;
\draw    (175.9,81.72) .. controls (156.07,46.63) and (210.23,42.81) .. (200.31,80.96) ;
\draw    (175.9,81.72) .. controls (183.53,80.96) and (194.21,80.19) .. (200.31,80.96) ;
\draw    (161.41,84.77) .. controls (138.52,20.69) and (232.35,28.32) .. (213.28,80.19) ;
\draw    (161.41,84.77) .. controls (106.48,106.89) and (165.22,176.32) .. (209.47,92.4) ;
\draw    (194.21,92.4) .. controls (170.56,140.46) and (131.65,112.24) .. (168.27,96.21) ;
\draw [color={rgb, 255:red, 0; green, 0; blue, 0 }  ,draw opacity=1 ][line width=1.5]    (184.29,106.89) .. controls (180.48,103.84) and (177.43,100.79) .. (173.61,94.69) ;
\draw    (178.95,110.71) .. controls (171.32,103.08) and (175.14,107.66) .. (168.27,96.21) ;
\draw    (186.58,102.32) .. controls (183.53,100.03) and (180.48,96.21) .. (178.95,93.16) ;
\draw    (194.21,92.4) .. controls (190.4,91.64) and (184.29,92.4) .. (178.95,93.16) ;
\draw    (213.28,80.19) .. controls (284.99,106.13) and (226.25,154.19) .. (190.4,121.39) ;
\draw    (209.47,92.4) .. controls (243.04,94.69) and (230.07,136.65) .. (198.02,110.71) ;
\draw   (324.38,89.35) .. controls (324.38,65.75) and (354.18,46.63) .. (390.94,46.63) .. controls (427.7,46.63) and (457.5,65.75) .. (457.5,89.35) .. controls (457.5,112.94) and (427.7,132.07) .. (390.94,132.07) .. controls (354.18,132.07) and (324.38,112.94) .. (324.38,89.35) -- cycle ;
\draw    (370.83,93.16) .. controls (380.81,101.55) and (394.52,103.08) .. (409.49,93.16) ;
\draw    (373.95,96.21) .. controls (382.68,87.82) and (394.52,87.06) .. (405.12,95.45) ;
\draw [color={rgb, 255:red, 0; green, 0; blue, 0 }  ,draw opacity=1 ][line width=1.5]    (363.97,77.14) .. controls (362.72,51.97) and (421.96,49.68) .. (424.45,85.53) .. controls (427.57,151.14) and (296.63,106.89) .. (368.96,78.67) .. controls (379.56,74.09) and (397.64,77.14) .. (416.97,76.38) ;
\draw [color={rgb, 255:red, 0; green, 0; blue, 0 }  ,draw opacity=1 ][line width=1.5]    (425.7,76.38) .. controls (441.91,76.38) and (454.38,106.89) .. (421.96,126.73) ;
\draw [color={rgb, 255:red, 0; green, 0; blue, 0 }  ,draw opacity=1 ][line width=1.5]  [dash pattern={on 1.69pt off 2.76pt}]  (380.18,99.27) .. controls (404.5,105.37) and (380.81,119.86) .. (421.96,126.73) ;
\draw [color={rgb, 255:red, 0; green, 0; blue, 0 }  ,draw opacity=1 ][line width=1.5]    (363.97,84.77) .. controls (363.35,93.93) and (371.45,100.79) .. (380.18,99.27) ;
\draw    (258,89) -- (305.5,89) ;
\draw [shift={(307.5,89)}, rotate = 180] [color={rgb, 255:red, 0; green, 0; blue, 0 }  ][line width=0.75]    (10.93,-3.29) .. controls (6.95,-1.4) and (3.31,-0.3) .. (0,0) .. controls (3.31,0.3) and (6.95,1.4) .. (10.93,3.29)   ;

\draw (196.07,152.62) node    {$(N(\tilde{D}),\tilde{D})$};
\draw (390.71,155.38) node    {$(\mathcal{S}, \tilde{D})$};

\end{tikzpicture}
    \caption{Canonical surface realization}
    \label{fig:canonical}
\end{figure}

We can obtain a link $\hat{D}$ in $\mathcal{S} \times I$ from $(\mathcal{S}, \tilde{D})$, where $(\mathcal{S}, \tilde{D})$ is the canonical surface realization of a virtual link diagram $D$.
The pair $(\mathcal{S} \times I, \hat{D})$ is called a {\it canonical space realization}.
Additionally, we can obtain the exterior by removing an open regular neighborhood of $\hat{D}$ from $\mathcal{S} \times I$.
This exterior is called the {\it canonical exterior} of $D$.

The number of components of $\mathcal{S}$ is called the {\it splitting number} of $D$, and we denote this by $s(D)$.
Similarly, the sum of genera of components of $\mathcal{S}$ is called the {\it supporting genus} of $D$, and we denote this by $sg(D)$.
These are also defined for virtual link $L$.
\begin{itemize}
    \item $s(L) = \max \{ s(D) | D \mbox{ is a diagram of } L\}$
    \item $sg(L) = \min \{ sg(D) | D \mbox{ is a diagram of } L\}$
\end{itemize}
$s(L)$ is called the splitting number of $L$, and $sg(L)$ is called the supporting genus of $L$.
Lemma \ref{lem:classical} follows from the definition of the supporting genus of a virtual link.
\begin{lem}\label{lem:classical}
    A virtual link $L$ is classical if and only if $sg(L) = 0$.
\end{lem}

For any virtual link $L$,
there is a diagram which has the splitting number of $L$ and the supporting genus of $L$, simultaneously.
\begin{thm}[Kuperberg \cite{K}]
    For any virtual link $L$, there is a diagram $D$ that satisfies $s(D) = s(L)$ and $sg(D) = sg(L)$.
\end{thm}
Such a diagram $D$ is called a {\it minimal diagram} of $L$.
Suppose that $D$ is a minimal diagram of a virtual link $L$.
A {\it minimal surface realization} of $L$ is the canonical surface realization of $D$.
Similarly, we also define the {\it minimal space realization} and the {\it minimal exterior}.
It is known that the minimal space realization of $L$ and the minimal exterior of $L$ is unique.
\begin{thm}[Kuperberg \cite{K}]\label{thm:uniqueness}
    For any virtual link $L$, the minimal space realization of $L$ and the minimal exterior of $L$ are uniquely determined.
\end{thm}

Furthermore, from the proof process of Theorem \ref{thm:uniqueness}, 
we see a construction method of the minimal exterior of $L$ from the canonical exterior of a virtual link diagram $D$.
Let $\hat{D}$ be a link in a thickened closed orienable surface $\mathcal{S} \times I$ and $M = \text{cl}(\mathcal{S} \times I - N(\hat{D}))$ be the exterior of $\hat{D}$.
Suppose that $A$ is a properly embedded annulus in $M$.
$A$ is said to be {\it vertical} if $\partial A = a_0 \cup a_1$ and $A \cap \mathcal{S} \times \{i\} = a_i (i = 0,1)$, and $A$ is said to be {\it essential} if $A$ is incompressible, $\partial$-incomporessible, and not $\partial$-parallel.
We define the following two operations for $M$.
\begin{ope}\label{ope:splitting_mani}
    Suppose that $\hat{D} = \hat{D}' \cup \hat{D}''$ (possibly $\hat{D}'' = \emptyset$) and $F$ is an embedded $2$-sphere or a properly embedded disk in $M$ such that there is a submanifold $M' = \text{cl}(\mathbb{B}^3 - N(\hat{D'})) \subset M$ and $F \subset \partial M' - \partial N(\hat{D})$.
    First, cut $M$ open along $F$.
    Then $M'$ and $M''$ denote $\text{cl}(\mathbb{B}^3 - N(\hat{D}'))$ and the remaining component, respectively.
    Next, add a component $\text{cl}(\mathbb{S}^2 \times I - B)$ to $M' \cup M''$, where $B$ is a $3$-ball in $\mathbb{S}^2 \times I$, and glue the copy of $F \subset \partial M'$ and $\partial B \subset \partial (\mathbb{S}^2 \times I - B)$.
    Next, shrink the copy of $F$ in $M''$ to a point.
    Finally, remove the components which contain no components of $\partial N(\hat{D})$.
\end{ope}

\begin{ope}[Destabilization]\label{ope:destabilization_mani}
    Suppose that $A$ is a vertical essential annulus in $M$. 
    Cut $M$ open along $A$, and fill each copy of $A$ with a $2$-handle $D^2 \times I$ as shown in Figure \ref{fig:des}.
    Then, remove the components which contain no components of $\partial N(\hat{D})$.
\end{ope}
We call Operation \ref{ope:destabilization_mani} {\it destabilization}.
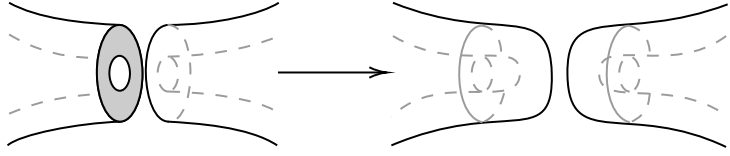
\begin{figure}[htbp]
    \centering
    \tikzset{every picture/.style={line width=0.75pt}} 

\begin{tikzpicture}[x=0.75pt,y=0.75pt,yscale=-0.8,xscale=0.8]

\draw [color={rgb, 255:red, 155; green, 155; blue, 155 }  ,draw opacity=1 ] [dash pattern={on 4.5pt off 4.5pt}]  (40.5,50) .. controls (65.5,65) and (88.5,66) .. (111.5,64) ;
\draw [color={rgb, 255:red, 155; green, 155; blue, 155 }  ,draw opacity=1 ] [dash pattern={on 4.5pt off 4.5pt}]  (40.5,100) .. controls (64.5,90) and (77.5,89) .. (112.5,86) ;
\draw    (40.5,30) .. controls (58.5,40) and (86.5,43) .. (110.5,44) ;
\draw    (140.5,44) .. controls (171.5,44) and (199.5,37) .. (210.5,30) ;
\draw    (39.5,120) .. controls (47.5,113) and (83.5,106) .. (110.5,105) ;
\draw    (140.5,105) .. controls (157.5,106) and (191.5,108) .. (209.5,120) ;
\draw   (110.5,105) .. controls (102.57,105.08) and (96.01,91.49) .. (95.84,74.64) .. controls (95.68,57.8) and (101.97,44.08) .. (109.9,44.01) .. controls (117.83,43.93) and (124.39,57.52) .. (124.56,74.36) .. controls (124.72,91.21) and (118.43,104.92) .. (110.5,105) -- cycle ;
\draw    (140.5,44) .. controls (120.5,46) and (123.5,105) .. (141.1,104.99) ;
\draw [color={rgb, 255:red, 155; green, 155; blue, 155 }  ,draw opacity=1 ] [dash pattern={on 4.5pt off 4.5pt}]  (140.5,44) .. controls (160.5,44) and (158.5,106) .. (141.1,104.99) ;
\draw  [color={rgb, 255:red, 155; green, 155; blue, 155 }  ,draw opacity=1 ][dash pattern={on 4.5pt off 4.5pt}] (140.72,86.01) .. controls (137.14,86.04) and (134.19,81.14) .. (134.13,75.07) .. controls (134.07,68.99) and (136.92,64.03) .. (140.5,64) .. controls (144.08,63.97) and (147.03,68.86) .. (147.09,74.94) .. controls (147.15,81.02) and (144.29,85.97) .. (140.72,86.01) -- cycle ;
\draw [color={rgb, 255:red, 155; green, 155; blue, 155 }  ,draw opacity=1 ] [dash pattern={on 4.5pt off 4.5pt}]  (140.5,64) .. controls (177.5,63) and (198.5,56) .. (210.5,51) ;
\draw [color={rgb, 255:red, 155; green, 155; blue, 155 }  ,draw opacity=1 ] [dash pattern={on 4.5pt off 4.5pt}]  (140.72,86.01) .. controls (179.5,87) and (204.5,94) .. (210.5,100) ;
\draw    (282.5,30) .. controls (349.5,57) and (381.5,27) .. (382.61,75) .. controls (383.72,123.01) and (352.5,90) .. (281.5,120) ;
\draw [color={rgb, 255:red, 155; green, 155; blue, 155 }  ,draw opacity=1 ] [dash pattern={on 4.5pt off 4.5pt}]  (282.5,51) .. controls (324.5,76) and (363.5,51) .. (362.5,75) .. controls (361.5,99) and (311.5,72) .. (281.5,102) ;
\draw  [color={rgb, 255:red, 155; green, 155; blue, 155 }  ,draw opacity=1 ][dash pattern={on 4.5pt off 4.5pt}] (338.72,86.01) .. controls (335.14,86.04) and (332.19,81.14) .. (332.13,75.07) .. controls (332.07,68.99) and (334.92,64.03) .. (338.5,64) .. controls (342.08,63.97) and (345.03,68.86) .. (345.09,74.94) .. controls (345.15,81.02) and (342.29,85.97) .. (338.72,86.01) -- cycle ;
\draw [color={rgb, 255:red, 155; green, 155; blue, 155 }  ,draw opacity=1 ]   (338.31,44.51) .. controls (321.5,45) and (317.5,100) .. (338.91,105.5) ;
\draw [color={rgb, 255:red, 155; green, 155; blue, 155 }  ,draw opacity=1 ] [dash pattern={on 4.5pt off 4.5pt}]  (352.5,87) .. controls (350.5,96) and (345.5,103) .. (338.91,105.5) ;
\draw [color={rgb, 255:red, 155; green, 155; blue, 155 }  ,draw opacity=1 ] [dash pattern={on 4.5pt off 4.5pt}]  (338.31,44.51) .. controls (348.5,45) and (351.5,58) .. (350.5,63) ;
\draw    (493.5,31) .. controls (426.5,56) and (391.39,28) .. (392.5,76) .. controls (393.61,124) and (431.5,92) .. (492.5,121) ;
\draw [color={rgb, 255:red, 155; green, 155; blue, 155 }  ,draw opacity=1 ] [dash pattern={on 4.5pt off 4.5pt}]  (492.5,51) .. controls (457.5,77) and (413.5,50) .. (412.5,74) .. controls (411.5,98) and (456.5,76) .. (492.5,101) ;
\draw  [color={rgb, 255:red, 155; green, 155; blue, 155 }  ,draw opacity=1 ][dash pattern={on 4.5pt off 4.5pt}] (431.72,86.01) .. controls (428.14,86.04) and (425.19,81.14) .. (425.13,75.07) .. controls (425.07,68.99) and (427.92,64.03) .. (431.5,64) .. controls (435.08,63.97) and (438.03,68.86) .. (438.09,74.94) .. controls (438.15,81.02) and (435.29,85.97) .. (431.72,86.01) -- cycle ;
\draw [color={rgb, 255:red, 155; green, 155; blue, 155 }  ,draw opacity=1 ]   (431.31,44.51) .. controls (414.5,45) and (410.5,100) .. (431.91,105.5) ;
\draw [color={rgb, 255:red, 155; green, 155; blue, 155 }  ,draw opacity=1 ] [dash pattern={on 4.5pt off 4.5pt}]  (444.5,87) .. controls (443.5,92) and (444.5,100) .. (431.91,105.5) ;
\draw [color={rgb, 255:red, 155; green, 155; blue, 155 }  ,draw opacity=1 ] [dash pattern={on 4.5pt off 4.5pt}]  (431.31,44.51) .. controls (438.5,46) and (444.5,52) .. (444.5,63) ;
\draw    (210,74) -- (276.5,74) ;
\draw [shift={(278.5,74)}, rotate = 180] [color={rgb, 255:red, 0; green, 0; blue, 0 }  ][line width=0.75]    (10.93,-3.29) .. controls (6.95,-1.4) and (3.31,-0.3) .. (0,0) .. controls (3.31,0.3) and (6.95,1.4) .. (10.93,3.29)   ;
\draw  [fill={rgb, 255:red, 204; green, 204; blue, 204 }  ,fill opacity=1 ] (95.84,74.64) .. controls (95.68,57.8) and (101.97,44.08) .. (109.9,44.01) .. controls (117.83,43.93) and (124.39,57.52) .. (124.56,74.36) .. controls (124.72,91.21) and (118.43,104.92) .. (110.5,105) .. controls (102.57,105.08) and (96.01,91.49) .. (95.84,74.64) -- cycle (110.31,85.51) .. controls (113.89,85.47) and (116.74,80.52) .. (116.68,74.44) .. controls (116.62,68.36) and (113.67,63.46) .. (110.09,63.5) .. controls (106.52,63.53) and (103.66,68.49) .. (103.72,74.57) .. controls (103.78,80.64) and (106.73,85.54) .. (110.31,85.51) -- cycle ;

\end{tikzpicture}
    \caption{Destabilization}
    \label{fig:des}
\end{figure}

\begin{thm}[Kuperberg \cite{K}] \label{thm:K}
    Let $D$ be a virtual link diagram of a virtual link $L$, and let $M$ be the canonical exterior of $D$.
    The minimal exterior of $L$ can be obtained by repeatedly performing Operation \ref{ope:splitting_mani} and destabilization to $M$.
    Furthermore, the minimal exterior obtained by performing Operation \ref{ope:splitting_mani} and destabilization in any order is unique.
\end{thm}

We denote the sum of genera of connected components of a surface $F$ by $g(F)$.
By Theorem \ref{thm:K}, we have the following algorithm for classical link recognition.
\begin{alg} \label{alg:1}
    Let $L$ be a virtual link and $D$ be a diagram of $L$.
\begin{enumerate}
    \item Construct the canonical exterior $M$ of $D$.
    \item Construct the minimal exterior $M_\text{min}$ of $L$ by repeatedly performing Operation \ref{ope:splitting_mani} and destabilization to $M$.
    \item Output ``yes'' if $g(\mathcal{S} \times \{0\}) = g(\mathcal{S} \times \{1\}) = 0$, otherwise output ``no'', where $\mathcal{S} \times \{i\} (i=0,1)$ are the two copies of the supporting surface of a minimal diagram of $L$ in the boundary of $M_\text{min}$.
\end{enumerate}
\end{alg}
In order to prove Theorem \ref{thm:main}, we use the operation defined bellow, {\it splitting} ,instead of Operation \ref{ope:splitting_mani}.
Let $\hat{D}$ be a link in a thickened closed orientable surface $\mathcal{S} \times I$ and $M = \text{cl}(\mathcal{S} \times I - N(\hat{D}))$ be the exterior of $\hat{D}$.
An properly embedded $2$-sphere $F$ in $M$ is said to be \textit{inessential} if there is a $3$-ball $B$ in $M$ such that $F = \partial B$, and a properly embedded disk $F'$ is said to be \textit{inessential} if there is a disk $S$ in $\partial M$ such that $\partial F' = \partial S$ and there is a $3$-ball $B$ in $M$ whose boundary is $F' \cup S$.
A properly embedded $2$-sphere and a properly embedded disk in $M$ are said to be \textit{essential} if they are not inessential.

\begin{defi}[splitting] \label{def:splitting}
    Suppose that $F$ is an essential $2$-sphere in $M$ or an essential disk in $M$ whose boundary is in $\mathcal{S} \times \{0\} \cup \mathcal{S} \times \{1\}$.
    Then, {\it splitting} is the operation cutting $M$ open along $F$, shrinking each copy of $F$ to a point, and removing the components which contain no components of $\partial N(D)$ as shown in Figure \ref{fig:splitting}.
\end{defi}

\begin{figure}[htbp]
    \centering
    \input{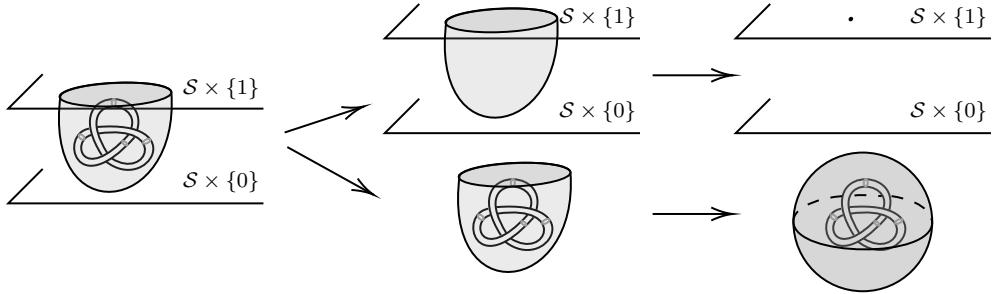}
    \caption{Splitting ($F$ is a disk)}
    \label{fig:splitting}
\end{figure}

There are the following two differences between Operation \ref{ope:splitting_mani} and splitting.

The first difference is that the surface used for cutting is changed to an essential 2-sphere or an essential disk whose boundary is in $\mathcal{S} \times \{0\} \cup \mathcal{S} \times \{1\}$.
This difference affects the operation if and only if there is a $2$-sphere component $\mathcal{S}'$ in the supporting surface $\mathcal{S}$.
Let $M'$ denotes the component of $M$ which contains $\mathcal{S}' \times \{0\}$ and $\mathcal{S}' \times \{1\}$ and $\hat{D}'$ denote the link in $\mathcal{S}' \times I$. 
Since a $2$-sphere $F \subset M'$ used for splitting is an essential $2$-sphere,
$F$ may separate $\mathcal{S}' \times \{0\}$ and $\mathcal{S}' \times \{1\}$.
In this case, we obtain two components $M'_1 = \mathbb{B}^3 - N(\hat{D}'_1)$ and $M'_2 = \mathbb{B}^3 - N(\hat{D}'_2)$ by shrinking the copies of $F$, where $\hat{D}' = \hat{D}'_1 \cup \hat{D}_2$.
$M'_1$ and $M'_2$ are not 3-manifolds obtained by removing an open regular neighborhood of a link from a thickened surface.
However, the boundaries obtained from $\mathcal{S}' \times \{0\}$ and $\mathcal{S}' \times \{1\}$ keep 2-spheres by this operation.
Therefore, even though this operation is performed, classical link recognition can be solved by determining whether $g(\mathcal{S}_k \times \{0\}) = g(\mathcal{S}_k \times \{1\})  = 0$, where $\mathcal{S}_k \times \{i\}$ is the boundary obtained from $\mathcal{S}_k \times \{i\}$.

The second difference is that both of copies of $F$ are shrunk to points.
By this change, we obtain $\mathbb{S}^3-N(\hat{D}')$ if $F$ is a 2-sphere and 
we obtain $\mathbb{B}^3-N(\hat{D}')$ if $F$ is a disk.
From the same discussion of the first change, we can solve classical link recognition.

Next, we change Algorithm \ref{alg:1} as follows.
\begin{alg} \label{alg:2}
    Let $L$ be a virtual link and let $D$ be a virtual link diagram of $L$.
\begin{enumerate}
    \item Construct the canonical exterior $M_0$ of $D$.
    $\mathcal{S}_0 \times \{i\} (i=0,1)$ denote the two copies of the supporting surface of $D$ in the boundary of $M_0$.
    \item Output ``yes'' if $g(\mathcal{S}_k \times \{0\}) = g(\mathcal{S}_k \times \{1\}) = 0$, otherwise do the following.
    \item Run one of the following steps.
    \begin{enumerate}
        \item Do splitting if there is an essential $2$-sphere in $M_k$.
        \item Do destabilization if there are no essential $2$-spheres and there is an essential vertical annulus in $M_k$.
        \item Output ``no'' if there are no essential $2$-spheres and essential vertical annuli.
    \end{enumerate}
    \item We define $M_{k+1}$ as the 3-manifold obtained from $M_k$ and $\mathcal{S}_{k+1} \times \{i\}$ as the boundary of $M_{k+1}$ obtained from $\mathcal{S}_k \times \{i\}$.
    \item Return step 2.
\end{enumerate}
\end{alg}

\subsection{Algorithm for classical knot recognition}
{\it Classical knot recognition} is the problem restricting inputs of classical link recognition to virtual knot diagrams.
Since a virtual knot is a one component virtual link, classical knot recognition can be solved by Algorithm \ref{alg:2}.
However, it is difficult to bound the computational complexity of Algorithm \ref{alg:2}.
In this subsection, we give an algorithm for classical knot recognition which is easy to bound its computational complexity by adding one step to Algorithm \ref{alg:2}.

Suppose that $D$ is a diagram of a virtual knot $K$ and $M = \text{cl}(\mathcal{S} \times I - N(\hat{D}))$ is a $3$-manifold obtained from the canonical exterior of $D$ by zero or more destabilizations.
Let $A$ be a properly embedded annulus in $M$.
We say that $A$ is a \textit{classicalization annulus} if $\partial A \subset \mathcal{S} \times \{k\}$ ($k = 0$ or $1$) and $A$ separates $\partial N(\hat{D})$ and $\mathcal{S} \times \{1-k\}$.
The existence of a classicalization annulus implies that $K$ is a classical knot.
\begin{lem}\label{lem:classicalization_annulus}
    Let $K$ be a virtual knot and $D$ be a virtual knot diagram of $K$.
    Let $M$ be a $3$-manifold obtained from the canonical exterior of $D$ by zero or more destabilizations.
    Suppose that $M \simeq \text{cl}(\mathcal{S} \times I -N(\hat{D}))$, where $\mathcal{S}$ is a closed orientable surface.
    If there is a classicalization annulus $A$, then $K$ is classical.
\end{lem}
\begin{proof}
    Suppose that $A$ is a classicalization annulus in $M$.
    Without loss of generality, we may assume that $\partial A \subset \mathcal{S} \times \{0\}$.
    Since for any properly embedded annulus in $\mathcal{S} \times I$ whose boundary components are in $\mathcal{S} \times \{0\}$ is $\partial$-parallel, there is an annulus $A'$ in $\mathcal{S} \times \{0\} \subset \partial M$ such that $\partial A = \partial A'$ and $A \cup A' \subset \partial M'$, where $M'$ is a submanifold of $M$ which is $\text{cl}(D^2 \times \mathbb{S}^1 - N(\hat{D}))$.
    Since $\text{cl}(M - M') \simeq \mathcal{S} \times I$,
    there is an embedded annulus $B$ in $\text{cl}(M - M')$ such that $\partial B = \beta_1 \cup \beta_2$, $\beta_1 = A \cap B$ is essential in $A$, and $\beta_2 = B \cap \mathcal{S} \times \{1\}$.
    Let $B'$ denote the annulus $N(B) \cap \mathcal{S} \times \{1\}$.
    
    We condsider the submanifold $M'' = N(B) \cup M'$ of $M$.
    Then $M''$ is an embedded $3$-manifold $\text{cl}(\mathbb{S}^1 \times I \times I - N(\hat{D}))$ so that $\mathbb{S}^1 \times I \times \{0\} = A'$ and $\mathbb{S}^1 \times I \times \{1\} = B'$.
    Now $\mathbb{S}^1 \times \partial I \times I$ is the disjoint union of properly embedded vertical annuli in $M$, denoted by $V_1$ and $V_2$.
    Let $N$ be the $3$-manifold obtained from $M$ by destabilizations using $V_1$ and $V_2$ and $N'$ be the component of $N$ containing $\partial N(\hat{D})$.
    $N'$ is obtained from $M''$ by gluing two $2$-handles $D^2 \times I$ to $V_1$ and $V_2$, and hence $N' \simeq \text{cl}(\mathbb{S}^2 \times I - N(\hat{D}))$.
    Therefore, $K$ is classical since we have $g(K) = 0$.
\end{proof}

From Lemma \ref{lem:classicalization_annulus}, a virtual knot $K$ represected by a virtual knot diagram $D$ is classical if and only if 
(i) the genus of the supporting surface in the canonical exterior of $D$ is reduced to zero by repeating splitting and destabilization or
(ii) there is a classicalization annulus in the canonical exterior of $K$.
Thus, we have Algorithm \ref{alg:CKR}.
\begin{alg} \label{alg:CKR}
    Let $K$ be a virtual knot and let $D$ be a virtual knot diagram of $K$.
\begin{enumerate}
    \item Construct the canonical exterior $M_0$ of $D$.
    $\mathcal{S}_0 \times \{i\} (i=0,1)$ denote the two copies of the supporting surface of $D$ in the boundary of $M_0$.
    \item Output ``yes'' if $g(\mathcal{S}_k \times \{0\}) = g(\mathcal{S}_k \times \{1\}) = 0$, otherwise do the following.
    \item Run one of the following steps.
    \begin{enumerate}
        \item Do splitting if there is an essential $2$-sphere.
        \item Do destabilization if there are no essential $2$-spheres and there is an essential vertical annulus in $M_k$.
        \item Output ``yes'' if there are no essential $2$-spheres and essential vertical annuli and there is a classicalization annulus in $M_k$.
        \item Output ``no'' if there are no essential $2$-spheres, essential vertical annuli, and classicalization annuli in $M_k$.
    \end{enumerate}
    \item We define $M_{k+1}$ as the 3-manifold obtained from $M_k$ and $\mathcal{S}_{k+1} \times \{i\}$ as the boundary of $M_{k+1}$ obtained from $\mathcal{S}_k \times \{i\}$.
    \item Return step 2.
\end{enumerate}
\end{alg}
Theorem \ref{thm:K} and Lemma \ref{lem:classicalization_annulus} imply that the output of Algorithm \ref{alg:CKR} is correct.
We show that classical knot recognition is in NP by constructing a non-deterministic Turing machine based on Algorithm \ref{alg:CKR} in Section 5.
Furthermore, it is shown that Algorithm \ref{alg:CKR} runs in exponential time in Section 6.

From the next section, we prepare notations and lemmas for analyzing the time complexity of Algorithm \ref{alg:CKR}.

\section{Triangulations of the canonical exteriors of virtual links}
As mentioned above, a virtual link is represented as a link in a thickened orientable closed surface.
In this section, we give an algorithm to construct a triangulation of the canonical exterior of a virtual link diagram when the diagram is given.

Let $\Delta_i (i=1, \ldots, n)$ be tetrahedra and $T_{i_1}, \ldots, T_{i_4}$ be faces of $\Delta_i$.
$\phi_{i_j, k_l}$ denote gluing maps from $T_{i_j}$ to $T_{k_l}$.
We denote the pair of the set of the tetrahedra and the set of the gluing maps $(\{\Delta_i\}, \{\phi_{i_j, k_l}\})$ by $\mathcal{T}$.
$\mathcal{T}$ is called a {\it generalized triangulation} of a 3-manifold $M$ if $M$ is homeomorphic to the space obtained as the set $\{\Delta_i\}$ of the quotient by the gluing maps $\phi_{i_j, k_l}$.
We also define a generalized triangulation for $2$-manifold.
In this paper, we simply call a generalized triangulation a triangulation, and 
abusing the notation, we denote the quotient space by $\mathcal{T}$.

In this section, $D$ denotes a virtual link diagram, and $c$ denotes the number of real crossings of $D$.
We denote by $|\mathcal{T}|$ the number of $n$-simplices of a triangulation $\mathcal{T}$ of an $n$-manifold.

Let $(\mathcal{S}, \tilde{D})$ be the canonical surface realization of $D$.
A triangulation of $\mathcal{S}$ that includes $\tilde{D}$ in its 1-skeleton is called a  {\it good triangulation} of $(\mathcal{S}, \tilde{D})$.
Similarly, a triangulation of $\mathcal{S} \times I$ that includes $\hat{D}$ in its 1-skeleton is called a good triangulation of $(\mathcal{S} \times I, \hat{D})$,
where $(\mathcal{S} \times I, \hat{D})$ is the canonical space realization of $D$.

\begin{lem}\label{lem:goodTri}
    Let $(\mathcal{S}, \tilde{D})$ be the canonical surface realization of $D$.
    We can construct a good triangulation $\mathcal{T}$ of $(\mathcal{S}, \tilde{D})$ in time $\mathcal{O}(c)$.
    Moreover, $|\mathcal{T}| \in \mathcal{O}(c)$.
\end{lem}
\begin{proof}
  Place a triangulated rectangle at each real crossing as shown in Figure \ref{fig:canonicalSurfaceTri}.
  We obtain a good triangulation by gluing rectangles on adjacent real crossings.
\end{proof}

\begin{figure}[htbp]
    \centering
    \tikzset{every picture/.style={line width=0.75pt}} 

\begin{tikzpicture}[x=0.75pt,y=0.75pt,yscale=-1,xscale=1]

\draw  [color={rgb, 255:red, 74; green, 74; blue, 74 }  ,draw opacity=1 ] (228.5,69) -- (199.5,69) -- (199.5,40) -- cycle ;
\draw  [color={rgb, 255:red, 74; green, 74; blue, 74 }  ,draw opacity=1 ] (228.5,69) -- (228.5,40) -- (199.5,40) -- cycle ;
\draw  [color={rgb, 255:red, 74; green, 74; blue, 74 }  ,draw opacity=1 ] (257.5,98) -- (228.5,98) -- (228.5,69) -- cycle ;
\draw  [color={rgb, 255:red, 74; green, 74; blue, 74 }  ,draw opacity=1 ] (257.5,98) -- (257.5,69) -- (228.5,69) -- cycle ;
\draw  [color={rgb, 255:red, 74; green, 74; blue, 74 }  ,draw opacity=1 ] (228.5,69) -- (257.5,40) -- (228.5,40) -- cycle ;
\draw  [color={rgb, 255:red, 74; green, 74; blue, 74 }  ,draw opacity=1 ] (228.5,69) -- (257.5,40) -- (257.5,69) -- cycle ;
\draw  [color={rgb, 255:red, 74; green, 74; blue, 74 }  ,draw opacity=1 ] (199.5,98) -- (228.5,69) -- (228.5,98) -- cycle ;
\draw  [color={rgb, 255:red, 74; green, 74; blue, 74 }  ,draw opacity=1 ] (199.5,98) -- (228.5,69) -- (199.5,69) -- cycle ;
\draw [line width=1.5]    (228.5,40) -- (228.5,62) ;
\draw [line width=1.5]    (257.5,69) -- (199.5,69) ;
\draw [color={rgb, 255:red, 0; green, 0; blue, 0 }  ,draw opacity=1 ]   (54.55,67.25) .. controls (29.37,18.42) and (111.77,15.37) .. (89.64,70.3) .. controls (53.02,155.74) and (-14.87,70.3) .. (82.01,68.77) ;
\draw [color={rgb, 255:red, 0; green, 0; blue, 0 }  ,draw opacity=1 ]   (95.36,69.53) .. controls (139.99,71.82) and (119.4,154.21) .. (58.36,74.87) ;
\draw  [color={rgb, 255:red, 0; green, 0; blue, 0 }  ,draw opacity=1 ] (71.72,93.95) .. controls (71.72,91.84) and (73.42,90.13) .. (75.53,90.13) .. controls (77.64,90.13) and (79.34,91.84) .. (79.34,93.95) .. controls (79.34,96.05) and (77.64,97.76) .. (75.53,97.76) .. controls (73.42,97.76) and (71.72,96.05) .. (71.72,93.95) -- cycle ;
\draw  [color={rgb, 255:red, 74; green, 74; blue, 74 }  ,draw opacity=1 ] (329.5,69) -- (300.5,69) -- (300.5,40) -- cycle ;
\draw  [color={rgb, 255:red, 74; green, 74; blue, 74 }  ,draw opacity=1 ] (329.5,69) -- (329.5,40) -- (300.5,40) -- cycle ;
\draw  [color={rgb, 255:red, 74; green, 74; blue, 74 }  ,draw opacity=1 ] (358.5,98) -- (329.5,98) -- (329.5,69) -- cycle ;
\draw  [color={rgb, 255:red, 74; green, 74; blue, 74 }  ,draw opacity=1 ] (358.5,98) -- (358.5,69) -- (329.5,69) -- cycle ;
\draw  [color={rgb, 255:red, 74; green, 74; blue, 74 }  ,draw opacity=1 ] (329.5,69) -- (358.5,40) -- (329.5,40) -- cycle ;
\draw  [color={rgb, 255:red, 74; green, 74; blue, 74 }  ,draw opacity=1 ] (329.5,69) -- (358.5,40) -- (358.5,69) -- cycle ;
\draw  [color={rgb, 255:red, 74; green, 74; blue, 74 }  ,draw opacity=1 ] (300.5,98) -- (329.5,69) -- (329.5,98) -- cycle ;
\draw  [color={rgb, 255:red, 74; green, 74; blue, 74 }  ,draw opacity=1 ] (300.5,98) -- (329.5,69) -- (300.5,69) -- cycle ;
\draw [line width=1.5]    (329.5,40) -- (329.5,98) ;
\draw [line width=1.5]    (322.5,69) -- (300.5,69) ;
\draw [color={rgb, 255:red, 155; green, 155; blue, 155 }  ,draw opacity=1 ]   (259.59,69) .. controls (272.28,69) and (284.3,69) .. (298.69,69) ;
\draw [shift={(300.5,69)}, rotate = 180] [color={rgb, 255:red, 155; green, 155; blue, 155 }  ,draw opacity=1 ][line width=0.75]    (10.93,-3.29) .. controls (6.95,-1.4) and (3.31,-0.3) .. (0,0) .. controls (3.31,0.3) and (6.95,1.4) .. (10.93,3.29)   ;
\draw [shift={(257.5,69)}, rotate = 0] [color={rgb, 255:red, 155; green, 155; blue, 155 }  ,draw opacity=1 ][line width=0.75]    (10.93,-3.29) .. controls (6.95,-1.4) and (3.31,-0.3) .. (0,0) .. controls (3.31,0.3) and (6.95,1.4) .. (10.93,3.29)   ;
\draw [color={rgb, 255:red, 155; green, 155; blue, 155 }  ,draw opacity=1 ]   (229.91,37.8) .. controls (253.1,3.11) and (302.51,3.71) .. (328.34,38.39) ;
\draw [shift={(329.5,40)}, rotate = 234.9] [color={rgb, 255:red, 155; green, 155; blue, 155 }  ,draw opacity=1 ][line width=0.75]    (10.93,-3.29) .. controls (6.95,-1.4) and (3.31,-0.3) .. (0,0) .. controls (3.31,0.3) and (6.95,1.4) .. (10.93,3.29)   ;
\draw [shift={(228.5,40)}, rotate = 301.43] [color={rgb, 255:red, 155; green, 155; blue, 155 }  ,draw opacity=1 ][line width=0.75]    (10.93,-3.29) .. controls (6.95,-1.4) and (3.31,-0.3) .. (0,0) .. controls (3.31,0.3) and (6.95,1.4) .. (10.93,3.29)   ;
\draw [color={rgb, 255:red, 155; green, 155; blue, 155 }  ,draw opacity=1 ]   (197.35,70.2) .. controls (151.52,97.13) and (185.69,165.97) .. (329.5,98) ;
\draw [shift={(329.5,98)}, rotate = 514.7] [color={rgb, 255:red, 155; green, 155; blue, 155 }  ,draw opacity=1 ][line width=0.75]    (10.93,-3.29) .. controls (6.95,-1.4) and (3.31,-0.3) .. (0,0) .. controls (3.31,0.3) and (6.95,1.4) .. (10.93,3.29)   ;
\draw [shift={(199.5,69)}, rotate = 152.05] [color={rgb, 255:red, 155; green, 155; blue, 155 }  ,draw opacity=1 ][line width=0.75]    (10.93,-3.29) .. controls (6.95,-1.4) and (3.31,-0.3) .. (0,0) .. controls (3.31,0.3) and (6.95,1.4) .. (10.93,3.29)   ;
\draw [color={rgb, 255:red, 155; green, 155; blue, 155 }  ,draw opacity=1 ]   (231.79,99.77) .. controls (339.43,156.79) and (411.24,112.02) .. (359.29,69.64) ;
\draw [shift={(358.5,69)}, rotate = 398.53] [color={rgb, 255:red, 155; green, 155; blue, 155 }  ,draw opacity=1 ][line width=0.75]    (10.93,-3.29) .. controls (6.95,-1.4) and (3.31,-0.3) .. (0,0) .. controls (3.31,0.3) and (6.95,1.4) .. (10.93,3.29)   ;
\draw [shift={(228.5,98)}, rotate = 28.61] [color={rgb, 255:red, 155; green, 155; blue, 155 }  ,draw opacity=1 ][line width=0.75]    (10.93,-3.29) .. controls (6.95,-1.4) and (3.31,-0.3) .. (0,0) .. controls (3.31,0.3) and (6.95,1.4) .. (10.93,3.29)   ;
\draw [line width=1.5]    (228.5,76) -- (228.5,98) ;
\draw [line width=1.5]    (358.5,69) -- (336.5,69) ;
\draw    (126,68) -- (183.5,68) ;
\draw [shift={(185.5,68)}, rotate = 180] [color={rgb, 255:red, 0; green, 0; blue, 0 }  ][line width=0.75]    (10.93,-3.29) .. controls (6.95,-1.4) and (3.31,-0.3) .. (0,0) .. controls (3.31,0.3) and (6.95,1.4) .. (10.93,3.29)   ;

\end{tikzpicture}
    \vspace{-20pt}
    \caption{Construction of a good triangulation of the canonical surface realization}
    \label{fig:canonicalSurfaceTri}
\end{figure}
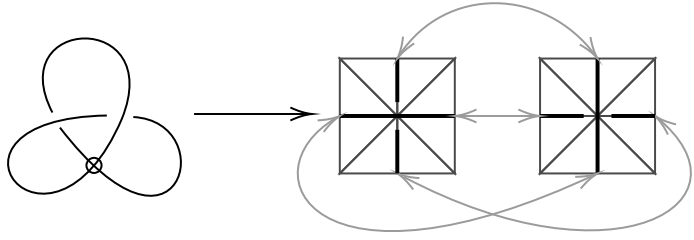

Similarly, We can obtain a good triangulation of $(\mathcal{S} \times I, \hat{D})$
by placing and gluing triangulated cubes (Figure \ref{fig:canonicalSpaceRealizationTri}).
\begin{lem}
    Let $(\mathcal{S} \times I, \hat{D})$ be the canonical surface realization of $D$.
    We can construct a good triangulation $\mathcal{T}$ of $(\mathcal{S} \times I, \hat{D})$ in time $\mathcal{O}(c)$.
    Moreover, $|\mathcal{T}| \in \mathcal{O}(c)$.
\end{lem}

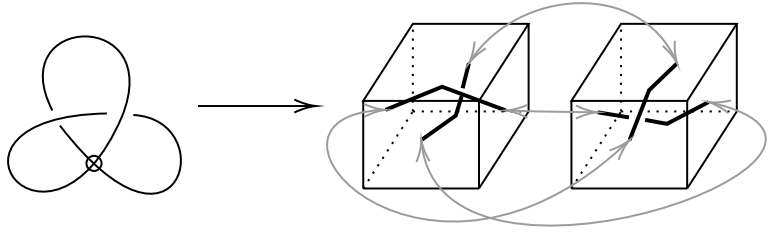
\begin{figure}[htbp]
    \centering
    \tikzset{every picture/.style={line width=0.75pt}} 

\begin{tikzpicture}[x=0.75pt,y=0.75pt,yscale=-1,xscale=1]

\draw    (223.32,110.62) -- (281.7,110.62) ;
\draw  [dash pattern={on 0.84pt off 2.51pt}]  (248.34,71.74) -- (306.72,71.74) ;
\draw    (281.7,110.62) -- (306.72,71.74) ;
\draw  [dash pattern={on 0.84pt off 2.51pt}]  (223.32,110.62) -- (248.34,71.74) ;
\draw    (281.7,110.62) -- (281.7,66.41) ;
\draw    (223.32,110.62) -- (223.32,66.41) ;
\draw    (306.72,71.74) -- (306.72,27.52) ;
\draw  [dash pattern={on 0.84pt off 2.51pt}]  (248.34,71.74) -- (248.34,27.52) ;
\draw [color={rgb, 255:red, 0; green, 0; blue, 0 }  ,draw opacity=1 ][line width=1.5]    (234.58,70.94) -- (263.09,59.28) -- (295.28,71.25) ;
\draw [color={rgb, 255:red, 0; green, 0; blue, 0 }  ,draw opacity=1 ][line width=1.5]    (252.65,86.17) -- (269.98,73.83) -- (273,63.94) ;
\draw [color={rgb, 255:red, 0; green, 0; blue, 0 }  ,draw opacity=1 ][line width=1.5]    (276.62,47.45) -- (273.44,60.01) ;
\draw   (248.34,27.52) -- (306.72,27.52) -- (281.7,66.41) -- (223.32,66.41) -- cycle ;
\draw    (328.32,110.62) -- (386.7,110.62) ;
\draw  [dash pattern={on 0.84pt off 2.51pt}]  (353.34,71.74) -- (411.72,71.74) ;
\draw    (386.7,110.62) -- (411.72,71.74) ;
\draw  [dash pattern={on 0.84pt off 2.51pt}]  (328.32,110.62) -- (353.34,71.74) ;
\draw    (328.32,110.62) -- (328.32,66.41) ;
\draw    (411.72,71.74) -- (411.72,27.52) ;
\draw  [dash pattern={on 0.84pt off 2.51pt}]  (353.34,71.74) -- (353.34,27.52) ;
\draw [color={rgb, 255:red, 0; green, 0; blue, 0 }  ,draw opacity=1 ][line width=1.5]    (357.65,86.17) -- (367.5,61) -- (381.62,47.45) ;
\draw [color={rgb, 255:red, 0; green, 0; blue, 0 }  ,draw opacity=1 ][line width=1.5]    (365.5,76) -- (376.5,78) -- (397.5,67) ;
\draw [color={rgb, 255:red, 0; green, 0; blue, 0 }  ,draw opacity=1 ][line width=1.5]    (341.52,72.18) -- (357.34,74.74) ;
\draw   (353.34,27.52) -- (411.72,27.52) -- (386.7,66.41) -- (328.32,66.41) -- cycle ;
\draw    (386.7,110.62) -- (386.7,66.41) ;
\draw [color={rgb, 255:red, 155; green, 155; blue, 155 }  ,draw opacity=1 ]   (232.21,71.04) .. controls (155.52,75.24) and (255.53,185) .. (357.65,86.17) ;
\draw [shift={(357.65,86.17)}, rotate = 495.94] [color={rgb, 255:red, 155; green, 155; blue, 155 }  ,draw opacity=1 ][line width=0.75]    (10.93,-3.29) .. controls (6.95,-1.4) and (3.31,-0.3) .. (0,0) .. controls (3.31,0.3) and (6.95,1.4) .. (10.93,3.29)   ;
\draw [shift={(234.58,70.94)}, rotate = 178.55] [color={rgb, 255:red, 155; green, 155; blue, 155 }  ,draw opacity=1 ][line width=0.75]    (10.93,-3.29) .. controls (6.95,-1.4) and (3.31,-0.3) .. (0,0) .. controls (3.31,0.3) and (6.95,1.4) .. (10.93,3.29)   ;
\draw [color={rgb, 255:red, 155; green, 155; blue, 155 }  ,draw opacity=1 ]   (252.93,88.99) .. controls (264.41,179.75) and (505.95,92.63) .. (399.14,67.38) ;
\draw [shift={(397.5,67)}, rotate = 372.58000000000004] [color={rgb, 255:red, 155; green, 155; blue, 155 }  ,draw opacity=1 ][line width=0.75]    (10.93,-3.29) .. controls (6.95,-1.4) and (3.31,-0.3) .. (0,0) .. controls (3.31,0.3) and (6.95,1.4) .. (10.93,3.29)   ;
\draw [shift={(252.65,86.17)}, rotate = 85.91] [color={rgb, 255:red, 155; green, 155; blue, 155 }  ,draw opacity=1 ][line width=0.75]    (10.93,-3.29) .. controls (6.95,-1.4) and (3.31,-0.3) .. (0,0) .. controls (3.31,0.3) and (6.95,1.4) .. (10.93,3.29)   ;
\draw [color={rgb, 255:red, 155; green, 155; blue, 155 }  ,draw opacity=1 ]   (297.48,71.32) .. controls (316.35,71.97) and (327.77,72) .. (339.64,72.16) ;
\draw [shift={(341.52,72.18)}, rotate = 180.8] [color={rgb, 255:red, 155; green, 155; blue, 155 }  ,draw opacity=1 ][line width=0.75]    (10.93,-3.29) .. controls (6.95,-1.4) and (3.31,-0.3) .. (0,0) .. controls (3.31,0.3) and (6.95,1.4) .. (10.93,3.29)   ;
\draw [shift={(295.28,71.25)}, rotate = 2.03] [color={rgb, 255:red, 155; green, 155; blue, 155 }  ,draw opacity=1 ][line width=0.75]    (10.93,-3.29) .. controls (6.95,-1.4) and (3.31,-0.3) .. (0,0) .. controls (3.31,0.3) and (6.95,1.4) .. (10.93,3.29)   ;
\draw [color={rgb, 255:red, 155; green, 155; blue, 155 }  ,draw opacity=1 ]   (278.04,45.24) .. controls (301.8,10.13) and (362.63,4.97) .. (381.07,46.19) ;
\draw [shift={(381.62,47.45)}, rotate = 247.37] [color={rgb, 255:red, 155; green, 155; blue, 155 }  ,draw opacity=1 ][line width=0.75]    (10.93,-3.29) .. controls (6.95,-1.4) and (3.31,-0.3) .. (0,0) .. controls (3.31,0.3) and (6.95,1.4) .. (10.93,3.29)   ;
\draw [shift={(276.62,47.45)}, rotate = 301.43] [color={rgb, 255:red, 155; green, 155; blue, 155 }  ,draw opacity=1 ][line width=0.75]    (10.93,-3.29) .. controls (6.95,-1.4) and (3.31,-0.3) .. (0,0) .. controls (3.31,0.3) and (6.95,1.4) .. (10.93,3.29)   ;
\draw [color={rgb, 255:red, 0; green, 0; blue, 0 }  ,draw opacity=1 ]   (66.55,71.25) .. controls (41.37,22.42) and (123.77,19.37) .. (101.64,74.3) .. controls (65.02,159.74) and (-2.87,74.3) .. (94.01,72.77) ;
\draw [color={rgb, 255:red, 0; green, 0; blue, 0 }  ,draw opacity=1 ]   (107.36,73.53) .. controls (151.99,75.82) and (131.4,158.21) .. (70.36,78.87) ;
\draw  [color={rgb, 255:red, 0; green, 0; blue, 0 }  ,draw opacity=1 ] (83.72,97.95) .. controls (83.72,95.84) and (85.42,94.13) .. (87.53,94.13) .. controls (89.64,94.13) and (91.34,95.84) .. (91.34,97.95) .. controls (91.34,100.05) and (89.64,101.76) .. (87.53,101.76) .. controls (85.42,101.76) and (83.72,100.05) .. (83.72,97.95) -- cycle ;
\draw    (140,69) -- (197.5,69) ;
\draw [shift={(199.5,69)}, rotate = 180] [color={rgb, 255:red, 0; green, 0; blue, 0 }  ][line width=0.75]    (10.93,-3.29) .. controls (6.95,-1.4) and (3.31,-0.3) .. (0,0) .. controls (3.31,0.3) and (6.95,1.4) .. (10.93,3.29)   ;

\end{tikzpicture}
    \vspace{-40pt}
    \caption{Construction of a good triangulation of the canonical space realization}
    \label{fig:canonicalSpaceRealizationTri}
\end{figure}

\begin{cor}\label{cor:comp}
    We can construct a triangulation $\mathcal{T}_E$ of the canonical exterior of $D$ in time $\mathcal{O}(c)$.
    Moreover, $|\mathcal{T}_E| \in \mathcal{O}(c)$.
\end{cor}
\begin{proof}
    We can construct the good triangulation $\mathcal{T}$ of the canonical space realization of $D$ in time $\mathcal{O}(c)$.
    We obtain $\mathcal{T}_E$ by doing the barycentric subdivision of $\mathcal{T}$ twice and then 
    removing the tetrahedra adjacent to $\hat{D}$.
\end{proof}

From Lemma \ref{lem:goodTri}, the following lemma holds.
\begin{cor}\label{cor:genus}
    The supporting genus $sg(D)$ of a virtual link diagram $D$ is $\mathcal{O}(c)$.
\end{cor}
\begin{proof}
    Let $\mathcal{S}$ be the supporting surface of $D$, and let $\mathcal{T}_{\mathcal{S}}$ be the triangulation of $\mathcal{S}$ obtained by Lemma \ref{lem:goodTri}.
    We denote the number of faces, edges and vertices of $\mathcal{T}_\mathcal{S}$ by $f, e$ and $v$, respectively.
    By the relationship of the genus of $\mathcal{S}$ and the Euler characteristic of $\mathcal{T}_\mathcal{S}$
    we have
    \begin{align*}
        2-2g(\mathcal{S}) = \chi(\mathcal{T}_{\mathcal{S}})
        \iff g(\mathcal{S}) &= 1 - \frac{\chi(\mathcal{T}_{\mathcal{S}})}{2}\\
        &= 1 - \frac{f-e+v}{2}.
    \end{align*}
    Since $e \leq 3|\mathcal{T}_{\mathcal{S}}|$ and $f,v > 0$, we see that
    \begin{align*}
        g(\mathcal{S}) &= 1 - \frac{f-e+v}{2}\\
        &\leq 1 + \frac{e}{2}\\
        &\leq 1 + \frac{3|\mathcal{T}_{\mathcal{S}}|}{2}.
    \end{align*}
    We have $g(\mathcal{S}) \in \mathcal{O}(c)$ because $|\mathcal{T}_{\mathcal{S}}| \in \mathcal{O}(c)$.
    Therefore, $sg(D) \in \mathcal{O}(c)$. 
\end{proof}

\setcounter{section}{3}
\section{Normal surface theory}

In the first half of this section, we give a brief overview of normal surface theory and 
show that one of normal surfaces which is a witness of classical knot recognition can be found as a {\it vertex surface}.
The Turing machine given in the proof of Theorem \ref{thm:main} makes use of the {\it crushing procedure} to reduce its running time.
Jaco and Rubinstein defined the crushing procedure on a triangulation $\mathcal{T}$ along a normal surface $F$ and 
analyzed its effect in the case where $F$ is a disk or a $2$-sphere in an orientable compact $3$-manifold (\cite{JR}).
In addition, Burton generalized this result to the setting of non-orientable $3$-manifolds (\cite{B_crush}).
In the last half of this section, we analyze the effect of the crushing procedure along a vertical normal annulus in the exterior of a link in a thickened orientable closed surface.

\subsection{The definition of a normal surface}
Suppose that $\mathcal{T}$ is a triangulation of a compact 3-manifold $M$ and $\Delta$ is a tetrahedron in $\mathcal{T}$.
\begin{defi}[Normal disk]
    A {\it normal disk} in $\Delta$ is a properly embedded disk $D$ in $\Delta$ if $D$ satisfies the following conditions:
    \begin{itemize}
        \item $D$ is a triangle or a quadrilateral,
        \item $D$ has no intersection with the vertices of $\Delta$,
        \item Each edge of $D$ connects different edges of $\Delta$.
    \end{itemize}
\end{defi}
There are seven types of normal disks in $\Delta$ as shown in Figure \ref{fig:normalDisk}.
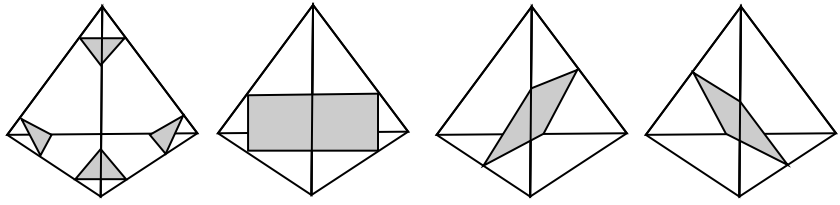
\begin{figure}[htbp]
    \centering
    \tikzset{every picture/.style={line width=0.75pt}} 

\begin{tikzpicture}[x=0.75pt,y=0.75pt,yscale=-1,xscale=1]

\draw   (88.45,28.8) -- (87.65,124.71) -- (40.5,93.54) -- cycle ;
\draw   (88.45,28.8) -- (136.41,92.74) -- (40.5,93.54) -- cycle ;
\draw  [fill={rgb, 255:red, 204; green, 204; blue, 204 }  ,fill opacity=1 ] (87.65,100.73) -- (100.44,115.92) -- (74.87,115.92) -- cycle ;
\draw  [fill={rgb, 255:red, 204; green, 204; blue, 204 }  ,fill opacity=1 ] (47.69,85.54) -- (62.88,93.54) -- (57.28,103.93) -- cycle ;
\draw  [fill={rgb, 255:red, 204; green, 204; blue, 204 }  ,fill opacity=1 ] (129.22,83.95) -- (120.42,103.13) -- (112.43,93.54) -- cycle ;
\draw  [fill={rgb, 255:red, 204; green, 204; blue, 204 }  ,fill opacity=1 ] (99.64,44.78) -- (77.26,44.78) -- (87.65,58.37) -- cycle ;
\draw   (194.75,28) -- (193.95,123.91) -- (146.8,92.74) -- cycle ;
\draw   (194.75,28) -- (242.71,91.94) -- (193.95,123.91) -- cycle ;
\draw   (194.75,28) -- (242.71,91.94) -- (146.8,92.74) -- cycle ;
\draw  [fill={rgb, 255:red, 204; green, 204; blue, 204 }  ,fill opacity=1 ] (161.98,73.56) -- (227.52,72.76) -- (227.52,101.53) -- (161.98,101.53) -- cycle ;
\draw   (305.05,28.8) -- (304.25,124.71) -- (257.09,93.54) -- cycle ;
\draw   (305.05,28.8) -- (353,92.74) -- (304.25,124.71) -- cycle ;
\draw   (305.05,28.8) -- (353,92.74) -- (257.09,93.54) -- cycle ;
\draw  [fill={rgb, 255:red, 204; green, 204; blue, 204 }  ,fill opacity=1 ] (305,70) -- (328.02,60.77) -- (311,93) -- (280.67,109.12) -- cycle ;
\draw   (410.55,28.8) -- (409.75,124.71) -- (362.59,93.54) -- cycle ;
\draw   (410.55,28.8) -- (458.5,92.74) -- (409.75,124.71) -- cycle ;
\draw   (410.55,28.8) -- (458.5,92.74) -- (362.59,93.54) -- cycle ;
\draw  [fill={rgb, 255:red, 204; green, 204; blue, 204 }  ,fill opacity=1 ] (386.57,62.17) -- (410.15,76.75) -- (434.12,108.72) -- (403,93) -- cycle ;
\draw    (194.75,28) -- (193.95,123.91) ;
\draw    (305.05,28.8) -- (304.25,124.71) ;
\draw    (410.55,28.8) -- (409.75,124.71) ;
\draw   (88.45,28.8) -- (136.41,92.74) -- (87.65,124.71) -- cycle ;

\end{tikzpicture}
    \caption{The seven types of normal disks}
    \label{fig:normalDisk}
\end{figure}

\begin{defi}[Normal surface]
    A {\it normal surface} in $M$ with respect to $\mathcal{T}$ is a properly embedded surface meeting each tetrahedron of $\mathcal{T}$ in a collection of disjoint normal disks.
\end{defi}
In particular, the boundary of a small regular neighborhood of a vertex of $\mathcal{T}$ is a normal surface.
This normal surface is called a {\it vertex link}.
Then the next lemma immediately follows.
\begin{lem}
    A normal surface $S$ is a disjoint union of vertex links if and only if any normal disk in $S$ is a triangle.
\end{lem}

Let $M$ be a $3$-manifold with an $n$-tetrahedra triangulation $\mathcal{T}$.
A normal surface $F$ in $M$ with respect to $\mathcal{T}$ can be described by a vector, a tuple of $7n$ integers, that counts the number of normal disks of each type in each tetrahedron.
This vector uniquely identifies a normal surface with respect to $\mathcal{T}$ up to normal isotopy: an isotopy which is invariant on each simplex of $\mathcal{T}$.

More generally, an integer vector $\mathbf{x} \in \mathbb{R}^{7n}$ represents a normal surface $F$ with respect to $\mathcal{T}$ if and only if:
\begin{enumerate}
    \item $\mathbf{x} \geq \mathbf{0}$,
    \item $\mathbf{x}$ satisfies the \textit{matching equations} $A\mathbf{x} = \mathbf{0}$ of $\mathcal{T}$,
    \item $\mathbf{x}$ satisfies the \textit{quadrilateral condition} for each tetrahedron of $\mathcal{T}$.
\end{enumerate}
The \textit{matching equations} is the condition for gluing normal disks together across adjacent tetrahedra.
Two different type quadrilateral normal disks can not be embedded in the same tetrahedron without intersections.
Therefore, the quadrilateral normal disks in the collection of normal disks $F \cap \Delta$ must have the same type, where $\Delta$ is a tetrahedron of $\mathcal{T}$.
This condition is called the \textit{quadrilateral condition} for $\Delta$.

Let $A\mathbf{x} = \mathbf{0}$ be the matching equations of a triangulation $\mathcal{T}$.
The {\it projective solution space} of the matching equations $A\mathbf{x} = \mathbf{0}$ is the set $\mathcal{P}  = \{ \mathbf{x} \in \mathbb{R}^{7n} | A\mathbf{x}=\mathbf{0} ,  \sum x_i = 1\}$.
\begin{defi}[Vertex solution]
    A vector $\mathbf{x} \in \mathcal{P}$ is called a {\it vertex solution} 
    if there are no vectors $\mathbf{y},\mathbf{z} \in \mathcal{P}$ such that the projection of $\mathbf{y} + \mathbf{z}$ is $\mathbf{x}$.
\end{defi}

\begin{defi}[Vertex surface]
    Suppose that $F$ is a connected two-sided normal surface with respect to a triangulation $\mathcal{T}$ and $\mathbf{x}$ is the vector representation of $F$.
    $F$ is called a {\it vertex surface} if the projection of $\mathbf{x}$ onto the hyperplane $\sum x_i = 1$ is a vertex solution of $\mathcal{P}$, where $\mathcal{P}$ is the projective solution space of the matching equations of $\mathcal{T}$.
\end{defi}
Note that if there is a one-sided connected normal surface $F$ that the projection of the vector representation $\mathbf{x}$ of $F$ is a vertex solution, then the normal surface represented by $2\mathbf{x}$ is a vertex surface.

The following theorem is used to analyze the computational complexity of an algorithm using vertex surfaces.
\begin{thm}[Hass, Lagarias and Pippenger \cite{HLP}] \label{thm:HLP}
    Suppose that $F$ is a vertex surface with respect to an $n$-tetrahedra triangulation of a $3$-manifold and $\mathbf{x} \in \mathbb{R}^{7n}$ is the vector representation of $F$.
    Then, $x_i \leq 2^{7n-1}$ for each $i \in \{1, \ldots, 7n\}$.
\end{thm}

\subsection{Vertex surfaces in the exterior of a link in a thickened closed orientable surface}
First, we consider vertex $2$-spheres in the exterior of a link in a thickened closed orientable surface.
Jaco and Tollefson showed that there is an essential vertex $2$-sphere in a reducible closed $3$-manifold $M$ with respect to a triangulation of $M$ (\cite[Theorem 5.2]{JT}).
The exterior of a link in a thickened closed orientable surface has a boundary, however, 
the following theorem can be proved as the same argument of \cite[Theorem 5.2]{JT}.
\begin{thm}\label{thm:essS2}
    Let $\hat{D}$ be a link in a thickened closed orientable surface $\mathcal{S} \times I$.
    Let $M$ be the exterior of $\hat{D}$ with a triangulation $\mathcal{T}$.
    Suppose that $g(\mathcal{S}) \neq 0$.
    Then, if $M$ is reducible, then there is an essential vertex $2$-sphere with respect to $\mathcal{T}$.
\end{thm}

Let $M$ be the exterior of a link $\hat{D}$ in a thickened closed orientable surface $\mathcal{S} \times I$, and suppose that $g(\mathcal{S}) \neq 0$.
By Theorem \ref{thm:essS2}, whenever $M$ is reducible, we can split $M$ by using a vertex $2$-sphere.

Next, we consider vertex annuli in the exterior $M$ of a link in a thickened closed orientable surface $\mathcal{S} \times I$ with respect to a triangulation $\mathcal{T}$ of $M$.
Even though there is a vertical essential annulus in $M$, there is not necessarily a vertex annulus which is vertical and essential.
However, if $M$ contains a vertical essential annulus, then there is a vertex annulus which is vertical and essential or a vertex annulus which is a classicalization annulus with respect to $\mathcal{T}$ as shown in Theorem \ref{thm:essAn}.

The {\it weight} of a normal surface $F$ with respect to a triangulation $\mathcal{T}$, denoted by $wt(F)$, is the number of intersections $F \cap \mathcal{T}^{(1)}$, where $\mathcal{T}^{(1)}$ is the $1$-skeleton of $\mathcal{T}$.
We say that a normal surface $F$ is {\it least weight} if $wt(F)$ takes the minimal value among all normal surfaces which are isotopic to $F$.

Suppose that $F_1$ and $F_2$ are normal surfaces with respect to a triangulation $\mathcal{T}$ such that there are no tetrahedra containing different types quadrilateral disks in $F_1 \cup F_2$.
Then the {\it sum} of $F_1$ and $F_2$, denoted by $F_1 + F_2$, is the normal surface $F$ represented by the vector $\mathbf{x}_F = \mathbf{x}_{F_1} + \mathbf{x}_{F_2}$, where $\mathbf{x}_{F_1}$ and $\mathbf{x}_{F_2}$ is the vector representation of $F_1$ and $F_2$, respectively.
For any integer $n > 0$, let the \textit{integer multiple} of $F$, denoted by $nF$, is the normal surface represented by the vector $n\mathbf{x}_F$, where $\mathbf{x}_F$ is the vector represetation of $F$.
Under the assumption that normal surfaces $F_1$ and $F_2$ intersect transversely, 
the sum $F = F_1 + F_2$ is said to be in {\it reduced form} if $F$ can not be written as $F = F'_1 + F'_2$, where $F'_i$ is a normal surface which is isotopic to $F_i$ for each $i = 1,2$ and $F'_1 \cap F'_2$ has fewer components than $F_1 \cap F_2$.
A {\it patch} of $F_1 + F_2$ is a subsurface of $F_1 \cup F_2$ whose boundary consists of $F_1 \cap F_2$.
In particular, a patch of $F_1 + F_2$ which is a disk is called a {\it disk patch}.

Jaco and Tollefson \cite{JT} showed Theorem \ref{thm:disk_patch} and Theorem \ref{thm:sum_incomp} by extending the results of Jaco and Oertel \cite{JO}.
\begin{thm}[Jaco and Tollefson {\cite[Lemma 6.6]{JT}}] \label{thm:disk_patch}
    Let $M$ be an irreducible, $\partial$-irreducible $3$-manifold and $\mathcal{T}$ be a triangulation of $M$.
    Suppose that $F$ is a least weight, incompressible, $\partial$-incomprerssible, two-sided normal surface in $M$ with respect to $\mathcal{T}$ and $F$ is a not disk.
    If there are an integer $n > 0$ and normal surfaces $F_1$ and $F_2$ such that $nF = F_1 + F_2$ is in reduced form and each intersection curve of $F_1 \cap F_2$ is two-sided in $F_1$ and $F_2$, then each patch of $F_1 + F_2$ is incompressible and $\partial$-incompressible and there are no disk patch.
\end{thm}

Let $M$ be an irreducible, $\partial$-irreducible $3$-manifold with a triangulation $\mathcal{T}$ and $F$ be a normal surface which is neither a $2$-sphere nor a disk in $M$ with respect to $\mathcal{T}$.
Suppose that $F$ is least weight, incompressible, $\partial$-incomprerssible and two-sided, $nF = F_1 + F_2$, and each intersection curve in $F_1 \cap F_2$ is two-sided.
If $F_1 + F_2$ is not in reduced form, then let $F'_1$ and $F'_2$ be normal surfaces such that $nF = F'_1 + F'_2$ is in reduced form and each $F'_i$ is isotopic to $F_i$.
If $F'_i$ is a normal $2$-sphere, then there is a disk patch in $F'_i$, and this contradicts Theorem \ref{thm:disk_patch}.
Therefore, each $F'_i$ is not a $2$-sphere, and so $F_i$ is not a $2$-sphere.
Similarly, each $F_i$ is not a disk.

\begin{thm}[Jaco and Tollefson {\cite[Theorem 6.5]{JT}}] \label{thm:sum_incomp}
    Let $M$ be an irreducible, $\partial$-irreducible $3$-manifold and $\mathcal{T}$ be a triangulation of $M$.
    Suppose that $F$ is a least weight normal surface in $M$ with respect to $\mathcal{T}$, there are an integer $n > 0$ and normal surfaces $F_1$ and $F_2$ such that $nF = F_1 + F_2$, and $F$ is neither a disk nor a $2$-sphere, where $F_1$ and $F_2$ are normal surfaces in $\mathcal{T}$.
    If $F$ is two-sided, incompressible, and $\partial$-incompressible, then each $F_i$ is incompressible and $\partial$-incompressible and not a disk.
\end{thm}

\begin{thm}\label{thm:essAn}
    Let $M$ be the exterior of a link in a thickened closed orientable surface $\mathcal{S} \times I$.
    Suppose that $g(\mathcal{S}) \neq 0$ and $M$ is irreducible and $\partial$-irreducible.
    Let $\mathcal{T}$ be a triangulation of $M$.
    If there is a vertical essential annulus in $M$, then there is a vertex annulus which is vertical and essential or a vertex annulus which is a classicalization annulus with respect to $\mathcal{T}$.
\end{thm}
\begin{proof}
    Suppose that $A$ is a vertical essential annulus in $M$.
    Since $M$ is irreducible and $\partial$-irreducible and $A$ is incompressible in $M$, if $A$ is not a normal surface with respect to $\mathcal{T}$, then there is a normal surface which is isotopic to $A$, which we again denote by $A$ (See \cite{MR2341532} for details).
    Let us assume that $A$ is least weight in its isotopy class.
    
    If $A$ is not a vertex surface, then there are an integer $n>0$ and vertex surfaces $F_1, \dots , F_k$ such that $nA = F_1 + F_2 + \dots + F_k$.
    As we have already observed, each $F_i$ is neither a disk nor a $2$-sphere, and hence $\chi(F_i) \leq 0$ for each $F_i$, where $\chi(F_i)$ is the Euler charactaristic of $F_i$.
    On the other hand, by computing the Euler charactaristic of $nA$ and $F_1 + \dots + F_k$, we have $\chi(nA) = \sum_{i=1}^k \chi(F_i) = 0$.
    Thus, $\chi(F_i) = 0$ for each $i = 1, \dots, k$, and hence each $F_i$ is a torus or an annulus.
    Since $A$ is not closed, there is at least one vertex surface $F_j$ which is not closed, i.e., $F_j$ is an annulus.
    By Theorem \ref{thm:sum_incomp}, $F_j$ is incompressible and $\partial$-incompressible.
    Because $A$ does not intersect $\partial \mathcal{T}_{N(\hat{D})}$, $F_j \cap \partial \mathcal{T}_{N(\hat{D})} = \emptyset$, and so $\partial F_j \subset \mathcal{S} \times \{0\} \cup \mathcal{S} \times \{1\}$.
    Therefore, $F_j$ is a vertex annulus which is vertical and essential or a vertex annulus which is a classicalization annulus with respect to $\mathcal{T}$.
\end{proof}

\subsection{The crushing procedure}
Let $M$ be a compact $3$-manifold with an $n$-tetrahedra triangulation $\mathcal{T}$ and $F$ be a vertex surface in $M$ with respect to $\mathcal{T}$.
Since $F$ may contains $2^{7n-1}$ normal disks, the operation cutting $\mathcal{T}$ along $F$ takes exponential time, 
and so splitting and destabilization take exponential time and add exponential cells if we run these operations simply.
In Section 5, we show that there are algorithms to run splitting and destabilization in polynomial time.
In order to reduce the running time of splitting and destabilization, we use the {\it crushing procedure} along $F$ instead of cutting $\mathcal{T}$ along $F$.
The crushing procedure is defined by Jaco and Rubinstein (\cite{JR}), runs in polynomial time of $n$, and does not increase the number of tetrahedra.

\begin{defi}[Jaco and Rubinstein, \cite{JR}] \label{def:crush}
    Suppose that $\mathcal{T}$ is a triangulation of a compact 3-manifold $M$ and $F$ is a connected normal surface in $M$ with respect to $\mathcal{T}$.
    The following operation is called the {\it crushing procedure} along $F$.
    \begin{enumerate}
        \item Cut $M$ open along $F$.
                $M'$ denotes the resulting $3$-manifold, and we obtain the new cell-decomposition $\mathcal{C}'$ of $M'$ from $\mathcal{T}$.
                If $F$ is two-sided in $M$, we obtain two copies of $F$ in $\partial M'$,
                otherwise, $F$ is one-sided, we obtain the double cover of $F$ in $\partial M'$.
        \item Shrink each copy of $F$ or the double cover of $F$ to a point.
                Let $M$ denote the resulting topological space and $\mathcal{C}$ denote the cell decomposition obtained from $\mathcal{C}'$.
                Note that $M$ may be not a $3$-manifold.
                Now, each cell of $\mathcal{C}$ is one of the following cells:
        \begin{itemize}
            \item 3-sided foot ball:\\
                    a cell obtained from a region between two parallel triangular normal disks or a vertex of the tetrahedron and a triangular normal disk.
            \item 4-sided foot ball:\\
                    a cell obtained from an region between two parallel quadrilateral normal disks.
            \item triangular purse:\\
                    a cell obtained from an region between a triangular normal disk and a  quadrilateral normal disk.
            \item tetrahedron:\\
                    a cell obtained from a tetrahedron which has no quadrilateral normal surfaces.
        \end{itemize}

        \item Flatten each football to an edge and each triangular purse to a face as shown in Figure \ref{fig:flat}.
        \item Remove edges and faces that do not belong to any tetrahedron and break apart tetrahedra which are connected only by vertices or edges.
    \end{enumerate}
\end{defi}

\begin{figure}[htbp]
    \centering
    \tikzset{every picture/.style={line width=0.75pt}} 

\begin{tikzpicture}[x=0.75pt,y=0.75pt,yscale=-1,xscale=1]

\draw    (54.7,109.11) .. controls (51.8,58.49) and (80.77,16.6) .. (119.88,20.09) ;
\draw    (54.7,109.11) .. controls (115.54,116.09) and (124.23,56.74) .. (119.88,20.09) ;
\draw    (54.7,109.11) .. controls (98.92,86.61) and (116,46.46) .. (119.88,20.09) ;
\draw [color={rgb, 255:red, 155; green, 155; blue, 155 }  ,draw opacity=1 ]   (70.73,53.3) .. controls (70.73,60.98) and (75,72.09) .. (86.1,72.94) ;
\draw [color={rgb, 255:red, 155; green, 155; blue, 155 }  ,draw opacity=1 ]   (99.77,78.07) .. controls (101.48,81.48) and (109.17,78.92) .. (111.73,72.94) ;
\draw [color={rgb, 255:red, 155; green, 155; blue, 155 }  ,draw opacity=1 ] [dash pattern={on 0.84pt off 2.51pt}]  (80.98,50.73) .. controls (93.79,52.44) and (104.9,57.57) .. (108.31,68.67) ;
\draw    (204.37,109.11) .. controls (201.47,58.49) and (230.44,16.6) .. (269.55,20.09) ;
\draw    (204.37,109.11) .. controls (265.2,116.09) and (273.89,56.74) .. (269.55,20.09) ;
\draw    (204.37,109.11) .. controls (248.58,86.61) and (265.67,46.46) .. (269.55,20.09) ;
\draw [color={rgb, 255:red, 155; green, 155; blue, 155 }  ,draw opacity=1 ]   (220.4,53.3) .. controls (220.4,60.98) and (224.67,72.09) .. (235.77,72.94) ;
\draw [color={rgb, 255:red, 155; green, 155; blue, 155 }  ,draw opacity=1 ]   (249.44,78.07) .. controls (251.15,81.48) and (258.83,78.92) .. (261.4,72.94) ;
\draw [color={rgb, 255:red, 155; green, 155; blue, 155 }  ,draw opacity=1 ] [dash pattern={on 0.84pt off 2.51pt}]  (223.81,49.88) .. controls (228.94,48.17) and (233.21,49.88) .. (237.48,52.44) ;
\draw [color={rgb, 255:red, 0; green, 0; blue, 0 }  ,draw opacity=1 ] [dash pattern={on 0.84pt off 2.51pt}]  (204.37,109.11) .. controls (222.11,78.92) and (246.02,37.07) .. (269.55,20.09) ;
\draw [color={rgb, 255:red, 155; green, 155; blue, 155 }  ,draw opacity=1 ] [dash pattern={on 0.84pt off 2.51pt}]  (246.02,55.86) .. controls (252,59.28) and (257.98,64.4) .. (258.83,67.82) ;
\draw    (325,110) .. controls (328.5,66) and (354.5,31) .. (383.5,22) ;
\draw    (325,110) -- (433.5,110) ;
\draw  [dash pattern={on 0.84pt off 2.51pt}]  (325,110) .. controls (362.5,80) and (375.5,51) .. (383.5,22) ;
\draw    (433.5,110) .. controls (413.5,92) and (391.5,57) .. (383.5,22) ;
\draw    (433.5,110) .. controls (435.5,86) and (414.5,30) .. (383.5,22) ;
\draw [color={rgb, 255:red, 155; green, 155; blue, 155 }  ,draw opacity=1 ]   (367.5,79) .. controls (374.5,86) and (386.5,88) .. (396.5,83) ;
\draw [color={rgb, 255:red, 155; green, 155; blue, 155 }  ,draw opacity=1 ] [dash pattern={on 0.84pt off 2.51pt}]  (372.5,69) .. controls (382.5,65) and (392.5,68) .. (397.5,75) ;
\draw [color={rgb, 255:red, 155; green, 155; blue, 155 }  ,draw opacity=1 ]   (220.4,53.3) .. controls (220.4,60.98) and (224.67,72.09) .. (235.77,72.94) ;

\draw (85.87,127.5) node   [align=left] {3-sided football};
\draw (236.87,127.5) node   [align=left] {4-sided football};
\draw (380.87,127.5) node   [align=left] {triangular purse};

\end{tikzpicture}
    \caption{Cells in the cell-decomposition}
    \label{fig:cell}
\end{figure}
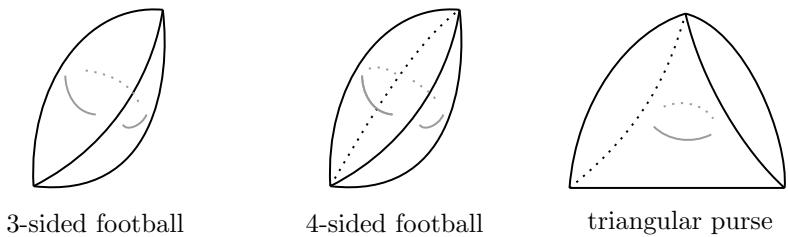

\begin{figure}[htbp]
    \centering
    \tikzset{every picture/.style={line width=0.75pt}} 

\begin{tikzpicture}[x=0.75pt,y=0.75pt,yscale=-0.7,xscale=0.7]

\draw    (10.7,110.11) .. controls (7.8,59.49) and (36.77,17.6) .. (75.88,21.09) ;
\draw    (10.7,110.11) .. controls (71.54,117.09) and (80.23,57.74) .. (75.88,21.09) ;
\draw    (10.7,110.11) .. controls (54.92,87.61) and (72,47.46) .. (75.88,21.09) ;
\draw [color={rgb, 255:red, 155; green, 155; blue, 155 }  ,draw opacity=1 ]   (26.73,54.3) .. controls (26.73,61.98) and (31,73.09) .. (42.1,73.94) ;
\draw [color={rgb, 255:red, 155; green, 155; blue, 155 }  ,draw opacity=1 ]   (55.77,79.07) .. controls (57.48,82.48) and (65.17,79.92) .. (67.73,73.94) ;
\draw [color={rgb, 255:red, 155; green, 155; blue, 155 }  ,draw opacity=1 ] [dash pattern={on 0.84pt off 2.51pt}]  (36.98,51.73) .. controls (49.79,53.44) and (60.9,58.57) .. (64.31,69.67) ;
\draw    (197.37,112.11) .. controls (194.47,61.49) and (223.44,19.6) .. (262.55,23.09) ;
\draw    (197.37,112.11) .. controls (258.2,119.09) and (266.89,59.74) .. (262.55,23.09) ;
\draw    (197.37,112.11) .. controls (241.58,89.61) and (258.67,49.46) .. (262.55,23.09) ;
\draw [color={rgb, 255:red, 155; green, 155; blue, 155 }  ,draw opacity=1 ]   (213.4,56.3) .. controls (213.4,63.98) and (217.67,75.09) .. (228.77,75.94) ;
\draw [color={rgb, 255:red, 155; green, 155; blue, 155 }  ,draw opacity=1 ]   (242.44,81.07) .. controls (244.15,84.48) and (251.83,81.92) .. (254.4,75.94) ;
\draw [color={rgb, 255:red, 155; green, 155; blue, 155 }  ,draw opacity=1 ] [dash pattern={on 0.84pt off 2.51pt}]  (216.81,52.88) .. controls (221.94,51.17) and (226.21,52.88) .. (230.48,55.44) ;
\draw [color={rgb, 255:red, 0; green, 0; blue, 0 }  ,draw opacity=1 ] [dash pattern={on 0.84pt off 2.51pt}]  (197.37,112.11) .. controls (215.11,81.92) and (239.02,40.07) .. (262.55,23.09) ;
\draw [color={rgb, 255:red, 155; green, 155; blue, 155 }  ,draw opacity=1 ] [dash pattern={on 0.84pt off 2.51pt}]  (239.02,58.86) .. controls (245,62.28) and (250.98,67.4) .. (251.83,70.82) ;
\draw    (374,112) .. controls (377.5,68) and (403.5,33) .. (432.5,24) ;
\draw    (374,112) -- (482.5,112) ;
\draw  [dash pattern={on 0.84pt off 2.51pt}]  (374,112) .. controls (411.5,82) and (424.5,53) .. (432.5,24) ;
\draw    (482.5,112) .. controls (462.5,94) and (440.5,59) .. (432.5,24) ;
\draw    (482.5,112) .. controls (484.5,88) and (463.5,32) .. (432.5,24) ;
\draw [color={rgb, 255:red, 155; green, 155; blue, 155 }  ,draw opacity=1 ]   (416.5,81) .. controls (423.5,88) and (435.5,90) .. (445.5,85) ;
\draw [color={rgb, 255:red, 155; green, 155; blue, 155 }  ,draw opacity=1 ] [dash pattern={on 0.84pt off 2.51pt}]  (421.5,71) .. controls (431.5,67) and (441.5,70) .. (446.5,77) ;
\draw [color={rgb, 255:red, 155; green, 155; blue, 155 }  ,draw opacity=1 ]   (213.4,56.3) .. controls (213.4,63.98) and (217.67,75.09) .. (228.77,75.94) ;
\draw    (81,61) -- (138.5,61) ;
\draw [shift={(140.5,61)}, rotate = 180] [color={rgb, 255:red, 0; green, 0; blue, 0 }  ][line width=0.75]    (10.93,-3.29) .. controls (6.95,-1.4) and (3.31,-0.3) .. (0,0) .. controls (3.31,0.3) and (6.95,1.4) .. (10.93,3.29)   ;
\draw    (117.7,109.11) -- (182.88,20.09) ;
\draw    (269,62) -- (326.5,62) ;
\draw [shift={(328.5,62)}, rotate = 180] [color={rgb, 255:red, 0; green, 0; blue, 0 }  ][line width=0.75]    (10.93,-3.29) .. controls (6.95,-1.4) and (3.31,-0.3) .. (0,0) .. controls (3.31,0.3) and (6.95,1.4) .. (10.93,3.29)   ;
\draw    (305.7,110.11) -- (370.88,21.09) ;
\draw    (483,60) -- (540.5,60) ;
\draw [shift={(542.5,60)}, rotate = 180] [color={rgb, 255:red, 0; green, 0; blue, 0 }  ][line width=0.75]    (10.93,-3.29) .. controls (6.95,-1.4) and (3.31,-0.3) .. (0,0) .. controls (3.31,0.3) and (6.95,1.4) .. (10.93,3.29)   ;
\draw  [fill={rgb, 255:red, 204; green, 204; blue, 204 }  ,fill opacity=1 ] (582.5,22) -- (632.5,110) -- (524,110) -- cycle ;

\end{tikzpicture}
    \caption{Flattening footballs and triangular purse}
    \label{fig:flat}
\end{figure}
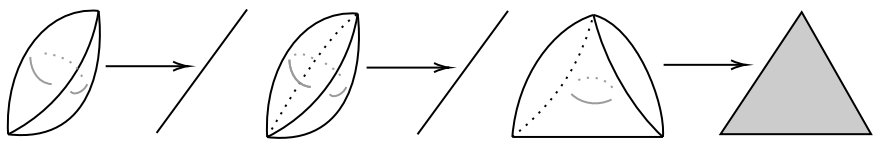

Furthermore, we call the operations up to the second step of Definition \ref{def:crush} the \textit {non-destructive crushing procedure}.

The crushing procedure may change the underlying 3-manifold of a triangulation.
Furthermore, the underlying space after the crushing procedure may not be a 3-manifold.
Jaco and Rubinstein described the effect of the crushing procedure only when the underlying 3-manifold is orientable.
Then, Burton described the effect in general case by using a sequential combination of the atomic moves.
\begin{defi}[Burton \cite{B_crush}]
    Let $\mathcal{T}$ be a triangulation of a compact $3$-manifold $M$ and $S$ be a normal surface in $M$ with respect to $\mathcal{T}$.  
    Let $\mathcal{C}$ denote the cell-decomposition obtained by the non-destructive crushing procedure using $S$.
    We call the following three moves on $\mathcal{C}$ the {\it atomic moves}:
    \begin{itemize}
        \item flatting a triangular pillow to a triangle as shown in Figure \ref{fig:atomic}\subref{fig:atomic1},
        \item flatting a bigonal pillow to a bigon as shown in Figure \ref{fig:atomic}\subref{fig:atomic2},
        \item flatting a bigon to an edge as shown in Figure \ref{fig:atomic}\subref{fig:atomic3}.
    \end{itemize}
    After each atomic move, we remove any 2-faces, edges, or vertices which does not belong to a 3-cell, and we break apart 3-cells which are connected only by vertices or edges.
\end{defi}

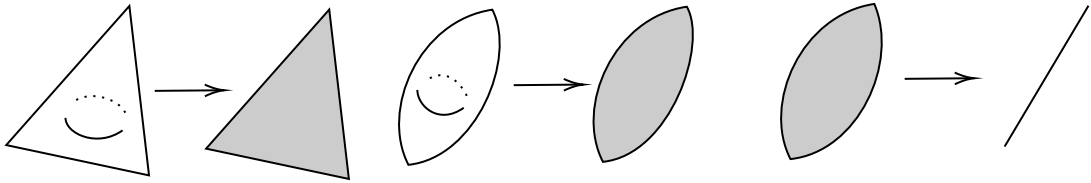
\begin{figure}[htbp]
    \begin{minipage}{0.30\textwidth}
        \centering
        \tikzset{every picture/.style={line width=0.75pt}} 

\begin{tikzpicture}[x=0.75pt,y=0.75pt,yscale=-0.9,xscale=0.9]

\draw   (70.8,1.4) -- (81.8,96.4) -- (1.8,79.4) -- cycle ;
\draw  [fill={rgb, 255:red, 204; green, 204; blue, 204 }  ,fill opacity=1 ] (182.8,3.4) -- (193.8,98.4) -- (113.8,81.4) -- cycle ;
\draw    (35,64.14) .. controls (35,73.14) and (53,81.14) .. (67,71.14) ;
\draw  [dash pattern={on 0.84pt off 2.51pt}]  (41,54.14) .. controls (50,49.14) and (65,53.14) .. (69,62.14) ;
\draw    (85,49) -- (122,48.68) ;
\draw [shift={(124,48.67)}, rotate = 539.51] [color={rgb, 255:red, 0; green, 0; blue, 0 }  ][line width=0.75]    (10.93,-3.29) .. controls (6.95,-1.4) and (3.31,-0.3) .. (0,0) .. controls (3.31,0.3) and (6.95,1.4) .. (10.93,3.29)   ;

\end{tikzpicture}
        \subcaption{Flatting a triangular pillow to a triangle}
        \label{fig:atomic1}
    \end{minipage}
    \begin{minipage}{0.04\textwidth}
    \end{minipage}
    \begin{minipage}{0.31\textwidth}
        \centering
        \tikzset{every picture/.style={line width=0.75pt}} 

\begin{tikzpicture}[x=0.75pt,y=0.75pt,yscale=-0.9,xscale=0.9]

\draw   (69,7.67) .. controls (83,34.67) and (60,90.67) .. (22,94.67) .. controls (6,63.67) and (28,13.67) .. (69,7.67) -- cycle ;
\draw  [fill={rgb, 255:red, 204; green, 204; blue, 204 }  ,fill opacity=1 ] (178,6) .. controls (190,27) and (169,89) .. (131,93) .. controls (115,62) and (137,12) .. (178,6) -- cycle ;
\draw    (27,52.67) .. controls (27,61.67) and (39,72.67) .. (53,62.67) ;
\draw  [dash pattern={on 0.84pt off 2.51pt}]  (34,46.14) .. controls (43,41.14) and (51,47.67) .. (55,56.67) ;
\draw    (81,50) -- (118,49.68) ;
\draw [shift={(120,49.67)}, rotate = 539.51] [color={rgb, 255:red, 0; green, 0; blue, 0 }  ][line width=0.75]    (10.93,-3.29) .. controls (6.95,-1.4) and (3.31,-0.3) .. (0,0) .. controls (3.31,0.3) and (6.95,1.4) .. (10.93,3.29)   ;

\end{tikzpicture}
        \vspace{6mm}
        \subcaption{Flatting a bigonal pillow to a bigon}
        \label{fig:atomic2}
    \end{minipage}
    \begin{minipage}{0.04\textwidth}
    \end{minipage}
    \begin{minipage}{0.31\textwidth}
        \centering
        \vspace{-4mm}
        \tikzset{every picture/.style={line width=0.75pt}} 

\begin{tikzpicture}[x=0.75pt,y=0.75pt,yscale=-0.9,xscale=0.9]

\draw  [fill={rgb, 255:red, 204; green, 204; blue, 204 }  ,fill opacity=1 ] (64,8) .. controls (78,41) and (55,91) .. (17,95) .. controls (1,64) and (23,14) .. (64,8) -- cycle ;
\draw    (81,50) -- (118,49.68) ;
\draw [shift={(120,49.67)}, rotate = 539.51] [color={rgb, 255:red, 0; green, 0; blue, 0 }  ][line width=0.75]    (10.93,-3.29) .. controls (6.95,-1.4) and (3.31,-0.3) .. (0,0) .. controls (3.31,0.3) and (6.95,1.4) .. (10.93,3.29)   ;
\draw    (137,88) -- (184,9) ;

\end{tikzpicture}
        \vspace{7mm}
        \subcaption{Flatting a bigon to an edge}
        \label{fig:atomic3}
    \end{minipage}
    \caption{The atomic moves}
    \label{fig:atomic}
\end{figure}

\begin{lem}[Burton \cite{B_crush}]\label{lem:crushing_lemma}
    Let $\mathcal{T}$ be a triangulation of a compact $3$-manifold $M$ and $S$ be a normal surface in $M$ with respect to $\mathcal{T}$. 
    Let $\mathcal{C}_0$ denote the cell-decomposition obtained from $T$ by the non-destructive crushing procedure using $S$ and 
    $\mathcal{T}'$ denote the triangulation obtained from $\mathcal{T}$ by the crushing procedure using $S$.
    Then, there is a sequence of cell-decompositions $\mathcal{C}_0 \to \mathcal{C}_1 \to \cdots \to \mathcal{C}_n = \mathcal{T}'$, and $\mathcal{C}_{i+1}$ is obtained from $\mathcal{C}_i$ by one of the atomic moves.
\end{lem}
Burton proved the following theorem from Lemma \ref{lem:crushing_lemma}.
\begin{thm}[Burton \cite{B_crush}]\label{thm:B_crush}
    Suppose that $\mathcal{T}$ is a triangulation of a compact 3-manifold $M$ and $S$ is a normal 2-sphere or a normal disk in $M$ with respect to $\mathcal{T}$.
    We denote the triangulation which is obtained by the crushing procedure using $S$ by $\mathcal{T}'$.
    Let $\mathcal{C}_0$ be the cell-decomposition obtained by the non-destructive crushing procedure and $\mathcal{C}_0 \to \mathcal{C}_1 \to \cdots \to \mathcal{C}_n = \mathcal{T}'$ denotes a sequence of the cell-decompositions, where $\mathcal{C}_{i+1}$ is obtained from $\mathcal{C}_i$ by an atomic move.
    Suppose that $\mathcal{C}_0$ contains no two-sided projective planes.
    We denote the underlying $3$-manifolds of $\mathcal{C}_i$ by $M_i$.
    Then, $M_{i+1}$ is obtained from $M_{i}$ by one of the following operations:
    \begin{itemize}
        \item cutting open along a properly embedded disk $S$ in $M_i$,
        \item cutting open along an embedded 2-sphere $S \subset M_i$ and filling the resulting boundary spheres with 3-balls,
        \item removing a 3-ball, a 3-sphere, a lens space $L(3,1)$, a projective space $\mathbb{R}P^3$, $\mathbb{S}^2 \times \mathbb{S}^1$ or a twisted $\mathbb{S}^1$ bundle $\mathbb{S}^2 \tilde{\times} \mathbb{S}^1$ component,
        \item filling a boundary sphere in $\partial M_i$ with a 3-ball.
    \end{itemize}
\end{thm}

Next, we consider that how the crushing procedure affects the number of tetrahedra.
The crushing procedure removes tetrahedra containing at least one quadrilateral normal disk 
and leaves tetrahedra without any quadrilateral normal disks.
Therefore, the crushing procedure does not increase the number of tetrahedra.

The crushing procedure includes the operation to cut open a triangulation along a normal surface.
Thus, it takes exponential time to execute simply if the normal surface contains exponential normal disks.
However, it is known that there is an algorithm runs the crushing procedure in linear time.
\begin{thm}[Burton \cite{B_unknot}] \label{thm:crush}
    Let $\mathcal{T}$ be an $n$-tetrahedra triangulation of a 3-manifold $M$ and $F$ be a normal surface in $M$ with respect to $\mathcal{T}$.
    We can run the crushing procedure using $F$ in time $\mathcal{O}(n)$, where $n$ is the number of tetrahedra in $\mathcal{T}$.
\end{thm}

\subsection{The effect of the crushing procedure using an annulus}
In Section 5, we give an algorithm to destabilize a triangulation of a $3$-manifold.
In the algorithm, we run the crushing procedure using a normal annulus.
In order to analyze the effect of the crushing procedure, we fix and use the following notation in the rest of this subsection.
Let $\mathcal{T}$ be a triangulation of the exterior $M$ of a link in a thickened closed orientable surface $\mathcal{S} \times I$ and $A$ be a normal annulus in $M$ with respect to $\mathcal{T}$.
Let $\mathcal{C}$ denote the cell-decomposition obtained from $\mathcal{T}$ by the non-destructive crushing procedure using $A$ and  
$\mathcal{T}'$ denote the triangulation obtained from $\mathcal{T}$ by the crushing procedure using $A$.
Note that the underlying space $M'$ of $\mathcal{T}'$ may not be a $3$-manifold, i.e., a point in $M'$ may not have an open neighborhood which is homeomorphic to  $\mathbb{R}^3$ or $\{(x_1, x_2, x_3)  \in \mathbb{R}^3| x_3 \geq 0\}$.
Such a point is called a \textit{singular point}, and a topological space $X$ which has singular points is called a {\it singular 3-manifold} if the topological space obtained by removing singular points from $X$ is a $3$-manifold. 
We define the boundary of a singular 3-manifold $X$ by the union of the set of singular points and the set of points which has an open neighborhood which is homeomorphic to $\{(x_1, x_2, x_3)  \in \mathbb{R}^3| x_3 \geq 0\}$, and denote it by $\partial X$.

From Lemma \ref{lem:crushing_lemma}, there is a sequence of the atomic moves $a_0 \rightarrow \cdots \rightarrow a_{n-1}$ and the cell-decompositions $\mathcal{C}_i$, where $\mathcal{C}_0 = \mathcal{C}$, $\mathcal{C}_i$ is the cell-decomposition obtained from $\mathcal{C}_{i-1}$ by the atomic move $a_{i-1}$, and $\mathcal{C}_n = \mathcal{T}'$.
There are two triangulations of $\mathcal{S}$ in the boundary of $\mathcal{T}$.
We denote these triangulations by $\mathcal{S}^0$ and $\mathcal{S}^1$.
Let $\mathcal{S}^k_0 \subset \mathcal{C}_0$ denote the union of $2$-cells obtained from the triangles in $\mathcal{S}^k$, and for any $i$ other than $0$, let $\mathcal{S}_i^k \subset \mathcal{C}_i$ denote the union of $2$-cells obtained from the $2$-cells in $\mathcal{S}_{i-1}^k$, where $k = 0, 1$.
Note that $\mathcal{S}_i^{k}$ is the subset of $\partial \mathcal{C}_i$ whose points were in $\mathcal{S}^k$ before the crushing procedure.
Hereinafter, if we write $\mathcal{S}^k_i$, then it means $\mathcal{S}^0_i$ or $\mathcal{S}^1_i$ for any integer $i$.  

\begin{lem}\label{lem:crush_annulus}
    Let $M$ be the exterior of a link in a thickened closed orientable surface $\mathcal{S} \times I$ with a triangulation $\mathcal{T}$.
    Suppose that there is a normal vertical annulus $A$ in $M$ with respect to $\mathcal{T}$.
    We denote the cell-decomposition obtained from $\mathcal{T}$ by the non-destructive crushing procedure using $A$ by $\mathcal{C}_0$ and the triangulation obtained from $\mathcal{T}$ by the crushing procedure using $A$ by $\mathcal{T}'$.
    For any $i\ (1 \leq i \leq n)$, let $M_i$ be the underlying singular $3$-manifold of $\mathcal{C}_i$.
    Then, $M_{i+1}$ is homeomorphic to $M_{i}$ or $M_{i+1}$ is obtained from $M_{i}$ by one of the following operations:
        \begin{enumerate}
            \renewcommand{\labelenumi}{(\alph{enumi})}
            \item removing a $3$-ball or a $3$-sphere component,
            \item filling a boundary sphere in $\partial M_i$ with a $3$-ball,
            \item cutting open $M_i$ along a properly embedded disk in $M_i$,
            \item cutting open $M_i$ along an embedded $2$-sphere in $M_i$ and filling the resulting boundaries with $3$-balls,
            \item cutting open $M_i$ along a bigon $B$ which satisfies the following conditions:
            \begin{itemize}
                \item $\partial B \cap \mathcal{S}_i^{0} \neq \emptyset$,
                \item $\partial B \cap \mathcal{S}_i^{1} \neq \emptyset$,
                \item $v_1$ and $v_2$ are identified with a singular point, where $v_1$ and $v_2$ are the vertices of $B$.
            \end{itemize}
        \end{enumerate}
\end{lem}

\begin{proof}
    
    The cell-decomposition $\mathcal{C}_{i}$ is obtained from $\mathcal{C}_{i-1}$ by the atomic move $a_{i-1}$.
    First, suppose that $a_{i-1}$ is the atomic move flatting a bigon $B \subset \mathcal{C}_{i-1}$ to an edge.
    Let $v_1$ and $v_2$ denote the vertices of $B$, and let $e_1$ and $e_2$ denote the edges of $B$.
    If $B$ contains no singular points, then we can prove that $a_{i-1}$ has the effect (b), (c), or (d) as the same argument of the proof of Theorem \ref{thm:B_crush}.
    Assume that $B$ contains at least one singular point.
    Note that the number of singular points in $B$ is one or two.
    We consider the following cases:
    \begin{itemize}
        \item The entire of $B$ lies in $\partial \mathcal{C}_{i-1}$:
        \begin{itemize}
            \item The edges $e_1$ and $e_2$ are not identified:
            \begin{itemize}
                \item If $v_1$ and $v_2$ are not identified, then $B$ forms a disk with one or two singular points in $\partial \mathcal{C}_{i-1}$.
                The atomic move $a_{i-1}$ is the move flatting the disk to an edge, so $a_{i-1}$ does not change the underlying singular $3$-manifold $M_{i-1}$.
                \item If $v_1$ and $v_2$ are identified, then $B$ forms a disk whose two boundary points are identified at a singular point.
                The atomic move $a_{i-1}$ is the move flatting the surface to a circle, so $a_{i-1}$ does not change the underlying singular $3$-manifold $M_{i-1}$.
            \end{itemize}
            \item The edges $e_1$ and $e_2$ are identified:
            \begin{itemize}
                \item The vertices $v_1$ and $v_2$ are not identified:
                \begin{itemize}
                    \item If either $v_1$ or $v_2$ is identified to a singular point $s$, then $B$ is a $2$-sphere boundary which touches another boundary at the point $s$ as shown in Figure \ref{fig:bigon_sphere_boundary}.
                    In this case, $a_{i-1}$ has the effect filling a $2$-sphere boundary with a $3$-ball.
                    \begin{figure}[htbp]
                        \centering
                        \tikzset{every picture/.style={line width=0.75pt}} 

\begin{tikzpicture}[x=0.75pt,y=0.75pt,yscale=-1,xscale=1]

\draw  [fill={rgb, 255:red, 0; green, 0; blue, 0 }  ,fill opacity=0.1 ] (80,49.67) .. controls (111,79) and (111,120) .. (80,137.67) .. controls (50,110) and (51,79) .. (80,49.67) -- cycle ;
\draw    (17,65) -- (147,64.8) ;
\draw    (17,65) -- (57,24.8) ;

\draw [color={rgb, 255:red, 155; green, 155; blue, 155 }  ,draw opacity=1 ]   (58,95) .. controls (58,110) and (103,109) .. (103,95) ;
\draw [color={rgb, 255:red, 155; green, 155; blue, 155 }  ,draw opacity=1 ] [dash pattern={on 4.5pt off 4.5pt}]  (58,95) .. controls (57,87) and (104,85) .. (103,95) ;
\draw [color={rgb, 255:red, 155; green, 155; blue, 155 }  ,draw opacity=1 ]   (80,137.67) .. controls (51,110) and (50,80) .. (80,49.67) ;
\draw    (80,49.67) .. controls (80,43) and (82.2,26.8) .. (97.2,26.8) ;
\draw  [fill={rgb, 255:red, 0; green, 0; blue, 0 }  ,fill opacity=1 ] (77.6,52.07) .. controls (77.6,50.74) and (78.67,49.67) .. (80,49.67) .. controls (81.33,49.67) and (82.4,50.74) .. (82.4,52.07) .. controls (82.4,53.39) and (81.33,54.47) .. (80,54.47) .. controls (78.67,54.47) and (77.6,53.39) .. (77.6,52.07) -- cycle ;
\draw    (100,81) .. controls (112.2,80.8) and (129.2,86.8) .. (129.2,110.8) ;
\draw    (209,65) -- (307,64.57) ;
\draw    (209,65) -- (248.69,24.8) ;
\draw    (268,48.9) -- (268,134.5) ;
\draw    (162,87) -- (212,86.59) ;
\draw [shift={(214,86.57)}, rotate = 539.53] [color={rgb, 255:red, 0; green, 0; blue, 0 }  ][line width=0.75]    (10.93,-3.29) .. controls (6.95,-1.4) and (3.31,-0.3) .. (0,0) .. controls (3.31,0.3) and (6.95,1.4) .. (10.93,3.29)   ;

\draw (155.61,26.5) node  [font=\normalsize] [align=left] {singulart point $\displaystyle s$};
\draw (123,113.4) node [anchor=north west][inner sep=0.75pt]    {$B$};
\draw (174,67.4) node [anchor=north west][inner sep=0.75pt]    {$a_{i-1}$};
\draw (109,45.4) node [anchor=north west][inner sep=0.75pt]    {$\partial \mathcal{C}_{i-1}$};
\draw (278,44.4) node [anchor=north west][inner sep=0.75pt]    {$\partial \mathcal{C}_{i}$};

\end{tikzpicture}
                        \caption{The case where $B$ forms a $2$-sphere in the boundary with a singular point}
                        \label{fig:bigon_sphere_boundary}
                    \end{figure}
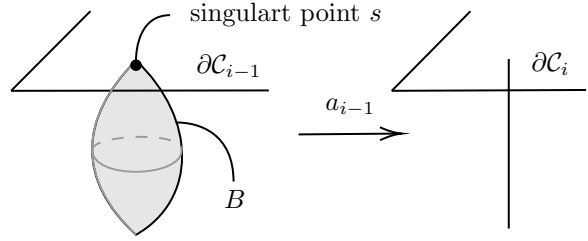
                    \item Similarly, if both of $v_1$ and $v_2$ are identified to singular points, then the atomic move $a_{i-1}$ has the effect filling a $2$-sphere boundary with a $3$-ball.
                \end{itemize}
                \item The case where $v_1$ and $v_2$ are identified does not occur.
                Singular points occur only when we shrink copies of the annulus $A$ to points.
                However, $A$ is a surface obtained by removing two open disks from a disk in this case as shown Figure \ref{fig:A_does_not_anulus}.
                This contradicts that $A$ is an annulus.
                \begin{figure}
                    \centering
                    \tikzset{every picture/.style={line width=0.75pt}} 

\begin{tikzpicture}[x=0.75pt,y=0.75pt,yscale=-1,xscale=1]

\draw   (60.98,127.24) .. controls (56.74,127.22) and (53.33,120.58) .. (53.37,112.42) .. controls (53.41,104.25) and (56.88,97.65) .. (61.12,97.67) .. controls (65.36,97.69) and (68.76,104.32) .. (68.73,112.49) .. controls (68.69,120.65) and (65.22,127.26) .. (60.98,127.24) -- cycle ;
\draw   (61.29,61.17) .. controls (57.05,61.15) and (53.64,54.52) .. (53.68,46.35) .. controls (53.72,38.18) and (57.19,31.58) .. (61.43,31.6) .. controls (65.67,31.62) and (69.08,38.26) .. (69.04,46.42) .. controls (69,54.59) and (65.53,61.19) .. (61.29,61.17) -- cycle ;
\draw [color={rgb, 255:red, 0; green, 0; blue, 0 }  ,draw opacity=0.3 ]   (61.12,97.67) .. controls (109.48,97.89) and (109.02,61.4) .. (61.29,61.17) ;
\draw [color={rgb, 255:red, 0; green, 0; blue, 0 }  ,draw opacity=0.3 ]   (60.98,127.24) .. controls (147.72,128.28) and (146.91,32) .. (61.43,31.6) ;
\draw [color={rgb, 255:red, 0; green, 0; blue, 0 }  ,draw opacity=0.3 ]   (125.92,79.72) .. controls (126.13,87.9) and (97.19,87.14) .. (96.94,79.59) ;
\draw [color={rgb, 255:red, 0; green, 0; blue, 0 }  ,draw opacity=0.3 ] [dash pattern={on 4.5pt off 4.5pt}]  (125.92,79.72) .. controls (125.57,72.8) and (96.63,73.29) .. (96.94,79.59) ;
\draw  [fill={rgb, 255:red, 0; green, 0; blue, 0 }  ,fill opacity=0.1 ] (36.94,79.3) .. controls (37.08,49.07) and (47.48,24.61) .. (60.16,24.67) .. controls (72.84,24.73) and (83.01,49.29) .. (82.87,79.52) .. controls (82.72,109.75) and (72.33,134.21) .. (59.64,134.15) .. controls (46.96,134.09) and (36.79,109.54) .. (36.94,79.3) -- cycle (60.98,127.24) .. controls (65.22,127.26) and (68.69,120.65) .. (68.73,112.49) .. controls (68.76,104.32) and (65.36,97.69) .. (61.12,97.67) .. controls (56.88,97.65) and (53.41,104.25) .. (53.37,112.42) .. controls (53.33,120.58) and (56.74,127.22) .. (60.98,127.24) -- cycle (61.29,61.17) .. controls (65.53,61.19) and (69,54.59) .. (69.04,46.42) .. controls (69.08,38.26) and (65.67,31.62) .. (61.43,31.6) .. controls (57.19,31.58) and (53.72,38.18) .. (53.68,46.35) .. controls (53.64,54.52) and (57.05,61.15) .. (61.29,61.17) -- cycle ;
\draw    (60.16,24.67) .. controls (107.5,24.1) and (135.67,25.7) .. (160,10.06) ;
\draw    (59.64,134.15) .. controls (106.98,133.58) and (140.02,135.29) .. (160,150.06) ;
\draw [color={rgb, 255:red, 0; green, 0; blue, 0 }  ,draw opacity=0.3 ]   (215.24,77.44) .. controls (223.93,82.65) and (250.87,94.3) .. (250.74,77.85) .. controls (250.61,61.4) and (223.93,73.97) .. (215.24,77.44) -- cycle ;
\draw [color={rgb, 255:red, 0; green, 0; blue, 0 }  ,draw opacity=0.3 ]   (215.24,77.44) .. controls (225.67,91) and (280.41,117.94) .. (279.72,77.99) .. controls (279.03,38.03) and (227.41,64.93) .. (215.24,77.44) -- cycle ;
\draw [color={rgb, 255:red, 0; green, 0; blue, 0 }  ,draw opacity=0.3 ]   (279.72,77.99) .. controls (279.93,86.17) and (250.99,85.4) .. (250.74,77.85) ;
\draw [color={rgb, 255:red, 0; green, 0; blue, 0 }  ,draw opacity=0.3 ] [dash pattern={on 4.5pt off 4.5pt}]  (279.72,77.99) .. controls (279.37,71.06) and (250.42,71.56) .. (250.74,77.85) ;
\draw    (215.24,77.44) .. controls (235.23,32.79) and (295.67,26.7) .. (320,11.06) ;
\draw    (215.24,77.44) .. controls (227.41,116.2) and (300.02,134.29) .. (320,149.06) ;
\draw    (24.95,61.4) .. controls (24.08,71.76) and (29.3,79.58) .. (46.67,77.84) ;
\draw    (154.42,77.91) -- (196.73,77.55) ;
\draw [shift={(198.73,77.53)}, rotate = 539.51] [color={rgb, 255:red, 0; green, 0; blue, 0 }  ][line width=0.75]    (10.93,-3.29) .. controls (6.95,-1.4) and (3.31,-0.3) .. (0,0) .. controls (3.31,0.3) and (6.95,1.4) .. (10.93,3.29)   ;
\draw    (140,19) .. controls (152,19.06) and (151,140.06) .. (140,140.06) ;
\draw  [dash pattern={on 4.5pt off 4.5pt}]  (140,19) .. controls (129,20.06) and (130,138.06) .. (140,140.06) ;
\draw    (303,18.94) .. controls (315,19) and (314,140) .. (303,140) ;
\draw  [dash pattern={on 4.5pt off 4.5pt}]  (303,18.94) .. controls (292,20) and (293,138) .. (303,140) ;

\draw (275.21,92.91) node [anchor=north west][inner sep=0.75pt]    {$B$};
\draw (17.89,45.12) node [anchor=north west][inner sep=0.75pt]    {$A$};

\end{tikzpicture}
                    \caption{The case where $B$ lies in $\partial \mathcal{C}_{i-1}$, the edges $e_1$ and $e_2$ are identified, and the vertices $v_1$ and $v_2$ are identified.}
                    \label{fig:A_does_not_anulus}
                \end{figure}
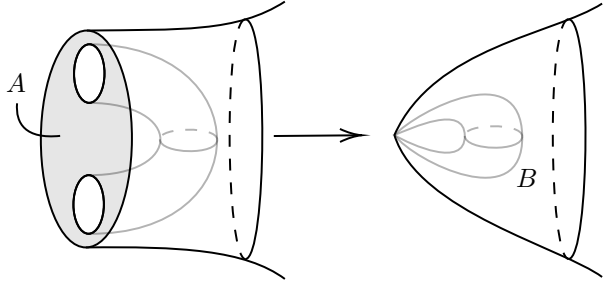
            \end{itemize}
        \end{itemize}
        
        \item Either $v_1$ or $v_2$ is in $\partial \mathcal{C}_{i-1}$ and the other points of $B$ is in the interior of $\mathcal{C}_{i-1}$:
        \begin{itemize}
            \item The edges $e_1$ and $e_2$ are not identified, then $B$ forms a disk in $\mathcal{C}_{i-1}$ with a singular point.
            The atomic move $a_{i-1}$ is the move flatting the disk to an edge, so $a_{i-1}$ does not change the underlying singular $3$-manifold $M_{i-1}$.
            \item the edges $e_1$ and $e_2$ are identified, then $B$ forms a $2$-sphere in $\mathcal{C}_{i-1}$ with a singular point.
            In this case, the atomic move $a_{i-1}$ has the effect cutting the underlying singular $3$-manifold $M_{i-1}$ along an embedded $2$-sphere and filling the resulting boundaries with $3$-balls.
        \end{itemize}
        
        \item Both $v_1$ and $v_2$ are in $\partial \mathcal{C}_{i-1}$ and the other points of $B$ is in the interior of $\mathcal{C}_{i-1}$:
        \begin{itemize}
            \item The edges $e_1$ and $e_2$ are not identified:
            \begin{itemize}
                \item The vertices $v_1$ and $v_2$ are not identified, then $B$ forms a disk with one or two singular points.
                The atomic move $a_{i-1}$ is the move flatting the disk to an edge, so $a_{i-1}$ does not change the underlying singular $3$-manifold $M_{i-1}$.
                \item The vertices $v_1$ and $v_2$ are identified, then $B$ forms a disk whose two boundary points are identified at a singular point.
                The atomic move $a_{i-1}$ is the move flatting the surface to a circle, so $a_{i-1}$ does not change the underlying singular $3$-manifold $M_{i-1}$.
            \end{itemize}
            \item Both $e_1$ and $e_2$ are identified:
            \begin{itemize}
                \item The vertices $v_1$ and $v_2$ are not identified, then $B$ forms a $2$-sphere with one or two  singular points.
                The atomic move $a_{i-1}$ is the move flatting the $2$-sphere to an edge, so $a_{i-1}$ has the effect cutting the underlying singular $3$-manifold $M_{i-1}$ along an embedded $2$-sphere and filling the resulting boundaries with $3$-balls.
                \item The vertices $v_1$ and $v_2$ are identified, then $B$ forms a $2$-sphere whose two boundary points are identified at a singular point.
                The atomic move $a_{i-1}$ is the move flatting the surface to a circle, so $a_{i-1}$ has the effect cutting the underlying singular $3$-manifold $M_{i-1}$ along an embedded $2$-sphere and filling the resulting boundaries with $3$-balls.
            \end{itemize}
        \end{itemize}
        
        \item Either $e_1$ or $e_2$ is in $\partial \mathcal{C}_{i-1}$ and the other points are in the interior of $\mathcal{C}_{i-1}$:
        \begin{itemize}
            \item The vertices $v_1$ and $v_2$ are not identified, then $B$ forms a disk with one or two singular points.
            Note that $e_1$ and $e_2$ are not identified from the assumption. 
            The atomic move $a_{i-1}$ is the move flatting the disk to an edge, so $a_{i-1}$ does not change the underlying singular $3$-manifold $M_{i-1}$.
            \item The vertices $v_1$ and $v_2$ are identified, then $B$ forms a disk whose two boundary points are identified at a singular point.
            The atomic move $a_{i-1}$ is the move flatting the surface to a circle, so $a_{i-1}$ does not change the underlying singular $3$-manifold $M_{i-1}$.
        \end{itemize}
        
        \item Both the edges $e_1$ and $e_2$ are in $\partial \mathcal{C}_{i-1}$ and the other points are in the interior of $\mathcal{C}_{i-1}$:
        \begin{itemize}
            \item The edges $e_1$ and $e_2$ are not identified:
            \begin{itemize}
                \item The vertices $v_1$ and $v_2$ are not identified:
                    \begin{itemize} 
                        \item Either $v_1$ or $v_2$ is identified to a singular point, then $B$ forms a properly embedded disk with a singular point.
                        The atomic move $a_{i-1}$ is the move flatting the disk to an edge, so $a_{i-1}$ has the effect cutting open the underlying singular $3$-manifold $M_{i-1}$ along a properly embedded disk.
                        \item The vertices $v_1$ and $v_2$ are identified to singular points, $e_1 \subset \mathcal{S}_{i-1}^{0}$, and $e_2 \subset \mathcal{S}_{i-1}^{1}$, then $B$ forms a properly embedded disk with the two singular points, and $a_{i-1}$ is the move flatting the disk to an edge.
                        Note that the edge obtained from the disk is divided into two edges because the $3$-cells containing the edge are connected only by the edge as shown in Figure \ref{fig:bigon_to_circle}.
                        Therefore, $a_{i-1}$ has the effect cutting the underlying singular $3$-manifold $M_{i-1}$ along a properly embedded disk.
                        In the case where $e_1 \subset \mathcal{S}_{i-1}^{1}$ and $e_2 \subset \mathcal{S}_{i-1}^{0}$, we can obtain the same results.
                        \begin{figure}
                            \centering
                            \input{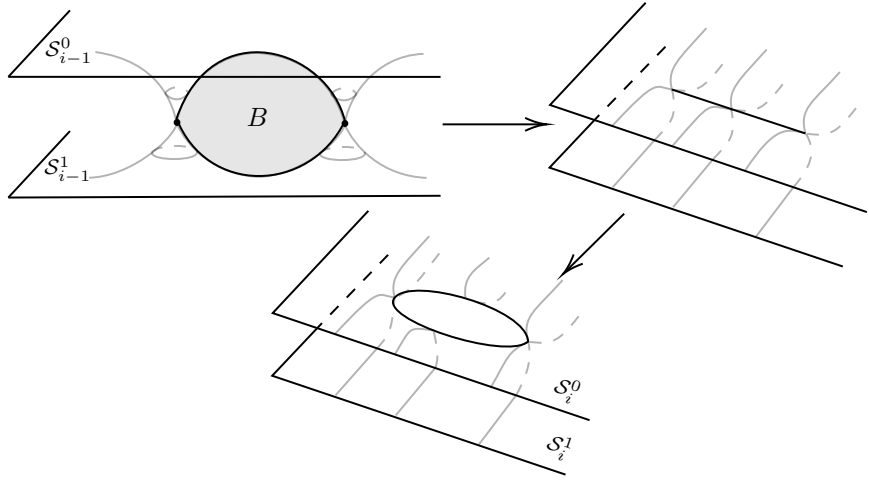}
                            \caption{The case where $v_1$ and $v_2$ are not identified and $e_1 \subset \mathcal{S}_{i-1}^{0}$, and $e_2 \subset \mathcal{S}_{i-1}^{1}$}
                            \label{fig:bigon_to_circle}
                        \end{figure}
                        \item The vertices $v_1$ and $v_2$ are identified to singular points and $e_1$ and $e_2$ are in $\mathcal{S}_{i-1}^{k}$, then $B$ forms a properly embedded disk, and so $a_{i-1}$ has the effect cutting the underlying singular $3$-manifold $M_{i-1}$ along a properly embedded disk.
                    \end{itemize}
                \item The vertices $v_1$ and $v_2$ are identified:
                \begin{itemize}
                    \item The edge $e_1$ lies in $\mathcal{S}_{i-1}^{0}$ and the edge $e_2$ lies in $\mathcal{S}_{i-1}^{1}$, then $B$ forms a disk whose two boundary points are identified at a singular point as shown in Figure \ref{fig:bigon_to_circle2}, and so $a_{i-1}$ is the move flatting the surface to a circle.
                    Therefore, $a_{i-1}$ has the effect cutting the underlying singular $3$-manifold $M_{i-1}$ along a bigon whose vertices are identified.
                    In the case where $e_1 \subset \mathcal{S}_{i-1}^{1}$ and $e_2 \subset \mathcal{S}_{i-1}^{0}$, we can obtain the same results.
                    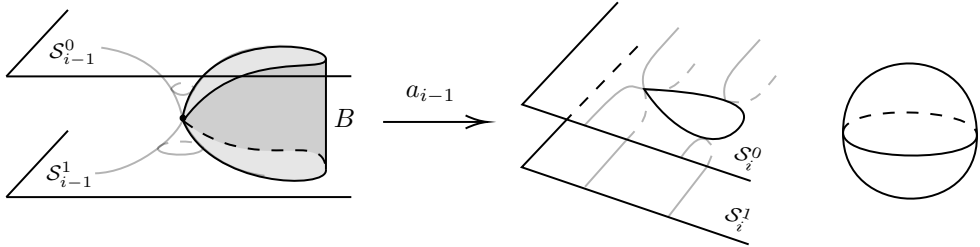
\begin{figure}
                        \centering
                        \tikzset{every picture/.style={line width=0.75pt}} 

\begin{tikzpicture}[x=0.75pt,y=0.75pt,yscale=-1,xscale=1]

\draw    (10,41.43) -- (184,41.33) ;
\draw    (10,41.43) -- (41.09,7.88) ;
\draw    (10,102.35) -- (184,102.33) ;
\draw    (10,102.35) -- (41.09,68.8) ;
\draw [color={rgb, 255:red, 0; green, 0; blue, 0 }  ,draw opacity=0.3 ]   (57.89,26.91) .. controls (89.76,31.5) and (96.75,52.36) .. (99.09,63.21) ;
\draw [color={rgb, 255:red, 0; green, 0; blue, 0 }  ,draw opacity=0.3 ]   (54.79,90.34) .. controls (74.22,89.92) and (95.2,73.23) .. (99.09,63.21) ;
\draw [color={rgb, 255:red, 0; green, 0; blue, 0 }  ,draw opacity=0.3 ]   (140.2,91.6) .. controls (117.2,88.6) and (103.75,74.06) .. (99.09,63.21) ;
\draw [color={rgb, 255:red, 0; green, 0; blue, 0 }  ,draw opacity=0.3 ]   (138.72,27.33) .. controls (113.08,28.16) and (99.86,51.53) .. (99.09,63.21) ;
\draw [color={rgb, 255:red, 0; green, 0; blue, 0 }  ,draw opacity=0.3 ]   (86.4,77.58) .. controls (85.92,82.13) and (110.53,82.13) .. (109.11,77.07) ;
\draw [color={rgb, 255:red, 0; green, 0; blue, 0 }  ,draw opacity=0.3 ] [dash pattern={on 4.5pt off 4.5pt}]  (86.4,77.58) .. controls (89.71,73.03) and (107.69,73.53) .. (109.11,77.07) ;
\draw [color={rgb, 255:red, 0; green, 0; blue, 0 }  ,draw opacity=0.3 ]   (93.02,47.73) .. controls (95.86,51.27) and (101.06,52.28) .. (104.85,47.73) ;
\draw [color={rgb, 255:red, 0; green, 0; blue, 0 }  ,draw opacity=0.3 ] [dash pattern={on 4.5pt off 4.5pt}]  (93.02,47.73) .. controls (96.33,43.17) and (103.43,44.19) .. (104.85,47.73) ;
\draw  [fill={rgb, 255:red, 0; green, 0; blue, 0 }  ,fill opacity=1 ] (97.9,62.45) .. controls (97.9,61.76) and (98.43,61.19) .. (99.09,61.19) .. controls (99.74,61.19) and (100.27,61.76) .. (100.27,62.45) .. controls (100.27,63.15) and (99.74,63.72) .. (99.09,63.72) .. controls (98.43,63.72) and (97.9,63.15) .. (97.9,62.45) -- cycle ;
\draw    (200.04,64.41) -- (250.56,64.41) ;
\draw [shift={(252.56,64.41)}, rotate = 180] [color={rgb, 255:red, 0; green, 0; blue, 0 }  ][line width=0.75]    (10.93,-3.29) .. controls (6.95,-1.4) and (3.31,-0.3) .. (0,0) .. controls (3.31,0.3) and (6.95,1.4) .. (10.93,3.29)   ;
\draw [color={rgb, 255:red, 0; green, 0; blue, 0 }  ,draw opacity=1 ]   (99.09,61.19) .. controls (108.2,16.6) and (171.2,24.6) .. (171.2,32.6) .. controls (171.2,40.6) and (127.2,26.6) .. (100.27,62.45) ;
\draw [color={rgb, 255:red, 0; green, 0; blue, 0 }  ,draw opacity=1 ]   (171.2,32.6) -- (171.2,88.6) ;
\draw    (99.09,63.72) .. controls (113.2,97.6) and (172.2,98.6) .. (171.2,88.6) ;
\draw  [dash pattern={on 4.5pt off 4.5pt}]  (99.09,63.72) .. controls (136.2,92.6) and (170.2,66.6) .. (171.2,88.6) ;
\draw  [draw opacity=0][fill={rgb, 255:red, 0; green, 0; blue, 0 }  ,fill opacity=0.1 ] (99.09,63.72) .. controls (134,93) and (169,67) .. (171.2,88.6) .. controls (171,49) and (171,66) .. (171.2,32.6) .. controls (173,38) and (127,28) .. (99.09,63.72) -- cycle ;
\draw  [draw opacity=0][fill={rgb, 255:red, 0; green, 0; blue, 0 }  ,fill opacity=0.1 ] (99.09,63.72) .. controls (112,95) and (170,100) .. (171.2,88.6) .. controls (171,49) and (171,66) .. (171.2,32.6) .. controls (172,25) and (109,12) .. (99.09,63.72) -- cycle ;
\draw    (318.45,4.02) -- (270.2,55.63) -- (385.4,94.2) ;
\draw    (295.03,60.7) -- (269.95,87.26) -- (384.4,126.2) ;
\draw  [dash pattern={on 4.5pt off 4.5pt}]  (328.41,26.21) -- (295.03,60.7) ;
\draw [color={rgb, 255:red, 0; green, 0; blue, 0 }  ,draw opacity=0.3 ]   (350.1,82.66) .. controls (359.4,69) and (362,74) .. (363,73) ;
\draw [color={rgb, 255:red, 0; green, 0; blue, 0 }  ,draw opacity=0.3 ]   (301.4,66.4) .. controls (324.8,41.12) and (323.45,46.52) .. (331.71,47.84) ;
\draw [color={rgb, 255:red, 0; green, 0; blue, 0 }  ,draw opacity=0.3 ]   (349.1,16.83) .. controls (335.21,32.34) and (328.85,37.08) .. (331.71,47.84) ;
\draw [color={rgb, 255:red, 0; green, 0; blue, 0 }  ,draw opacity=0.3 ] [dash pattern={on 4.5pt off 4.5pt}]  (322.78,73.52) .. controls (330.88,64.74) and (334.93,62.05) .. (331.71,47.84) ;
\draw [color={rgb, 255:red, 0; green, 0; blue, 0 }  ,draw opacity=0.3 ] [dash pattern={on 4.5pt off 4.5pt}]  (357.2,35.05) .. controls (347.75,47.2) and (344.37,49.9) .. (331.71,47.84) ;
\draw [color={rgb, 255:red, 0; green, 0; blue, 0 }  ,draw opacity=0.3 ]   (322.78,73.52) -- (301.86,97.14) ;
\draw [color={rgb, 255:red, 0; green, 0; blue, 0 }  ,draw opacity=0.3 ] [dash pattern={on 4.5pt off 4.5pt}]  (363.4,87) .. controls (368.4,75) and (367.4,79) .. (363,73) ;
\draw [color={rgb, 255:red, 0; green, 0; blue, 0 }  ,draw opacity=0.3 ]   (363.4,87) -- (343.91,112.21) ;
\draw   (331.71,47.84) .. controls (344,67) and (366,81) .. (379,68) .. controls (393,53) and (356,49) .. (331.71,47.84) -- cycle ;
\draw [color={rgb, 255:red, 0; green, 0; blue, 0 }  ,draw opacity=0.3 ]   (390.1,22.83) .. controls (376.21,38.34) and (369.85,43.08) .. (372.71,53.84) ;
\draw [color={rgb, 255:red, 0; green, 0; blue, 0 }  ,draw opacity=0.3 ] [dash pattern={on 4.5pt off 4.5pt}]  (398.2,41.05) .. controls (388.75,53.2) and (385.37,55.9) .. (372.71,53.84) ;
\draw    (432,70) .. controls (432.5,85) and (500.5,85) .. (499.5,71) ;
\draw  [dash pattern={on 4.5pt off 4.5pt}]  (432,70) .. controls (431.5,55) and (499.5,57) .. (499.5,71) ;
\draw    (432,70) .. controls (431.5,23) and (499.5,23) .. (499.5,71) ;
\draw    (432,70) .. controls (431.5,114) and (500.5,114) .. (499.5,71) ;

\draw (173.75,56.54) node [anchor=north west][inner sep=0.75pt]    {$B$};
\draw (42.05,91.5) node  [font=\footnotesize]  {$\mathcal{S}_{i-1}^{1}$};
\draw (43.63,29.71) node  [font=\footnotesize]  {$\mathcal{S}_{i-1}^{0}$};
\draw (379.18,113) node  [font=\footnotesize,rotate=-18.86]  {$\mathcal{S}_{i}^{1}$};
\draw (383.92,82.14) node  [font=\footnotesize,rotate=-19.27]  {$\mathcal{S}_{i}^{0}$};
\draw (210,44.4) node [anchor=north west][inner sep=0.75pt]    {$a_{i-1}$};

\end{tikzpicture}
                        \caption{The case where $v_1$ and $v_2$ are identified and $e_1 \subset \mathcal{S}_{i-1}^{0}$ and $e_2 \subset \mathcal{S}_{i-1}^{1}$}
                        \label{fig:bigon_to_circle2}
                    \end{figure}
                    \item The edges $e_1$ and $e_2$ are in $\mathcal{S}_{i-1}^{k}$, then $B$ forms a disk whose two boundary points are identified at a singular point as shown in Figure \ref{fig:bigon_to_circle3}, and so $a_{i-1}$ is the move flatting the surface to a circle.
                    \begin{figure}
                        \centering
                        \tikzset{every picture/.style={line width=0.75pt}} 

\begin{tikzpicture}[x=0.75pt,y=0.75pt,yscale=-1,xscale=1]

\draw   (203.25,90.68) -- (2,90.68) -- (86.94,30) ;
\draw   (127.39,28.97) .. controls (201.12,62.88) and (53.66,69.76) .. (127.39,28.97) .. controls (-98.72,98.27) and (342.69,103.19) .. (127.39,28.97) -- cycle ;
\draw [color={rgb, 255:red, 0; green, 0; blue, 0 }  ,draw opacity=0.3 ]   (85.56,78.69) .. controls (96.75,88.57) and (118.49,108.33) .. (116.51,55.65) ;
\draw [color={rgb, 255:red, 0; green, 0; blue, 0 }  ,draw opacity=0.3 ]   (158.66,80.01) .. controls (155.36,84.62) and (133.63,113.6) .. (136.27,56.3) ;
\draw  [draw opacity=0][fill={rgb, 255:red, 0; green, 0; blue, 0 }  ,fill opacity=0.1 ] (188.29,66.84) .. controls (162.61,109.64) and (75.68,92.52) .. (63.17,69.47) .. controls (50,49.72) and (114.22,35.55) .. (127.39,28.97) .. controls (145.48,34.57) and (195.53,55.65) .. (188.29,66.84) -- cycle (127.39,28.97) .. controls (114.22,35.55) and (86.22,54.33) .. (126.39,56.96) .. controls (163.92,54.33) and (145.48,34.57) .. (127.39,28.97) -- cycle ;
\draw  [draw opacity=0][fill={rgb, 255:red, 0; green, 0; blue, 0 }  ,fill opacity=0.1 ] (188.29,66.84) .. controls (161.29,108.99) and (75.68,92.52) .. (63.17,69.47) .. controls (96.75,93.25) and (190.92,79.35) .. (188.29,66.84) -- cycle ;
\draw [color={rgb, 255:red, 0; green, 0; blue, 0 }  ,draw opacity=0.3 ]   (72.39,51.46) .. controls (69.55,67.24) and (99.27,69.74) .. (105.97,50.8) ;
\draw [color={rgb, 255:red, 0; green, 0; blue, 0 }  ,draw opacity=0.3 ]   (147.46,50.8) .. controls (151.87,69.88) and (175.12,67.33) .. (177.75,52.12) ;
\draw    (214,59) -- (281,59) ;
\draw [shift={(283,59)}, rotate = 180] [color={rgb, 255:red, 0; green, 0; blue, 0 }  ][line width=0.75]    (10.93,-3.29) .. controls (6.95,-1.4) and (3.31,-0.3) .. (0,0) .. controls (3.31,0.3) and (6.95,1.4) .. (10.93,3.29)   ;
\draw   (390.4,85.67) -- (272.87,85.67) -- (331.2,44) ;
\draw    (410.3,73.56) -- (519,73.56) ;
\draw    (410.3,73.56) .. controls (433.54,32.58) and (504.76,38.14) .. (519,73.56) ;
\draw    (464.65,73.56) .. controls (454.53,68.7) and (456.03,46.47) .. (465.03,45.08) ;
\draw [color={rgb, 255:red, 0; green, 0; blue, 0 }  ,draw opacity=0.3 ] [dash pattern={on 4.5pt off 4.5pt}]  (465.03,45.08) .. controls (471.02,44.39) and (474.77,68.7) .. (464.65,73.56) ;
\draw    (435.55,73.56) .. controls (425.96,68.48) and (426.79,54.11) .. (432.04,52.72) ;
\draw [color={rgb, 255:red, 0; green, 0; blue, 0 }  ,draw opacity=0.3 ] [dash pattern={on 4.5pt off 4.5pt}]  (432.04,52.72) .. controls (440.27,52.21) and (441.04,68) .. (435.55,73.56) ;
\draw    (496.51,73.56) .. controls (490.51,70.78) and (489.76,52.03) .. (497.26,52.03) ;
\draw [color={rgb, 255:red, 0; green, 0; blue, 0 }  ,draw opacity=0.3 ] [dash pattern={on 4.5pt off 4.5pt}]  (497.26,52.03) .. controls (503.26,51.33) and (504.01,70.08) .. (496.51,73.56) ;
\draw  [fill={rgb, 255:red, 0; green, 0; blue, 0 }  ,fill opacity=1 ] (124.99,28.97) .. controls (124.99,27.64) and (126.07,26.57) .. (127.39,26.57) .. controls (128.72,26.57) and (129.79,27.64) .. (129.79,28.97) .. controls (129.79,30.29) and (128.72,31.37) .. (127.39,31.37) .. controls (126.07,31.37) and (124.99,30.29) .. (124.99,28.97) -- cycle ;
\draw   (362,48.8) .. controls (377,62.8) and (382,75.8) .. (355,76.8) .. controls (320,75.8) and (341,57.8) .. (362,48.8) -- cycle ;

\draw (162,93.4) node [anchor=north west][inner sep=0.75pt]    {$\partial \mathcal{C}_{i-1}$};
\draw (239,36.4) node [anchor=north west][inner sep=0.75pt]    {$a_{i-1}$};
\draw (80,5) node [anchor=north west][inner sep=0.75pt]   [align=left] {singular point};
\draw (344,89.4) node [anchor=north west][inner sep=0.75pt]    {$\partial \mathcal{C}_{i}$};

\end{tikzpicture}
                        \caption{The case where $v_1$ and $v_2$ are identified and $e_1$ and $e_2$ are in $\mathcal{S}_{i-1}^{k}$}
                        \label{fig:bigon_to_circle3}
                    \end{figure}
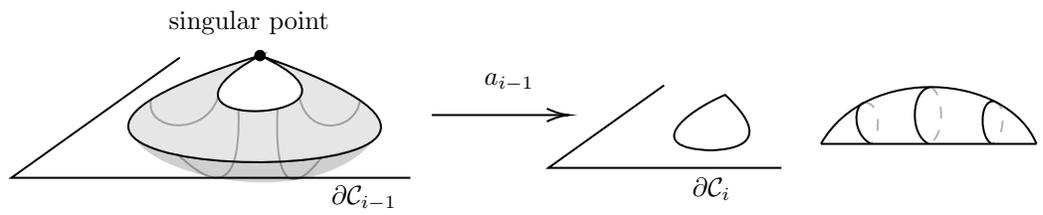
                \end{itemize}
            \end{itemize}
        \end{itemize}
        \item The edges $e_1$ and $e_2$ are identified:
        \begin{itemize}
            \item The vertices $v_1$ and $v_2$ are not identified:
            \begin{itemize}
                \item Either $v_1$ or $v_2$ is identified to a singular points, then $B$ forms a $2$-sphere with a singular point.
                In this case, the atomic move $a_{i-1}$ is the move flatting the $2$-sphere to an edge, and so $a_{i-1}$ has the effect cutting the underlying singular $3$-manifold $M_{i-1}$ along an embedded $2$-sphere and filling the resulting boundaries with $3$-balls.
                \item The vertices $v_1$ and $v_2$ are identified to singular points, then $B$ forms a $2$-sphere with two singular points.
                In this case, the atomic move $a_{i-1}$ is the move flatting the $2$-sphere to an edge, and so $a_{i-1}$ has the effect cutting the underlying singular $3$-manifold $M_{i-1}$ along an embedded $2$-sphere and filling the resulting boundaries with $3$-balls.
            \end{itemize}
            \item The vertices $v_1$ and $v_2$ are identified, then $B$ forms a $2$-sphere whose two points are identified at a singular point, and so the atomic move $a_{i-1}$ is the move flatting the surface to a circle.
            Therefore, $a_{i-1}$ has the effect cutting the underlying singular $3$-manifold $M_{i-1}$ along an embedded $2$-sphere and filling the resulting boundaries with $3$-balls.
        \end{itemize}
    \end{itemize}
    
    Next, we assume that $a_{i-1}$ flats a triangular pillow to a triangle.
    Let $P$ denote the triangular pillow flattened by $a_{i-1}$, and $t_1$ and $t_2$ denote the triangles of $P$.
    \begin{itemize}
        \item If $t_1$ and $t_2$ are in $\partial \mathcal{C}_{i-1}$, then $P$ is a $3$-ball component.
        Thus, $a_{i-1}$ has the effect removing a $3$-ball component from $M_{i-1}$.
        \item  We consider the case where $t_1$ and $t_2$ are identified.
        If $t_1$ and $t_2$ are identified without a twist, then $P$ is a $3$-sphere component.
        If $t_1$ and $t_2$ are identified with a twist, then $P$ is a lens space $L(3,1)$ component.
        A $3$-sphere component and a $L(3,1)$ component in $M_i$ occur by the operation (d) because (d) is the only  operation which may add a component with no boundaries.
        If there is a $L(3,1)$ component, then there is an embedded $2$-sphere bounding a $3$-manifold cl$(L(3,1) - \mathbb{B}^3)$ in the canonical exterior cl$(\mathcal{S} \times I - N(\hat{D}))$.
        This contradicts that $\mathcal{S} \times I$ is irreducible, where $\mathcal{S}$ is the supporting surface of the virtual link diagram $D$.
        Therefore, $P$ forms a $3$-sphere, and so $a_{i-1}$ has the effect removing a $3$-sphere component.
        \item In the case where at least one of the faces of $P$ identifies to a face of cells other than $P$, then $P$ is an embedded $3$-ball in $\mathcal{C}_{i-1}$.
        The atomic move $a_{i-1}$ is the move flatting this $3$-ball to a disk, hence $a_{i-1}$ does not change the topology of the underlying singular $3$-manifold $M_{i-1}$.
    \end{itemize}
    
    Finally, we assume that $a_{i-1}$ is the move flatting a bigonal pillow to a bigon.
    Let $P$ denote the bigonal pillow flattened by $a_{i-1}$, and $b_1$ and $b_2$ denote the bigons of $P$.
    \begin{itemize}
        \item If $b_1$ and $b_2$ are in $\partial \mathcal{C}_{i-1}$, then $P$ is a $3$-ball component.
        Thus, $a_{i-1}$ has the effect removing a $3$-ball component.
        \item 
        If $t_1$ and $t_2$ are identified without a twist, then $P$ is a $3$-sphere component.
        If $t_1$ and $t_2$ are identified with a twist, then $P$ is a projection space component.
        In this case, $P$ does not forms a projection space for the same reason that a $L(3,1)$ component does not exist in $M_i$, and so $P$ forms a $3$-sphere.
        Therefore, $a_{i-1}$ has the effect removing a $3$-sphere component.
        \item In the case where at least one of the faces of $P$ identifies to a face of cells ther than $P$, $a_{i-1}$ does not change the underlying $3$-manifold because $P$ is an embedded $3$-ball in $\mathcal{C}_i$.
    \end{itemize}
\end{proof}

\begin{defi}\label{def:VandE}
    Let $M$ be the exterior of a link in a thickened closed orientable surface $\mathcal{S} \times I$ and $\mathcal{T}$ be a triangulation of $M$.
    There are two triangulations of $\mathcal{S}$ in the boundary of $\mathcal{T}$, and we denote these by $\mathcal{S}^0$ and $\mathcal{S}^1$.
    Suppose that there is a normal vertical annulus $A$ in $\mathcal{T}$.
    Let $\mathcal{C}_0$ be the cell-decomposition obtained from $\mathcal{T}$ by the non-destructive crushing procedure using $A$ and $\mathcal{T}$ be the triangulation obtained from $\mathcal{T}'$ by the crushing procedure using $A$.
    From Lemma \ref{lem:crushing_lemma}, there is a sequence of cell-decompositions $\mathcal{C}_0 \to \mathcal{C}_1 \to \cdots \to \mathcal{C}_n = \mathcal{T}'$, and $\mathcal{C}_i$ is obtained from $\mathcal{C}_{i-1}$ by an atomic move if $1 \leq i \leq n$.
    Let $\mathcal{S}_i^{k}$ denote the union of the subset obtained from $S^k$ in $\partial \mathcal{C}_i$.
    Then, we define the subsets $V_i$ and $E_i$ of $\partial \mathcal{C}_i$ as follows:
    \begin{itemize}
        \item $\displaystyle E_i = \bigcup_{e_j\ \subset\ \mathcal{S}_i^{0}\ \cap\ \mathcal{S}_i^{1}} e_j$, where $e_j$ are edges of $\mathcal{C}_i$. 
        \item $\displaystyle V_i = \left( \bigcup_{v_j\ \in \ \mathcal{S}_i^{0}\ \cap\ \mathcal{S}_i^{1}} v_j \right) - \left( \bigcup_{v_j\ \in\ E_i} v_j \right)$, where $v_j$ are verticies of $\mathcal{C}_i$.
    \end{itemize}
\end{defi}
Note that the $2$-cells in $\mathcal{S}_i^{0} \cup \mathcal{S}_i^{1}$ are removed from $\mathcal{C}_i$ because the $2$-cells does not belong to any $3$-cells in $\mathcal{C}_i$, thus, there are no $2$-cells in $\mathcal{S}_i^{0} \cap \mathcal{S}_i^{1}$.
Therefore, $V_i$ is the set of all singular points in $\mathcal{C}_i$, and $\mathcal{S}_i^{0}\cap \mathcal{S}_i^{1} = V_i \cup E_i$. 
We define $V' \subset \mathcal{T}'$ by $V_n$ and $E'\subset \mathcal{T}'$ by $E_n$.

\begin{lem}\label{lem:V'_E'}
    Suppose the same situation of Definition \ref{def:VandE}.
    Then, $V'$ and $E'$ satisfies one of the following conditions:
    \begin{itemize}
        \item $|V'| = 2$ and $|E'| = 0$,
        \item $|V'| = 1$ and $|E'| = 0$,
        \item $|V'| = 1$, $|E'| = 2$, and each component of $E'$ is a circle constructed by one edge,
        \item $|V'| = 0$, $|E'| = 2$, and each component of $E'$ is a circle constructed by one edge,
        \item $|V'| = 0$, $|E'| = 4$, and each component of $E'$ is a circle constructed by one edge,
        \item $|V'| = 0$, $|E'| = 1$, and the component of $E'$ is a circle constructed by two edges,
        \item $|V'| = 0$ and $|E'| = 0$.
    \end{itemize}
\end{lem}
\begin{proof}
    $\mathcal{C}_0$ is obtained from $\mathcal{T}$ by cutting $\mathcal{T}$ along the annulus $A$ and then each copies of $A$ to a point.
    Therefore, $|V_0| = 2$ and $|E_0| = 0$.
    
    $\mathcal{C}_{i+1}$ is obtained from $\mathcal{C}_i$ by an atomic move $a_i$.
    From the proof of Lemma \ref{lem:crush_annulus}, if $|V_i| \neq |V_{i+1}|$ or $|E_i| \neq |E_{i+1}|$, then the atomic move $a_i$ is the move flatting a bigon $B$ to an edge and $B$ contains singular points.
    Furthermore, by the proof of Lemma \ref{lem:crush_annulus}, $B$ satisfies one of the following conditions:
    \begin{enumerate}
        \renewcommand{\labelenumi}{(\Roman{enumi})}
        \item $B \subset \partial \mathcal{C}_i$, $e_1 = e_2$, and $B$ contains one singular point (Figure \ref{fig:types}\subref{fig:type1}),
        \item $B \subset \partial \mathcal{C}_i$, $e_1 = e_2$, and $B$ contains two singular points (Figure \ref{fig:types}\subref{fig:type2}), 
        \item $\partial B \subset \partial \mathcal{C}_i$, $v_1 \neq v_2$, and $e_1 \subset \mathcal{S}_i^k$ and $e_2 \subset \mathcal{S}_i^l$ (Figure \ref{fig:types}\subref{fig:type3}),
        \item  $\partial B \subset \partial \mathcal{C}_i$, $v_1 = v_2$, and $e_1 \subset \mathcal{S}_i^k$ and $e_2 \subset \mathcal{S}_i^l$ (Figure \ref{fig:types}\subref{fig:type4}),
    \end{enumerate}
    where $e_1$ and $e_2$ are the edges of $B$ and $k \neq l$.
    We say that a bigon $B$ has type (I), (II), (III), or (IV) if $B$ has singular points and $B$ satisfies the condition (I), (II), (III), or (IV), respectively. 

    \begin{figure}
        \begin{tabular}{cc}
            \begin{minipage}{0.49\textwidth}
                \centering
                \tikzset{every picture/.style={line width=0.75pt}} 

\begin{tikzpicture}[x=0.75pt,y=0.75pt,yscale=-1,xscale=1]

\draw  [fill={rgb, 255:red, 0; green, 0; blue, 0 }  ,fill opacity=0.1 ] (74,29.27) .. controls (105,52.24) and (108,84.74) .. (75,85.41) .. controls (44,82.74) and (45,52.24) .. (74,29.27) -- cycle ;
\draw    (11,41.28) -- (141,41.12) ;
\draw    (11,41.28) -- (51,9.8) ;

\draw [color={rgb, 255:red, 155; green, 155; blue, 155 }  ,draw opacity=1 ]   (52,64.77) .. controls (52,76.52) and (97,75.74) .. (97,64.77) ;
\draw [color={rgb, 255:red, 155; green, 155; blue, 155 }  ,draw opacity=1 ] [dash pattern={on 4.5pt off 4.5pt}]  (52,64.77) .. controls (51,58.51) and (98,56.94) .. (97,64.77) ;
\draw [color={rgb, 255:red, 155; green, 155; blue, 155 }  ,draw opacity=1 ]   (75,85.41) .. controls (45,82.41) and (44,53.03) .. (74,29.27) ;
\draw    (11,106.28) -- (141,106.12) ;
\draw    (11,106.28) -- (51,74.8) ;

\draw  [fill={rgb, 255:red, 0; green, 0; blue, 0 }  ,fill opacity=1 ] (72.6,29.38) .. controls (72.6,28.86) and (73.14,28.44) .. (73.8,28.44) .. controls (74.46,28.44) and (75,28.86) .. (75,29.38) .. controls (75,29.9) and (74.46,30.32) .. (73.8,30.32) .. controls (73.14,30.32) and (72.6,29.9) .. (72.6,29.38) -- cycle ;

\draw (99,65.79) node [anchor=north west][inner sep=0.75pt]    {$B$};

\end{tikzpicture}
                \subcaption{type (I)}
                \label{fig:type1}
            \end{minipage} &
            
            \begin{minipage}{0.49\textwidth}
                \centering
                \tikzset{every picture/.style={line width=0.75pt}} 

\begin{tikzpicture}[x=0.75pt,y=0.75pt,yscale=-1,xscale=1]

\draw  [fill={rgb, 255:red, 0; green, 0; blue, 0 }  ,fill opacity=0.1 ] (74,29.27) .. controls (105,52.24) and (105,84.35) .. (74,98.18) .. controls (44,76.52) and (45,52.24) .. (74,29.27) -- cycle ;
\draw    (11,41.28) -- (141,41.12) ;
\draw    (11,41.28) -- (51,9.8) ;

\draw [color={rgb, 255:red, 155; green, 155; blue, 155 }  ,draw opacity=1 ]   (52,64.77) .. controls (52,76.52) and (97,75.74) .. (97,64.77) ;
\draw [color={rgb, 255:red, 155; green, 155; blue, 155 }  ,draw opacity=1 ] [dash pattern={on 4.5pt off 4.5pt}]  (52,64.77) .. controls (51,58.51) and (98,56.94) .. (97,64.77) ;
\draw [color={rgb, 255:red, 155; green, 155; blue, 155 }  ,draw opacity=1 ]   (74,98.18) .. controls (45,76.52) and (44,53.03) .. (74,29.27) ;
\draw    (11,106.28) -- (141,106.12) ;
\draw    (11,106.28) -- (51,74.8) ;

\draw  [fill={rgb, 255:red, 0; green, 0; blue, 0 }  ,fill opacity=1 ] (73.6,98.03) .. controls (73.6,97.51) and (74.14,97.09) .. (74.8,97.09) .. controls (75.46,97.09) and (76,97.51) .. (76,98.03) .. controls (76,98.55) and (75.46,98.97) .. (74.8,98.97) .. controls (74.14,98.97) and (73.6,98.55) .. (73.6,98.03) -- cycle ;
\draw  [fill={rgb, 255:red, 0; green, 0; blue, 0 }  ,fill opacity=1 ] (72.6,29.38) .. controls (72.6,28.86) and (73.14,28.44) .. (73.8,28.44) .. controls (74.46,28.44) and (75,28.86) .. (75,29.38) .. controls (75,29.9) and (74.46,30.32) .. (73.8,30.32) .. controls (73.14,30.32) and (72.6,29.9) .. (72.6,29.38) -- cycle ;

\draw (99,65.79) node [anchor=north west][inner sep=0.75pt]    {$B$};

\end{tikzpicture}
                \subcaption{type (II)}
                \label{fig:type2}
            \end{minipage}\\
            
            \begin{minipage}{0.49\textwidth}
                \centering
                \tikzset{every picture/.style={line width=0.75pt}} 

\begin{tikzpicture}[x=0.75pt,y=0.75pt,yscale=-1,xscale=1]

\draw    (6,38.45) -- (174.57,38.77) ;
\draw    (6,38.45) -- (31.53,4) ;
\draw    (6,101) -- (175.23,100.17) ;
\draw    (6,101) -- (31.53,66.55) ;
\draw [color={rgb, 255:red, 0; green, 0; blue, 0 }  ,draw opacity=0.3 ]   (28.91,25.59) .. controls (55.08,30.31) and (60.83,51.73) .. (62.74,62.87) ;
\draw [color={rgb, 255:red, 0; green, 0; blue, 0 }  ,draw opacity=0.3 ]   (26.36,90.72) .. controls (42.32,90.29) and (59.55,73.15) .. (62.74,62.87) ;
\draw [color={rgb, 255:red, 0; green, 0; blue, 0 }  ,draw opacity=0.3 ]   (96.58,90.29) .. controls (78.07,89.43) and (66.57,74.01) .. (62.74,62.87) ;
\draw [color={rgb, 255:red, 0; green, 0; blue, 0 }  ,draw opacity=0.3 ]   (95.3,26.02) .. controls (74.24,26.88) and (63.38,50.87) .. (62.74,62.87) ;
\draw [color={rgb, 255:red, 0; green, 0; blue, 0 }  ,draw opacity=0.3 ]   (52.32,77.62) .. controls (51.93,82.3) and (72.14,82.3) .. (70.98,77.1) ;
\draw [color={rgb, 255:red, 0; green, 0; blue, 0 }  ,draw opacity=0.3 ] [dash pattern={on 4.5pt off 4.5pt}]  (52.32,77.62) .. controls (55.04,72.94) and (69.81,73.46) .. (70.98,77.1) ;
\draw [color={rgb, 255:red, 0; green, 0; blue, 0 }  ,draw opacity=0.3 ]   (57.76,46.97) .. controls (60.09,50.6) and (64.37,51.64) .. (67.48,46.97) ;
\draw [color={rgb, 255:red, 0; green, 0; blue, 0 }  ,draw opacity=0.3 ] [dash pattern={on 4.5pt off 4.5pt}]  (57.76,46.97) .. controls (60.48,42.29) and (66.31,43.33) .. (67.48,46.97) ;
\draw  [fill={rgb, 255:red, 0; green, 0; blue, 0 }  ,fill opacity=1 ] (61.77,62.09) .. controls (61.77,61.37) and (62.21,60.79) .. (62.74,60.79) .. controls (63.28,60.79) and (63.72,61.37) .. (63.72,62.09) .. controls (63.72,62.81) and (63.28,63.39) .. (62.74,63.39) .. controls (62.21,63.39) and (61.77,62.81) .. (61.77,62.09) -- cycle ;
\draw [color={rgb, 255:red, 0; green, 0; blue, 0 }  ,draw opacity=0.3 ]   (95.3,26.02) .. controls (121.11,27.22) and (130.4,52.25) .. (132.31,63.39) ;
\draw [color={rgb, 255:red, 0; green, 0; blue, 0 }  ,draw opacity=0.3 ]   (96.58,90.29) .. controls (112.54,89.86) and (129.12,73.67) .. (132.31,63.39) ;
\draw [color={rgb, 255:red, 0; green, 0; blue, 0 }  ,draw opacity=0.3 ]   (166.15,90.81) .. controls (147.63,89.95) and (136.14,74.53) .. (132.31,63.39) ;
\draw [color={rgb, 255:red, 0; green, 0; blue, 0 }  ,draw opacity=0.3 ]   (164.87,26.54) .. controls (143.8,27.4) and (132.95,51.39) .. (132.31,63.39) ;
\draw [color={rgb, 255:red, 0; green, 0; blue, 0 }  ,draw opacity=0.3 ]   (126.55,47.49) .. controls (128.88,51.12) and (133.94,52.16) .. (137.04,47.49) ;
\draw  [fill={rgb, 255:red, 0; green, 0; blue, 0 }  ,fill opacity=1 ] (131.34,62.61) .. controls (131.34,61.89) and (131.78,61.31) .. (132.31,61.31) .. controls (132.85,61.31) and (133.28,61.89) .. (133.28,62.61) .. controls (133.28,63.33) and (132.85,63.91) .. (132.31,63.91) .. controls (131.78,63.91) and (131.34,63.33) .. (131.34,62.61) -- cycle ;
\draw [color={rgb, 255:red, 0; green, 0; blue, 0 }  ,draw opacity=0.3 ] [dash pattern={on 4.5pt off 4.5pt}]  (126.55,47.49) .. controls (129.27,42.81) and (135.88,43.85) .. (137.04,47.49) ;
\draw [color={rgb, 255:red, 0; green, 0; blue, 0 }  ,draw opacity=0.3 ]   (121.89,78.14) .. controls (121.5,82.82) and (141.71,82.82) .. (140.54,77.62) ;
\draw [color={rgb, 255:red, 0; green, 0; blue, 0 }  ,draw opacity=0.3 ] [dash pattern={on 4.5pt off 4.5pt}]  (121.89,78.14) .. controls (124.61,73.46) and (139.38,73.98) .. (140.54,77.62) ;
\draw  [fill={rgb, 255:red, 0; green, 0; blue, 0 }  ,fill opacity=0.1 ] (132.51,61.06) .. controls (117.69,98.92) and (75.22,100.77) .. (62.74,60.79) .. controls (77.87,6.19) and (124.99,23.42) .. (132.51,61.06) -- cycle ;

\draw (90.77,54.26) node [anchor=north west][inner sep=0.75pt]    {$B$};

\end{tikzpicture}
                \subcaption{type (III)}
                \label{fig:type3}
            \end{minipage} &
            \begin{minipage}{0.49\textwidth}
                \centering
                \tikzset{every picture/.style={line width=0.75pt}} 

\begin{tikzpicture}[x=0.75pt,y=0.75pt,yscale=-1,xscale=1]

\draw    (6,42.72) -- (164,42.63) ;
\draw    (6,42.72) -- (34.23,9.17) ;
\draw    (6,103.64) -- (164,103.63) ;
\draw    (6,103.64) -- (34.23,70.09) ;
\draw [color={rgb, 255:red, 0; green, 0; blue, 0 }  ,draw opacity=0.3 ]   (31.33,30.2) .. controls (60.26,34.79) and (66.62,55.66) .. (68.73,66.51) ;
\draw [color={rgb, 255:red, 0; green, 0; blue, 0 }  ,draw opacity=0.3 ]   (28.51,93.63) .. controls (46.15,93.21) and (65.2,76.52) .. (68.73,66.51) ;
\draw [color={rgb, 255:red, 0; green, 0; blue, 0 }  ,draw opacity=0.3 ]   (106.07,94.89) .. controls (85.18,91.89) and (72.97,77.35) .. (68.73,66.51) ;
\draw [color={rgb, 255:red, 0; green, 0; blue, 0 }  ,draw opacity=0.3 ]   (104.73,30.62) .. controls (81.44,31.46) and (69.44,54.82) .. (68.73,66.51) ;
\draw [color={rgb, 255:red, 0; green, 0; blue, 0 }  ,draw opacity=0.3 ]   (57.21,80.87) .. controls (56.78,85.43) and (79.12,85.43) .. (77.83,80.37) ;
\draw [color={rgb, 255:red, 0; green, 0; blue, 0 }  ,draw opacity=0.3 ] [dash pattern={on 4.5pt off 4.5pt}]  (57.21,80.87) .. controls (60.22,76.32) and (76.54,76.82) .. (77.83,80.37) ;
\draw [color={rgb, 255:red, 0; green, 0; blue, 0 }  ,draw opacity=0.3 ]   (63.22,51.02) .. controls (65.8,54.56) and (70.53,55.57) .. (73.97,51.02) ;
\draw [color={rgb, 255:red, 0; green, 0; blue, 0 }  ,draw opacity=0.3 ] [dash pattern={on 4.5pt off 4.5pt}]  (63.22,51.02) .. controls (66.23,46.47) and (72.68,47.48) .. (73.97,51.02) ;
\draw  [fill={rgb, 255:red, 0; green, 0; blue, 0 }  ,fill opacity=1 ] (67.66,65.75) .. controls (67.66,65.05) and (68.14,64.48) .. (68.73,64.48) .. controls (69.33,64.48) and (69.81,65.05) .. (69.81,65.75) .. controls (69.81,66.45) and (69.33,67.01) .. (68.73,67.01) .. controls (68.14,67.01) and (67.66,66.45) .. (67.66,65.75) -- cycle ;
\draw [color={rgb, 255:red, 0; green, 0; blue, 0 }  ,draw opacity=1 ]   (68.73,64.48) .. controls (77.01,19.89) and (134.22,27.89) .. (134.22,35.89) .. controls (134.22,43.89) and (94.26,29.89) .. (69.81,65.75) ;
\draw [color={rgb, 255:red, 0; green, 0; blue, 0 }  ,draw opacity=1 ]   (134.22,35.89) -- (134.22,91.89) ;
\draw    (68.73,67.01) .. controls (81.55,100.89) and (135.12,101.89) .. (134.22,91.89) ;
\draw  [dash pattern={on 4.5pt off 4.5pt}]  (68.73,67.01) .. controls (102.43,95.89) and (133.31,69.89) .. (134.22,91.89) ;
\draw  [draw opacity=0][fill={rgb, 255:red, 0; green, 0; blue, 0 }  ,fill opacity=0.1 ] (68.73,67.01) .. controls (100.44,96.29) and (132.22,70.29) .. (134.22,91.89) .. controls (134.03,52.29) and (134.03,69.29) .. (134.22,35.89) .. controls (135.85,41.29) and (94.08,31.29) .. (68.73,67.01) -- cycle ;
\draw  [draw opacity=0][fill={rgb, 255:red, 0; green, 0; blue, 0 }  ,fill opacity=0.1 ] (68.73,67.01) .. controls (80.46,98.29) and (133.13,103.29) .. (134.22,91.89) .. controls (134.03,52.29) and (134.03,69.29) .. (134.22,35.89) .. controls (134.94,28.29) and (77.74,15.29) .. (68.73,67.01) -- cycle ;

\draw (135.98,59.83) node [anchor=north west][inner sep=0.75pt]    {$B$};

\end{tikzpicture}
                \subcaption{type (IV)}
                \label{fig:type4}
            \end{minipage}
        \end{tabular}
        \caption{The types of a bigon $B$}
        \label{fig:types}
    \end{figure}
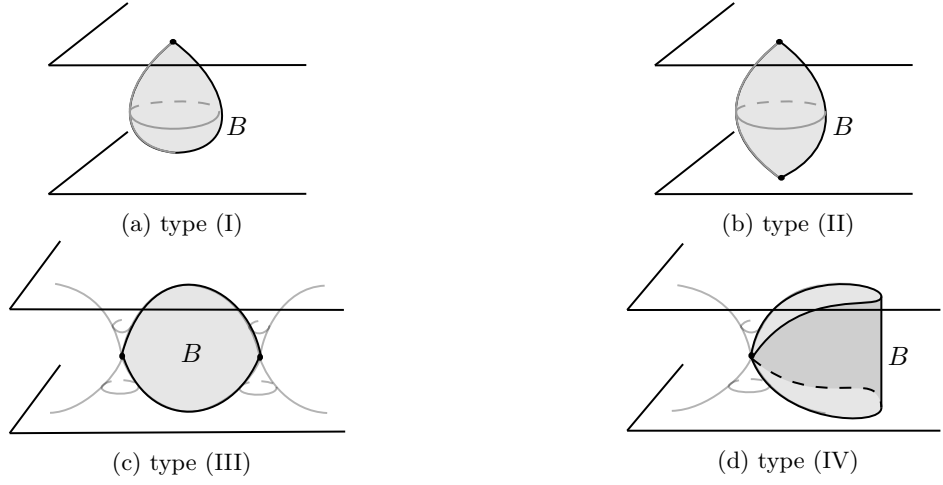
    \begin{itemize}
        \item If there are no bigons whose type is (I), (II), (III), or (IV) in $\mathcal{C}_i$ for any $i\ (0 \leq i \leq n)$, then $\mathcal{T}'$ satisfies $|V'| = 2$ and $|E'| = 0$.
        
        \item If there is a bigon $B$ whose type is (I) in $\mathcal{C}_i$ and $a_i$ is the move flatting $B$ to an edge, then $a_i$ removes the singular point contained in $B$.
        Note that there are no bigons whose type is (II) or (III) in any cell-decompositions $\mathcal{C}_j$ $(j > i)$ since the number of singular points in $\mathcal{C}_j$ is one or less.
        \begin{itemize}
            \item If there are no bigons whose type is (I) or (IV) in any cell-decompositions $\mathcal{C}_j$ $(j > i)$, then $|V'| = 1$ and $|E'| = 0$.
            \item If there is a bigon $B_j$ whose type is (I) in a cell-decomposition $\mathcal{C}_j$ $(j > i)$ and $a_j$ is the move flatting $B_j$, then $|V'| = 0$ and $|E'| = 0$.
            \item If there is a bigon whose type is (IV) in a cell-decomposition $\mathcal{C}_j$ $(j > i)$ and $a_j$ is the move flatting $B_j$, then $|V'| = 0$, $|E'| = 2$, and each component of $E'$ is a circle constructed by one edge.
        \end{itemize}

        \item Suppose that there is a bigon $B$ whose type is (II) in $\mathcal{C}_i$ and $a_i$ is the move flatting $B$ to an edge.
        In this case, $\mathcal{C}_{i+1}$ has no singular points, and so there are no bigons whose type are (I)--(IV) in any cell-decompositions $\mathcal{C}_j$ $(j > i)$.
        Therefore, $|V'| = 0$ and $|E'| = 0$.
        
        
        \item If there is a bigon $B$ whose type is (III) in $\mathcal{C}_i$ and $a_i$ is the move flatting $B$ to an edge, then the atomic move $a_i$ flats $B$ to a circle constructed by two edges as shown Figure \ref{fig:bigon_to_circle}.
        Therefore, $\mathcal{C}_{i+1}$ satisfies $|V_{i+1}| = 0$, $|E_{i+1}| = 1$, and the component of $E_{i+1}$ is a circle constructed by two edge.
        Let $\bar{e}_1$ and $\bar{e}_2$ denote the edges in $E_i$.
        We prove that $\bar{e}_1$ and $\bar{e}_2$ are not identified in $\mathcal{T}'$.
        Suppose that $\bar{e}_1$ and $\bar{e}_2$ are not identified in $\mathcal{C}_j$ and $\bar{e}_1$ and $\bar{e}_2$ are identified in $\mathcal{C}_{j+1}$, where $i < j \leq n-1$.
        In this case, there is a bigon $\bar{B}$ whose edges are $\bar{e}_1$ and $\bar{e}_2$ in $\mathcal{C}_j$, and $a_j$ is the move flatting $\bar{B}$ to an edge.
        The edges $\bar{e}_1$ and $\bar{e}_2$ are in $E_j \subset \mathcal{S}_j^0 \cap \mathcal{S}_j^1$, and so $\bar{B}$ is a component of $\mathcal{S}_j^k$, where $k = 0 \text{ or } 1$.
        However, a component of $\mathcal{S}_j^k$ has at least one triangle because a component of $\mathcal{S} \times \{k\} \subset \partial \mathcal{T}$ has triangles and these triangles are not removed by the atomic moves and the move shrinking the anunlus $A$.
        This contradicts that a component constructed by $B$ does not contain triangles, therefore, $\bar{e}_1$ and $\bar{e}_2$ are not identified in $\mathcal{T}'$.
        Thus, $|V'| = 0$, $|E'| = 1$, and the component of $E'$ is a circle constructed by two edges.
        
        \item Suppose that there is a bigon $B$ whose type is (IV) in $\mathcal{C}_i$ and $a_i$ is the move flatting $B$ to an edge.
        Note that the atomic moves do not eliminate a circle constructed by one edge.
        \begin{itemize}
            \item If there are no bigons whose type is (I) or (IV) in any cell-decompositions $\mathcal{C}_j$ $(j > i)$, then $|V'| = 1$ and $|E'| = 2$.
            \item If there is a bigon $B_j$ whose type is (I) in a cell-decomposition $\mathcal{C}_j$ $(j > i)$ and $a_j$ is the move flatting $B_j$, then $|V'| = 0$ and $|E'| = 2$.
            \item If there is a bigon whose type is (IV) in a cell-decomposition $\mathcal{C}_j$ $(j > i)$ and $a_j$ is the move flatting $B_j$, then $|V'| = 0$, $|E'| = 4$, and each component of $E'$ is a circle constructed by one edge.
        \end{itemize}
    \end{itemize}
\end{proof}

\section{The proof that classical knot recognition is in NP}


Let $D$ be a diagram of a virtual link $L$.
We can regard $L$ as a link $\hat{D}$ in $\mathcal{S} \times I$, where $\mathcal{S}$ is a closed orientable surface.
Let $M$ be the exterior $\text{cl}(\mathcal{S} \times I - N(\hat{D}))$.
By Theorem \ref{thm:K}, $L$ is classical if and only if $g(\mathcal{S})$ is reduced to zero by repeating splitting and destabilization, where $g(\mathcal{S})$ denotes the sum of the genera of connected components of $\mathcal{S}$.
In this section, we give an algorithm to split and destabilize on a triangulation of $M$.

\subsection{Splitting on a triangulation}
In this subsection, we give an algorithm to split on a triangulation of the exterior of a link in a thickened closed orientable surface $\mathcal{S} \times I$.

\begin{ope}\label{ope:splitting}
    Suppose that $\hat{D}$ is a link in a thickened closed orientable surface $\mathcal{S} \times I$, $M$ is the exterior of $\hat{D}$, and $\mathcal{T}$ is a triangulation of $M$.
    Given $\mathcal{T}$ and a normal 2-sphere $F$ in $M$ with respect to $\mathcal{T}$,
    do the following operations:
    \begin{enumerate}
        \item perform the crushing procedure on $\mathcal{T}$ using $F$,
        \item remove the components which contains no torus boundaries other than the boundary components obtained from the copies of $\mathcal{S}$.
    \end{enumerate}
\end{ope}

The splitting operation is performed at least once by Operation \ref{ope:splitting}.
\begin{lem}\label{lem:splitting}
    Suppose that $\hat{D}$ is a link in a thickened closed orientable surface $\mathcal{S} \times I$, $\mathcal{T}$ is a triangulation of the exterior $M = \text{cl}(\mathcal{S} \times I - N(\hat{D}))$ of $\hat{D}$, and $F$ is a normal $2$-sphere in $M$ with respect to $\mathcal{T}$.
    Let $\mathcal{T}'$ be the triangulation obtained by the step 1 of Operation \ref{ope:splitting} 
    and $M'$ be the underlying 3-manifold of $\mathcal{T}'$.
    Then, there is a sequence of 3-manifolds $M \to M_0 \rightarrow M_1 \rightarrow \cdots \rightarrow M_{n-1} \rightarrow M_n = M'$, and
    $M_{i+1}$ is obtained from $M_{i}$ by one of the following operations:
    \begin{itemize}
        \item splitting $M_i$ by an essential $2$-sphere,
        \item compressing the boundary of $M_i$,
        \item adding a 3-ball or a 3-sphere component,
        \item filling a boundary 2-sphere with a 3-ball,
        \item removing a $3$-ball or a $3$-sphere component.
    \end{itemize}
    Moreover, if the given normal surface $F$ is essential in $M$, then splitting is performed at least once.
\end{lem}

\begin{proof}
    Consider performing the crushing procedure using $F$ on $\mathcal{T}$.
    Let $\mathcal{C}_0$ be the cell-decomposition which is obtained by the non-destructive crushing procedure and $M_0$ be the underlying 3-manifold of $\mathcal{C}_0$.
    The $3$-manifold $M_0$ is obtained from $M$ by splitting if $F$ is essential, otherwise, $M_0$ is obtained by adding a $3$-sphere component to $M$.
    In either case, $M_0$ is obtained from $M$ by one of the above four operations.
    Since $M$ contains no two-sided projective planes, $M_0$ also contains no two-sided projective planes.
    By Theorem \ref{thm:B_crush}, there is a sequence of 3-manifold $M_0 \to M_1 \rightarrow \cdots \rightarrow M_{n-1} \rightarrow M_n = M'$ and, 
    $M_{i+1}$ is obtained from $M_{i}$ by one of the following operations:
    \begin{itemize}
        \item cutting open along a properly embedded disk $S$ in $M_i$,
        \item cutting open along a 2-sphere $S$ in $M_i$ and filling the resulting boundary spheres with 3-balls,
        \item removing a 3-ball, a 3-sphere, a lens space $L(3,1)$, a projective space $\mathbb{R}P^3$, $\mathbb{S}^2 \times \mathbb{S}^1$ or a twisted $\mathbb{S}^1$ bundle $\mathbb{S}^2 \tilde{\times} \mathbb{S}^1$ component,
        \item filling a boundary sphere in $\partial M_i$ with a 3-ball.
    \end{itemize}
    We can see that the following by observing these operations in detail for $M_i$.
    \begin{enumerate}
        \item The operation cutting $M_i$ open along a properly embedded disk $S$ in $M_i$:
        \begin{itemize}
            \item In the case where $\partial S$ is essential in $\partial M_i$:
            \begin{itemize}
                \item This case implies that there is a component $\hat{D}'$ of $\hat{D}$ and $\partial S$ is in $\partial N(\hat{D}')$, i.e., $\hat{D}'$ is the trivial knot.
                Thus, $M_{i+1}$ is obtained from $M_i$ by compressing the boundary $\partial N(\hat{D}')$
            \end{itemize}
            \item In the case where $\partial S$ is inessential in $\partial M_i$:
            \begin{itemize}
                \item If $S$ is essential in $M_i$, this operation is the splitting operation.
                \item If $S$ is inessential in $M_i$, this operation is the operation to add a 3-ball.
            \end{itemize}
        \end{itemize}

        \item The operation cutting open $M_i$ along an embedded 2-sphere $S \subset M_i$ and filling the resulting boundary spheres with 3-balls:
        \begin{itemize}
            \item If $S$ is essential in $M_i$, this operation is the splitting operation.
            \item If $S$ is inessential in $M_i$, this operation is the operation to add a 3-sphere.
        \end{itemize}

        \item The operation removing a 3-ball, a 3-sphere, a lens space $L(3,1)$, a projective space $\mathbb{R}P^3$, $\mathbb{S}^2 \times \mathbb{S}^1$ or a twisted $\mathbb{S}^1$ bundle $\mathbb{S}^2 \tilde{\times} \mathbb{S}^1$ component:
        \begin{itemize}
            \item A 3-ball component or a 3-sphere component may be added to $M_i$ by the above operations.
            This operation removes one of these components.
        \end{itemize}
        
        \item The operation filling a boundary sphere with a 3-ball:
        \begin{itemize}
            \item A boundary component in $\partial M_i$ which is a 2-sphere is one of the follows:
            \begin{itemize}
                \item a component of a copy of the supporting surface whose genus is $0$ ,
                \item a boundary component obtained by compressing the boundary of $M_j$ $(j<i)$,
                \item the boundary of a 3-ball component.
            \end{itemize}
            This operation deletes one of these boundaries.
        \end{itemize}
    \end{enumerate}

    From these observation, we see that $M_{i+1}$ is obtained from $M_i$ by one of the above four operations.
    Furthermore, $M_0$ is obtained by cutting $M$ open along the given normal $2$-sphere and shrinking each copy of $F$ to a point.
    Therefore, the splitting is performed at least once if $F$ is essential.
\end{proof}

By the proof of Lemma \ref{lem:splitting}, we see the following corollary.
\begin{cor}\label{cor:splitting_comp}
    Suppose the same situation of Operation \ref{lem:splitting}.
    Let $\hat{D'}$ and $\hat{D''}$ be sublinks of $\hat{D}$.
    A component of the 3-manifold $M'$ obtained from $M$ by Operation \ref{ope:splitting} is one of the following:
    \begin{itemize}
        \item $\text{cl}(\mathcal{S} \times I - N(\hat{D'}))$,
        \item $\mathbb{B}^3 - N(\hat{D''})$,
        \item $\mathbb{S}^3 - N(\hat{D''})$,
        \item a component removed some 3-balls from one of the above components.
    \end{itemize}
\end{cor}
Note that there is a $2$-sphere boundary other than the boundary components obtained from the copies of $\mathcal{S}$ in the resulting $3$-manifold $M'$ if and only if there is a split component $\hat{D''}$ of $\hat{D}$ which is the trivial knot and the boundary $\partial N(\hat{D''})$ is compressed.

In particular, we consider the case where $\hat{D} \subset \mathcal{S} \times I$ is a knot, $\mathcal{S}$ is connected, $g(\mathcal{S})$ is not zero, and $F$ is an essential $2$-sphere in the exterior $M$ of $\hat{D}$.
Since $g(\mathcal{S})$ is not zero, then there is an embedded $\text{cl}(\mathbb{B}^3 - N(\hat{D}))$ whose $2$-sphere boundary is $F$ in the exterior $M$, and the splitting operation using $F$ is performed by Operation \ref{ope:splitting}.
Therefore, the underlying $3$-manifold $M'$ of the resulting triangulation of Operation \ref{ope:splitting} is $\text{cl}(\mathbb{S}^3 - N(D))$ or the empty set.
Note that if $\hat{D}$ is the trivial knot, then $\partial N(\hat{D})$ may be compressed by an essential disk whose boundary is essential in $\partial N(\hat{D})$.
In this case, all components are removed in the step 2 of Operation \ref{ope:splitting}, i.e., $M'$ is the empty set.
Conversely, if $\hat{D}$ is not the trivial knot, then there are no essential disk whose boundary is essential in $\partial N(\hat{D})$, and so $\partial N(\hat{D})$ keeps its topology as the torus and the component containing $\partial N(\hat{D})$ is not removed.
Thus, if the underlying $3$-manifold of the resulting triangulation of Operation \ref{ope:splitting} is the empty set, then $\hat{D}$ is the trivial knot.
\begin{cor}\label{cor:splitting_comp_knot}
    Suppose that $\mathcal{S}$ is a closed orientable surface whose genus is not zero, $\hat{D}$ is a knot in $\mathcal{S} \times I$, and $\mathcal{T}$ is a triangulation of the exterior $M = \text{cl}(\mathcal{S} \times I - N(\hat{D}))$.
    If $F$ is an essential normal $2$-sphere in $M$ with respect to $\mathcal{T}$, then the underlying $3$-manifold $M'$ of the resulting triangulation of Operation \ref{ope:splitting} is $\text{cl}(\mathbb{S}^3 - N(\hat{D}))$ or the empty set.
    Furthermore, if $M'$ is the empty set, then $\hat{D}$ is the trivial knot.
\end{cor}


Next, we consider the number of tetrahedra in the triangulation obtained by Operation \ref{ope:splitting}.
If a normal surface $F$ is not a vertex link, then $F$ contains at least one quadrilateral normal disk, and
the crushing procedure removes the tetrahedra which contain quadrilateral normal disks.
Thus, Corollary \ref{cor:splittingNum} holds.
\begin{cor}\label{cor:splittingNum}
    Operation \ref{ope:splitting} reduces the number of tetrahedra if a normal surface $F$ is not a vertex link.
\end{cor}
In particular, Operation \ref{ope:splitting} using an essential normal $2$-sphere reduces the number of tetrahedra since any essential normal $2$-sphere is not a vertex link.

Next, we analyze the running time of Operation \ref{ope:splitting}.
\begin{lem}\label{lem:splittingTime}
    If a normal $2$-sphere $F$ is a vertex surface, then Operation \ref{ope:splitting} is carried out in time $\mathcal{O}(n^2)$, where $n$ is the number of tetrahedra of $\mathcal{T}$.
\end{lem}
\begin{proof}
    Let $\mathbf{x} = (x_1, \ldots, x_{7n})$ denote the vector representation of $F$.
    By Theorem \ref{thm:HLP}, for each integer $i = 1, \ldots, 7n$, $x_i$ is $2^{7n-1}$ or less.
    Thus, $x_i$ can be encoded with $\mathcal{O}(\log 2^{7n-1}) = \mathcal{O}(n)$ bits, and so $\mathbf{x}$ can be encoded with $\mathcal{O}(n^2)$ bits because there are $7n$ integers in $\mathbf{x}$, and so we can read $\mathbf{x}$ in time $\mathcal{O}(n^2)$.
    Consider the running time of each step of Operation \ref{ope:splitting}.
    The step 1 runs in $\mathcal{O}(n)$ since the crushing procedure can be carried out in time $\mathcal{O}(n)$ from Theorem \ref{thm:crush}.
    Suppose that $\mathcal{T}'$ is the triangulation obtained from $\mathcal{T}$ by the step 1 and $M'$ is the underlying $3$-manifold of $\mathcal{T}'$.
    Let $n'$ be the number of tetrahedra in $\mathcal{T}'$. 
    Since we can determine whether each component of $M'$ contains components of $\partial N(\hat{D})$ in $\mathcal{O}(n')$ time and $n' < n$ by Corollary \ref{cor:splittingNum}, the step 2 runs in $\mathcal{O}(n)$ time.
    Hence, Operation \ref{ope:splitting} is carried out in time $\mathcal{O}(n^2)$.
\end{proof}

\subsection{Destabilization on a triangulation}
In this subsection, we give an algorithm to destabilize on a triangulation $\mathcal{T}$ of the exterior $M$ of a link $\hat{D}$ in a thickened closed orientable surface $\mathcal{S} \times I$.
Recall that destabilization is the operation cutting $M$ along a vertical essential annulus and filling each copy of the annulus with $D^2 \times I$.
Suppose that $A$ is a vertical normal annulus in $M$ with respect to $\mathcal{T}$.
Let $\mathcal{T}'$ denote the triangulation obtained by crushing using $A$ on $\mathcal{T}$, and let $M'$ denote the underlying singular $3$-manifold.
Suppose that $V'$ and $E'$ is the subset of $\mathcal{T}'$ defined by Definition \ref{def:VandE}.
We define three operations removing components in $V'$ and $E'$.

\begin{ope}\label{ope:desingularization}
    Let $s \in V'$ be a singular point in $\mathcal{T}'$.
    Then, we call the following operation {\it desingularization}:
    \begin{itemize}
        \item stretching the singular point $s$ to an edge as depicted in Figure \ref{fig:desing}.
    \end{itemize}
\end{ope}
\begin{figure}[htbp]
    \centering
    \tikzset{every picture/.style={line width=0.75pt}} 

\begin{tikzpicture}[x=0.75pt,y=0.75pt,yscale=-0.7,xscale=0.7]

\draw   (25.94,57.87) -- (80.56,57.87) -- (67.28,40.79) -- cycle ;
\draw   (67.28,40.79) -- (80.56,57.87) -- (113.55,34.72) -- cycle ;
\draw   (80.56,57.87) -- (142.42,42.31) -- (113.55,34.72) -- cycle ;
\draw   (80.56,57.87) -- (102.95,71.91) -- (142.42,42.31) -- cycle ;
\draw   (47.58,81.02) -- (102.95,71.91) -- (80.56,57.87) -- cycle ;
\draw   (47.58,81.02) -- (80.56,57.87) -- (25.94,57.87) -- cycle ;
\draw [color={rgb, 255:red, 0; green, 0; blue, 0 }  ,draw opacity=1 ]   (28.57,24.09) .. controls (30.09,46.1) and (62.73,48.38) .. (80.56,57.87) ;
\draw  [color={rgb, 255:red, 0; green, 0; blue, 0 }  ,draw opacity=1 ][fill={rgb, 255:red, 0; green, 0; blue, 0 }  ,fill opacity=1 ] (77.91,57.87) .. controls (77.91,56.4) and (79.1,55.21) .. (80.56,55.21) .. controls (82.03,55.21) and (83.22,56.4) .. (83.22,57.87) .. controls (83.22,59.34) and (82.03,60.52) .. (80.56,60.52) .. controls (79.1,60.52) and (77.91,59.34) .. (77.91,57.87) -- cycle ;
\draw    (2.76,52.18) -- (25.94,57.87) ;
\draw    (3.52,66.6) -- (25.94,57.87) ;
\draw    (25.16,89.75) -- (47.58,81.02) ;
\draw    (102.95,71.91) -- (121.93,81.02) ;
\draw    (98.4,84.05) -- (102.95,71.91) ;
\draw    (51.34,93.16) -- (47.58,81.02) ;
\draw   (272.05,58.63) -- (326.67,58.63) -- (313.38,41.55) -- cycle ;
\draw   (313.38,41.55) -- (326.67,58.63) -- (359.65,35.48) -- cycle ;
\draw   (326.67,58.63) -- (388.52,43.07) -- (359.65,35.48) -- cycle ;
\draw   (326.67,58.63) -- (349.06,72.67) -- (388.52,43.07) -- cycle ;
\draw   (293.68,81.78) -- (349.06,72.67) -- (326.67,58.63) -- cycle ;
\draw   (293.68,81.78) -- (326.67,58.63) -- (272.05,58.63) -- cycle ;
\draw    (248.87,52.93) -- (272.05,58.63) ;
\draw    (249.63,67.36) -- (272.05,58.63) ;
\draw    (271.26,90.51) -- (293.68,81.78) ;
\draw    (349.06,72.67) -- (368.03,81.78) ;
\draw    (344.5,84.81) -- (349.06,72.67) ;
\draw    (297.44,93.92) -- (293.68,81.78) ;
\draw [color={rgb, 255:red, 0; green, 0; blue, 0 }  ,draw opacity=1 ]   (326.29,5.12) -- (326.67,58.63) ;
\draw    (272.05,58.63) -- (326.29,5.12) ;
\draw    (293.68,81.78) -- (326.29,5.12) ;
\draw    (326.29,5.12) -- (349.06,72.67) ;
\draw    (326.29,5.12) -- (388.52,43.07) ;
\draw    (326.29,5.12) -- (359.65,35.48) ;
\draw    (326.29,5.12) -- (313.38,41.55) ;
\draw    (142.42,42.31) -- (161.4,51.42) ;
\draw    (142.42,42.31) -- (160.64,38.51) ;
\draw    (388.52,43.07) -- (407.5,52.18) ;
\draw    (388.52,43.07) -- (406.74,39.27) ;
\draw    (173.7,47.09) -- (233,47.3) ;
\draw [shift={(235,47.31)}, rotate = 180.2] [color={rgb, 255:red, 0; green, 0; blue, 0 }  ][line width=0.75]    (10.93,-3.29) .. controls (6.95,-1.4) and (3.31,-0.3) .. (0,0) .. controls (3.31,0.3) and (6.95,1.4) .. (10.93,3.29)   ;

\draw (28.66,14.6) node  [color={rgb, 255:red, 0; green, 0; blue, 0 }  ,opacity=1 ] [align=left] {a singular point};

\end{tikzpicture}
    \caption{Desingularization}
    \label{fig:desing}
\end{figure}
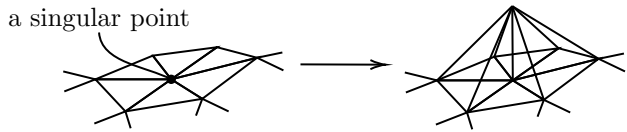
\begin{ope}\label{ope:eliminate_loop1}
    Let $c$ be a circle constructed by one edge in $E' \subset \mathcal{T}'$.
    \begin{enumerate}
        \item choosing a triangle $t$ in $\partial \mathcal{T}'$ containing the circle $c$.
        \item adding a cone which is obtained by gluing two faces of a tetrahedron.
        Let $c'$ be the circle in the bottom of the cone and $t'$ be the side of the cone.
        \item gluing $t$ and $t'$ so that $c$ and $c'$ are identified.
        \item desingularizing at the vertex in the circle $c = c'$.
    \end{enumerate}
    We illustrate Operation \ref{ope:eliminate_loop1} in Figure \ref{fig:eliminate_loop1}.
    \begin{figure}[htpb]
        \centering
        \input{tikz/eliminate_loop1.tex}
        \caption{Operation \ref{ope:eliminate_loop1}}
        \label{fig:eliminate_loop1}
    \end{figure}
\end{ope}

\begin{ope}\label{ope:eliminate_loop2}
   Let $c$ be a circle constructed by two edges in $E' \subset \mathcal{T}'$.
    \begin{enumerate}
        \item gluing the two edges of $c$.
        \item desingularizing at each end point of the edge identified the two edges. 
    \end{enumerate}
    We illustrate Operation \ref{ope:eliminate_loop2} in Figure \ref{fig:eliminate_loop2}.
    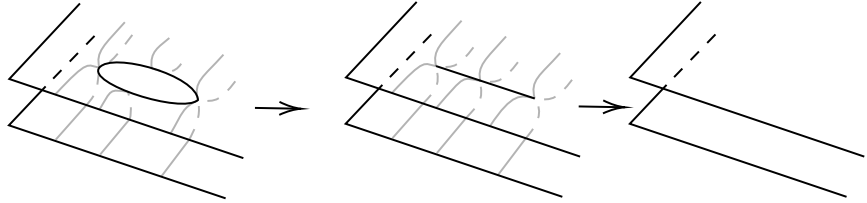
\begin{figure}[htpb]
        \centering
        \tikzset{every picture/.style={line width=0.75pt}} 

\begin{tikzpicture}[x=0.75pt,y=0.75pt,yscale=-1,xscale=1]

\draw    (215.75,5.02) -- (180.1,43.15) -- (298.17,83.09) ;
\draw    (198.45,46.9) -- (179.92,66.52) -- (289.21,103.39) ;
\draw  [dash pattern={on 4.5pt off 4.5pt}]  (223.11,21.42) -- (198.45,46.9) ;
\draw    (225.55,37.4) -- (275.28,53.76) ;
\draw [color={rgb, 255:red, 0; green, 0; blue, 0 }  ,draw opacity=0.3 ]   (225.59,58.59) .. controls (240.89,40.91) and (238.4,43.4) .. (247.99,44.87) ;
\draw [color={rgb, 255:red, 0; green, 0; blue, 0 }  ,draw opacity=0.3 ]   (252.89,67.48) .. controls (263.82,50.88) and (268.31,51.88) .. (275.28,53.76) ;
\draw [color={rgb, 255:red, 0; green, 0; blue, 0 }  ,draw opacity=0.3 ]   (203.16,51.11) .. controls (220.45,32.43) and (219.45,36.42) .. (225.55,37.4) ;
\draw [color={rgb, 255:red, 0; green, 0; blue, 0 }  ,draw opacity=0.3 ]   (238.4,14.48) .. controls (228.13,25.94) and (223.44,29.44) .. (225.55,37.4) ;
\draw [color={rgb, 255:red, 0; green, 0; blue, 0 }  ,draw opacity=0.3 ]   (260.83,22.46) .. controls (250.57,33.92) and (245.87,37.42) .. (247.99,45.37) ;
\draw [color={rgb, 255:red, 0; green, 0; blue, 0 }  ,draw opacity=0.3 ]   (288.12,30.85) .. controls (277.86,42.31) and (272.8,44.4) .. (275.28,53.76) ;
\draw [color={rgb, 255:red, 0; green, 0; blue, 0 }  ,draw opacity=0.3 ] [dash pattern={on 4.5pt off 4.5pt}]  (218.95,56.37) .. controls (224.93,49.89) and (227.92,47.89) .. (225.55,37.4) ;
\draw [color={rgb, 255:red, 0; green, 0; blue, 0 }  ,draw opacity=0.3 ] [dash pattern={on 4.5pt off 4.5pt}]  (244.38,27.95) .. controls (237.4,36.92) and (234.9,38.92) .. (225.55,37.4) ;
\draw [color={rgb, 255:red, 0; green, 0; blue, 0 }  ,draw opacity=0.3 ] [dash pattern={on 4.5pt off 4.5pt}]  (266.82,35.43) .. controls (259.84,44.4) and (257.34,46.4) .. (247.99,44.87) ;
\draw [color={rgb, 255:red, 0; green, 0; blue, 0 }  ,draw opacity=0.3 ] [dash pattern={on 4.5pt off 4.5pt}]  (294.11,44.32) .. controls (287.13,53.29) and (284.63,55.29) .. (275.28,53.76) ;
\draw [color={rgb, 255:red, 0; green, 0; blue, 0 }  ,draw opacity=0.3 ]   (218.95,56.37) -- (203.49,73.82) ;
\draw [color={rgb, 255:red, 0; green, 0; blue, 0 }  ,draw opacity=0.3 ] [dash pattern={on 4.5pt off 4.5pt}]  (240.89,63.35) .. controls (246.87,56.87) and (249.86,54.87) .. (247.49,44.38) ;
\draw [color={rgb, 255:red, 0; green, 0; blue, 0 }  ,draw opacity=0.3 ]   (240.89,63.35) -- (225.93,81.3) ;
\draw [color={rgb, 255:red, 0; green, 0; blue, 0 }  ,draw opacity=0.3 ] [dash pattern={on 4.5pt off 4.5pt}]  (270.8,73.82) .. controls (276.79,66.84) and (279.78,60.85) .. (275.41,53.85) ;
\draw [color={rgb, 255:red, 0; green, 0; blue, 0 }  ,draw opacity=0.3 ]   (270.8,73.82) -- (256.84,92.27) ;
\draw    (134.18,58.51) -- (155.42,58.9) ;
\draw [shift={(157.42,58.94)}, rotate = 181.07] [color={rgb, 255:red, 0; green, 0; blue, 0 }  ][line width=0.75]    (10.93,-3.29) .. controls (6.95,-1.4) and (3.31,-0.3) .. (0,0) .. controls (3.31,0.3) and (6.95,1.4) .. (10.93,3.29)   ;
\draw    (45.93,5.79) -- (10.28,43.92) -- (128.34,83.86) ;
\draw    (28.62,47.66) -- (10.09,67.29) -- (119.39,104.16) ;
\draw  [dash pattern={on 4.5pt off 4.5pt}]  (53.28,22.19) -- (28.62,47.66) ;
\draw [color={rgb, 255:red, 0; green, 0; blue, 0 }  ,draw opacity=0.3 ]   (55.77,59.36) .. controls (61.84,49.24) and (60.61,48.96) .. (70.2,50.43) ;
\draw [color={rgb, 255:red, 0; green, 0; blue, 0 }  ,draw opacity=0.3 ]   (91.45,71.8) .. controls (98.02,60.46) and (97.42,58.67) .. (105.45,54.53) ;
\draw [color={rgb, 255:red, 0; green, 0; blue, 0 }  ,draw opacity=0.3 ]   (33.33,51.88) .. controls (50.62,33.2) and (49.62,37.19) .. (55.73,38.16) ;
\draw [color={rgb, 255:red, 0; green, 0; blue, 0 }  ,draw opacity=0.3 ]   (68.57,15.25) .. controls (58.31,26.71) and (53.61,30.21) .. (55.73,38.16) ;
\draw [color={rgb, 255:red, 0; green, 0; blue, 0 }  ,draw opacity=0.3 ]   (91.01,23.23) .. controls (85.72,27.75) and (80.94,30.13) .. (82.14,38.49) ;
\draw [color={rgb, 255:red, 0; green, 0; blue, 0 }  ,draw opacity=0.3 ]   (118.3,31.62) .. controls (108.04,43.08) and (102.97,45.17) .. (105.45,54.53) ;
\draw [color={rgb, 255:red, 0; green, 0; blue, 0 }  ,draw opacity=0.3 ] [dash pattern={on 4.5pt off 4.5pt}]  (49.12,57.14) .. controls (55.11,50.65) and (54.44,48.52) .. (55.73,38.16) ;
\draw [color={rgb, 255:red, 0; green, 0; blue, 0 }  ,draw opacity=0.3 ] [dash pattern={on 4.5pt off 4.5pt}]  (72.35,21.66) .. controls (65.37,30.63) and (62.79,35.39) .. (55.73,38.16) ;
\draw [color={rgb, 255:red, 0; green, 0; blue, 0 }  ,draw opacity=0.3 ] [dash pattern={on 4.5pt off 4.5pt}]  (96.99,36.19) .. controls (93.48,41.48) and (88.7,40.88) .. (82.14,38.49) ;
\draw [color={rgb, 255:red, 0; green, 0; blue, 0 }  ,draw opacity=0.3 ] [dash pattern={on 4.5pt off 4.5pt}]  (124.28,45.08) .. controls (117.3,54.06) and (114.81,56.05) .. (105.45,54.53) ;
\draw [color={rgb, 255:red, 0; green, 0; blue, 0 }  ,draw opacity=0.3 ]   (49.12,57.14) -- (33.67,74.59) ;
\draw [color={rgb, 255:red, 0; green, 0; blue, 0 }  ,draw opacity=0.3 ] [dash pattern={on 4.5pt off 4.5pt}]  (71.06,64.12) .. controls (71.39,61.77) and (72.57,60.93) .. (70.2,50.43) ;
\draw [color={rgb, 255:red, 0; green, 0; blue, 0 }  ,draw opacity=0.3 ]   (71.06,64.12) -- (56.11,82.07) ;
\draw [color={rgb, 255:red, 0; green, 0; blue, 0 }  ,draw opacity=0.3 ] [dash pattern={on 4.5pt off 4.5pt}]  (100.98,74.59) .. controls (106.37,64.64) and (106.37,60.46) .. (105.59,54.62) ;
\draw [color={rgb, 255:red, 0; green, 0; blue, 0 }  ,draw opacity=0.3 ]   (100.98,74.59) -- (87.02,93.04) ;
\draw   (55.73,38.16) .. controls (49.01,46.19) and (97.06,61.77) .. (105.45,54.53) .. controls (104.82,42.07) and (62.44,30.13) .. (55.73,38.16) -- cycle ;
\draw    (359.08,5.02) -- (323.43,43.15) -- (441.5,83.09) ;
\draw    (341.78,46.9) -- (323.25,66.52) -- (432.55,103.39) ;
\draw  [dash pattern={on 4.5pt off 4.5pt}]  (366.44,21.42) -- (341.78,46.9) ;
\draw    (297.46,57.77) -- (318.7,58.16) ;
\draw [shift={(320.7,58.2)}, rotate = 181.07] [color={rgb, 255:red, 0; green, 0; blue, 0 }  ][line width=0.75]    (10.93,-3.29) .. controls (6.95,-1.4) and (3.31,-0.3) .. (0,0) .. controls (3.31,0.3) and (6.95,1.4) .. (10.93,3.29)   ;

\end{tikzpicture}
        \caption{Operation \ref{ope:eliminate_loop2}}
        \label{fig:eliminate_loop2}
    \end{figure}
\end{ope}

\begin{ope}\label{ope:des}
    Suppose that $\hat{D}$ is a link in a thickened closed orientable surface $\mathcal{S} \times I$ and $\mathcal{T}$ is a triangulation of the exterior $M = \text{cl}(\mathcal{S} \times I - N(\hat{D}))$ of $\hat{D}$.
    Given $\mathcal{T}$ and a vertical normal annulus $A$ in $M$ with respect to $\mathcal{T}$, carry out the following operations:
    \begin{enumerate}
        \item crush $\mathcal{T}$ using $A$, and let $\mathcal{T}'$ be the triangulation obtained from $\mathcal{T}$,
        \item if there are circles constructed by one edge in $E' \subset \mathcal{T}'$, then run Operation \ref{ope:eliminate_loop1} at each circle,
        \item if there are circles constructed by two edges in $E' \subset \mathcal{T}'$, then run Operation \ref{ope:eliminate_loop2} at each circle,
        \item if there are singular points in $\mathcal{T}'$, then desingularize at each singular point,
        \item remove the components which contains no torus boundaries other than the boundary components obtained from the copies of $\mathcal{S}$.
    \end{enumerate}
\end{ope}

\begin{lem}\label{lem:des}
    Suppose that $\hat{D}$ is a link in a thickened closed orientable surface $\mathcal{S} \times I$, $\mathcal{T}$ is a triangulation of the exterior $M = \text{cl}(\mathcal{S} \times I - N(\hat{D}))$ of $\hat{D}$, and $A$ is a vertical normal surface in $M$ with respect to $\mathcal{T}$.
    Consider running the steps 1--4 of Operation \ref{ope:des} on $\mathcal{T}$ using $A$.
    Let $\mathcal{T}''$ and $M''$ denote the resulting triangulation and the underlying $3$-manifold, respectively.
    Then, there is a sequence of $3$-manifolds $M \to N_0 \to N_1 \to \cdots \to N_n = M''$ such that $N_0$ is obtained from $M$ by destabilizing using $A$ and 
    for any $i\ (0 \leq i \leq n-1)$, $N_{i+1}$ is homeomorphic to $N_i$ or $N_{i+1}$ is obtained from $N_i$ by one of the following operations:
    \begin{enumerate}
        \renewcommand{\labelenumi}{\thelem.\arabic{enumi}}
        \item splitting using an essential $2$-sphere or an essential disk on $N_i$,
        \item compressing the boundary of $N_i$,
        \item adding a 3-ball or a 3-sphere component,
        \item filling a boundary 2-sphere in $\partial N_i$ with a 3-ball,
        \item removing a $3$-ball or a $3$-sphere component.
        \item destabilizing using a vertical essential annulus on $N_i$,
        \item cutting open $N_i$ along an embedded $2$-sphere and then filling one resulting boundary with a $3$-ball and gluing the other resulting boundary to the boundary obtained by removing a $3$-ball from $\mathbb{S}^2 \times I$.
    \end{enumerate}
\end{lem}
\begin{proof}
    Let $\mathcal{T}'$ be the triangulation obtained from $\mathcal{T}$ by crushing using $A$ and $M'$ be the underlying singular $3$-manifold of $\mathcal{T}'$.
    From Lemma \ref{lem:crushing_lemma}, there is a sequence of cell-decompositions $\mathcal{C}_0 \to \mathcal{C}_1 \to \cdots \to \mathcal{C}_n = \mathcal{T}'$, where $\mathcal{C}_0$ is the cell-decomposition obtained by the non-destructive crushing procedure using $A$ and $\mathcal{C}_{i+1}$ is obtained from $\mathcal{C}_i$ by an atomic move $a_i$.
    For any $i$, we denote the singular $3$-manifold of $\mathcal{C}_i$ by $M_i$.
    By Lemma \ref{lem:crush_annulus}, $M_{i+1}$ is homeomorphic to $M_{i}$ or $M_{i+1}$ is obtained from $M_i$ by one of the following operations:
    \begin{enumerate}
        \renewcommand{\labelenumi}{(\alph{enumi})}
        \item removing a $3$-ball or a $3$-sphere component,
        \item filling a boundary sphere in $\partial M_i$ with a $3$-ball,
        \item cutting open $M_i$ along a properly embedded disk in $M_i$,
        \item cutting open $M_i$ along an embedded $2$-sphere in $M_i$ and filling the resulting boundaries with $3$-balls,
        \item cutting open $M_i$ along a bigon $B$ which satisfies the following conditions:
        \begin{itemize}
            \item $\partial B \cap \mathcal{S}_i^{0} \neq \emptyset$,
            \item $\partial B \cap \mathcal{S}_i^{1} \neq \emptyset$,
            \item $v_1$ and $v_2$ are identified at a singular point, where $v_1$ and $v_2$ are the vertices of $B$,
        \end{itemize}
    \end{enumerate}
    where $\mathcal{S}_i^{k}$ is the subset of $\partial M_i$ obtained from $\mathcal{S} \times \{k\} \subset \partial M$.
    If $M_{i+1}$ is not homeomorphic to $M_i$, we define $\overline{a_i}$ as the operation to obtain $M_{i+1}$ from $M_{i}$, otherwise we define $\overline{a_i}$ as the operation doing nothing.
    
    \begin{claim}
        Each of the operations (a)--(d) corresponds to the operations \thelem.1--\thelem.5 in the statement of Lemma \ref{lem:des}.
    \end{claim}
    \noindent {\it Proof of Claim A.}
    We show Claim A by observing each of the operations (a)--(d).
    \begin{itemize}
        \item The operation (a) corresponds to the operation \thelem.5.
        \item The operation (b) corresponds to the operation \thelem.4.
        \item We consider the operation (c).
            Let $D$ denote the disk cutting open $M_i$.
            In the case where $\partial D$ is essential in $\partial M_i$, then the operation (c) is the move compressing the boundary of $M_i$, i.e., the operation (c) corresponds to the operation \thelem.2. 
            In the case where $\partial D$ is not essential in $\partial M_i$, if $D$ is essential in $M_i$, then the operation (c) is splitting using $D$, i.e., the operation (c) corresponds to the operation \thelem.1, otherwise, the operation (c) is the move adding a $3$-ball component, i.e., the operation (c) corresponds to the operation \thelem.3.
        \item We consider the operation (d).
            Let $S$ denote the $2$-sphere cutting open $M_i$ in the operation (d).
            If $S$ is essential in $M_i$, then the operation (d) is splitting using $S$, i.e., the operation (d) corresponds to the operation \thelem.1, otherwise, the operation (d) is the move adding a $3$-sphere component, i.e., the operation (d) corresponds to the operation \thelem.1.
    \end{itemize}
    Therefore, the operations (a)--(d) correspond to the operations \thelem.1--\thelem.5 in the statement of Lemma \ref{lem:des}. \qed
    
    As with the proof of Lemma \ref{lem:V'_E'}, we define the types of bigon $B$ in $\mathcal{C}_i$ containing singular points as follows:
    \begin{enumerate}
        \renewcommand{\labelenumi}{(\Roman{enumi})}
        \item $B \subset \partial \mathcal{C}_i$, $e_1 = e_2$, and $B$ contains one singular point (Figure \ref{fig:types}\subref{fig:type1}),
        \item $B \subset \partial \mathcal{C}_i$, $e_1 = e_2$, and $B$ contains two singular points (Figure \ref{fig:types}\subref{fig:type2}), 
        \item $\partial B \subset \partial \mathcal{C}_i$, $v_1 \neq v_2$, and $e_1 \subset \mathcal{S}_i^k$ and $e_2 \subset \mathcal{S}_i^l$ (Figure \ref{fig:types}\subref{fig:type3}),
        \item  $\partial B \subset \partial \mathcal{C}_i$, $v_1 = v_2$, and $e_1 \subset \mathcal{S}_i^k$ and $e_2 \subset \mathcal{S}_i^l$ (Figure \ref{fig:types}\subref{fig:type4}),
    \end{enumerate}
    where $e_1$ and $e_2$ are the edges of $B$ and $k \neq l$.
    \begin{claim}
    We consider applying the move $\overline{a_i}$ and designularization to $M_i$.
    Then, if $a_i$ is not an atomic move flatting a bigon of type (I), (II), (III), or (IV) in $\mathcal{C}_i$ to an edge, then the resulting singular $3$-manifold does not depend on the order of $\overline{a_i}$ and desingularization.
    \end{claim}
    \noindent {\it Proof of Claim B.}
    The operation $\overline{a_i}$ is the move doing nothing or one of the operations (a)--(e) by Lemma \ref{lem:crush_annulus}.
    We prove Claim B by dividing into these cases:
    \begin{itemize}
        \item If $\overline{a_i}$ is the move doing nothing, then it is clear that the resulting singular $3$-manifold does not depend on the order of $\overline{a_i}$ and desingularization.
        \item If $\overline{a_i}$ is the move removing a $3$-ball component or a $3$-sphere component, then the resulting singular $3$-manifold does not depend on the order of $\overline{a_i}$ and desingularization since a $3$-ball component and a $3$-sphere component in $M_i$ does not contain singular points.
        \item In the case where $\overline{a_i}$ is the move filling a $2$-sphere boundary $S$ in $\partial M_i$ with a $3$-ball, $S$ contains no singular points since an atomic move $a_i$ is not the move flatting a bigon of type (I) or (II) by the assumption.
        Therefore, in this case, the resulting singular $3$-manifold does not depend on the order of $\overline{a_i}$ and desingularization.
        \item We consider the case $\overline{a_i}$ is the move cutting open $M_i$ along a properly embedded disk $D$ in $M_i$.
        If $D$ contains no singular points, then the resulting singular $3$-manifold does not depend on the order of desingularization and $\overline{a_i}$.
        We suppose that $D$ contains singular points.
        Let $M'_i$ denote the $3$-manifold obtained from $M_i$ by desingularization at first, then the properly embedded disk $D'$ in $M'_i$ is obtained from $D$.
        The $3$-manifold obtained by cutting open $M'_i$ along $D'$ is homeomorphic to the $3$-manifold obtained from $M_i$ by cutting open $M_i$ along $D$ and then desingularization as shown Figure \ref{fig:cutting_desingularlization}.
        \begin{figure}
            \centering
            \input{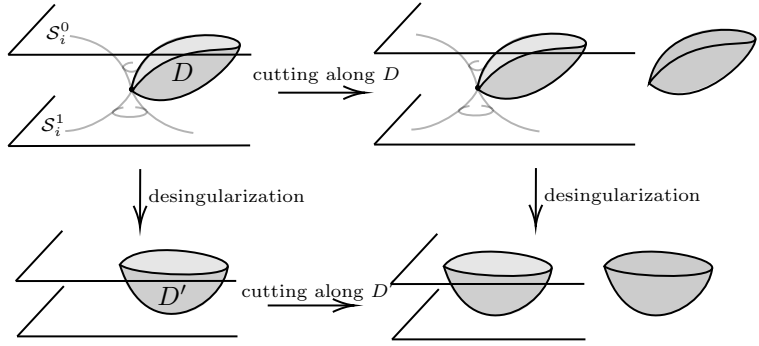}
            \caption{The operation cutting along a disk and desingularization}
            \label{fig:cutting_desingularlization}
        \end{figure}
        \item Suppose that $\overline{a_i}$ is the move cutting open $M_i$ along an embedded $2$-sphere in $M_i$ and filling the resulting boundaries with $3$-balls.
        In this case, we can prove the resulting singular $3$-manifold does not depend on the order of $\overline{a_i}$ and desingularization in the same way of the proof of the case where $\overline{a_i}$ is the move cutting open $M_i$ along an properly embedded disk.
        \item We consider the case where $\overline{a_i}$ is the move cutting open $M_i$ along a bigon $B$ which satisfies the following conditions:
        \begin{itemize}
            \item $\partial B \cap \mathcal{S}_i^{0} \neq \emptyset$,
            \item $\partial B \cap \mathcal{S}_i^{1} \neq \emptyset$,
            \item each of $v_1$ and $v_2$ is identified with a singular point, where $v_1$ and $v_2$ are the vertices of $B$.
        \end{itemize}
        In this case, the bigon in $\mathcal{C}_i$ corresponding to $B$ is a bigon of type (III) or (IV).
        This contradicts that $a_i$ is not an atomic move flatting a bigon of types (I), (II), (III), or (IV) to an edge.
    \end{itemize}
    \qed
    
    From Lemma \ref{lem:V'_E'}, $V'$ and $E'$ defined in Definition \ref{def:VandE} satisfy one of the following conditions:
    \begin{itemize}
        \item $|V'| = 2$ and $|E'| = 0$,
        \item $|V'| = 1$ and $|E'| = 0$,
        \item $|V'| = 1$, $|E'| = 2$, and each component of $E'$ is a circle constructed by one edge,
        \item $|V'| = 0$, $|E'| = 2$, and each component of $E'$ is a circle constructed by one edge,
        \item $|V'| = 0$, $|E'| = 4$, and each component of $E'$ is a circle constructed by one edge,
        \item $|V'| = 0$, $|E'| = 1$, and the component of $E'$ is a circle constructed by two edges,
        \item $|V'| = 0$ and $|E'| = 0$.
    \end{itemize}
    We prove Lemma \ref{lem:des} by dividing into these cases.
    
        \vspace{5pt}
        \noindent \textbf{The case where $|V'| = 2$ and $|E'| = 0$}\\
        \indent Suppose that $|V'| = 2$ and $|E'| = 0$.
        $M_0$ is obtained from $M$ by cutting open $M$ along $A$ and shrinking each copies of $A$ to a point.
        From the assumption, $M'$ has two singular points, and $M''$ is obtained from $M'$ by desingularization at each singular point.
        
        Let $N_0$ be the $3$-manifold obtained from $M_0$ by desingularizing at each singular point in $M_0$, and $N_0 \to N_1 \to \cdots \to N_n$ be the sequence of $3$-manifolds, where $N_{i+1}$ is obtained from $N_i$ by applying the operation $\overline{a_i}$.
        For each $\mathcal{C}_i$ $(0 \leq i \leq n)$, there are no bigons of types (I), (II), (III), or (IV) from the assumption, hence, we see that $\overline{a_i}$ is one of the operations (a)--(d) by the proof of Lemma \ref{lem:crush_annulus}, and so $\overline{a_i}$ is one of the operations \thelem.1--\thelem.5 by Claim A.
        
        Since we can swap the order of desingularization and each of the operations (a)--(d) if there are no bigons of types (I), (II), (III), or (IV) in each $\mathcal{C}_i$ by Claim B, $N_n$ is homeomorphic to $M''$.
        Now, $N_0$ is obtained from $M$ by cutting along $A$, shrinking each copy of $A$ to a point, and desingularizing at each singular point.
        This move is equal to destabilization using $A$.
        Therefore, $N_0$ is obtained from $M$ by destabilization using $A$, and $N_{i+1}$ is obtained from $N_i$ by the operations \thelem.1--\thelem.5 for any $i\ (0 \leq i < n)$.

        \vspace{5pt}
        \noindent \textbf{The case where $|V'| = 1$ and $|E'| = 0$}\\
        \indent We suppose that $|V'| = 1$ and $|E'| = 0$.
        This assumption implies that there is a type (I) bigon $B$ in a cell-decomposition $\mathcal{C}_i$ and $M_{i+1}$ is obtained from $M_i$ by filling the $2$-sphere boundary $S$ in $M_i$ corresponding to $B$ with a $3$-ball.
        By the proof of Lemma \ref{lem:V'_E'}, for each cell-decomposition $\mathcal{C}_j$ $(0 \leq j \leq n)$, there are no bigons of type (I), (II), (III), or (IV) other than $B$.
        Let $N_i$ be the $3$-manifold obtained from $M_i$ by desingularizing at the singular point contained in the $2$-sphere boundary $S$ in $M_i$.
        For each $j$ which is greater than $i$, let $N_{j}$ be the $3$-manifold obtained from $N_{j-1}$ by the operation $\overline{a_{j-1}}$.
        We see that $N_{i+1} \simeq M_{i+1}$ as shown in Figure \ref{fig:case_0_0}, and so $N_j \simeq M_j$ if $j \geq i$.
        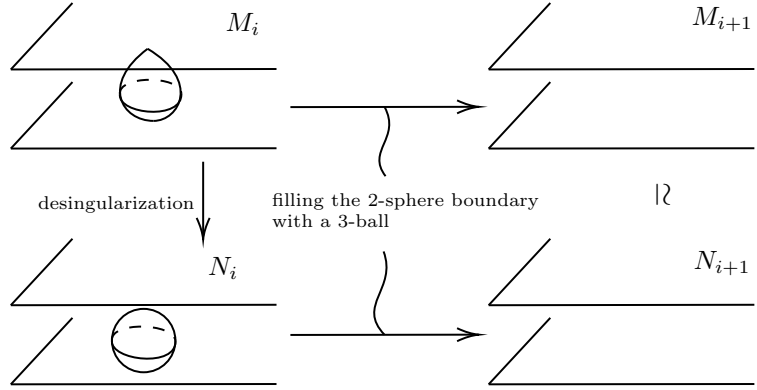
\begin{figure}
            \centering
            \tikzset{every picture/.style={line width=0.75pt}} 

\begin{tikzpicture}[x=0.75pt,y=0.75pt,yscale=-1,xscale=1]

\draw    (11,40.43) -- (145,40.2) ;
\draw    (11,40.43) -- (42.09,6.88) ;
\draw    (11,80.35) -- (145,80.2) ;
\draw    (11,80.35) -- (42.09,46.8) ;
\draw    (80,30) .. controls (56,48) and (67,66.39) .. (83,66.39) ;
\draw    (80,30) .. controls (107,46) and (96,66.39) .. (83,66.39) ;
\draw    (66,53) .. controls (65,63.2) and (98,66.39) .. (97,52.39) ;
\draw  [dash pattern={on 4.5pt off 4.5pt}]  (66,53) .. controls (65,43.2) and (96,43.39) .. (97,52.39) ;
\draw    (252,40.43) -- (386,40.2) ;
\draw    (252,40.43) -- (283.09,6.88) ;
\draw    (252,80.35) -- (386,80.2) ;
\draw    (252,80.35) -- (283.09,46.8) ;
\draw    (252,159.43) -- (386,159.2) ;
\draw    (252,159.43) -- (283.09,125.88) ;
\draw    (252,199.35) -- (386,199.2) ;
\draw    (252,199.35) -- (283.09,165.8) ;
\draw    (11,159.43) -- (145,159.2) ;
\draw    (11,159.43) -- (42.09,125.88) ;
\draw    (11,199.35) -- (145,199.2) ;
\draw    (11,199.35) -- (42.09,165.8) ;
\draw   (62,177.27) .. controls (62,168.44) and (69.16,161.27) .. (78,161.27) .. controls (86.84,161.27) and (94,168.44) .. (94,177.27) .. controls (94,186.11) and (86.84,193.27) .. (78,193.27) .. controls (69.16,193.27) and (62,186.11) .. (62,177.27) -- cycle ;
\draw    (62,177.27) .. controls (61,187.47) and (95,191.27) .. (94,177.27) ;
\draw  [dash pattern={on 4.5pt off 4.5pt}]  (62,177.27) .. controls (61,167.47) and (93,168.27) .. (94,177.27) ;
\draw    (108,87) -- (108,124.29) ;
\draw [shift={(108,126.29)}, rotate = 270] [color={rgb, 255:red, 0; green, 0; blue, 0 }  ][line width=0.75]    (10.93,-3.29) .. controls (6.95,-1.4) and (3.31,-0.3) .. (0,0) .. controls (3.31,0.3) and (6.95,1.4) .. (10.93,3.29)   ;
\draw    (152,59.43) -- (245,59.29) ;
\draw [shift={(247,59.29)}, rotate = 179.91] [color={rgb, 255:red, 0; green, 0; blue, 0 }  ][line width=0.75]    (10.93,-3.29) .. controls (6.95,-1.4) and (3.31,-0.3) .. (0,0) .. controls (3.31,0.3) and (6.95,1.4) .. (10.93,3.29)   ;
\draw    (152,174.43) -- (245,174.29) ;
\draw [shift={(247,174.29)}, rotate = 179.91] [color={rgb, 255:red, 0; green, 0; blue, 0 }  ][line width=0.75]    (10.93,-3.29) .. controls (6.95,-1.4) and (3.31,-0.3) .. (0,0) .. controls (3.31,0.3) and (6.95,1.4) .. (10.93,3.29)   ;
\draw    (199.5,59.36) .. controls (208,75.29) and (190,81.29) .. (200,94.29) ;
\draw    (199,132.29) .. controls (206,151.29) and (185,159.29) .. (199.5,174.36) ;

\draw (64.29,108.29) node  [font=\scriptsize] [align=left] {desingularization};
\draw (210.09,111.29) node  [font=\scriptsize] [align=left] {filling the $\displaystyle 2$-sphere boundary\\with a $\displaystyle 3$-ball};
\draw (118,10.4) node [anchor=north west][inner sep=0.75pt]    {$M_{i}$};
\draw (353,7.4) node [anchor=north west][inner sep=0.75pt]    {$M_{i+1}$};
\draw (109,133.4) node [anchor=north west][inner sep=0.75pt]    {$N_{i}$};
\draw (354,130.4) node [anchor=north west][inner sep=0.75pt]    {$N_{i+1}$};
\draw (345.1,95.5) node [anchor=north west][inner sep=0.75pt]  [font=\large,rotate=-90]  {$\simeq $};

\end{tikzpicture}
            \caption{The case where there is a boundary $2$-sphere containing two singular points}
            \label{fig:case_0_0}
        \end{figure}
        Now, we have the sequence of $3$-manifolds and singular $3$-manifolds
        \[
            M \to M_0 \to \cdots \to M_i \xrightarrow{\text{desingularization}} N_i \to N_{i+1} \to \cdots \to N_n \simeq M' \simeq M''.
        \]
        
        Let $N_0$ be the $3$-manifold obtained from $M_0$ by desingularizing at each singular point in $M_0$ and we inductively define $N_{j+1}$ as the $3$-manifold obtained from $N_{j}$ by the operation $\overline{a_{j-1}}$ for each $j \ (0 \leq j < i-1)$.
        $N_0$ is obtained from $M$ by cutting $M$ along $A$, shrinking each copy of $A$ to a point, and desingularizing at each singular point.
        Therefore, $N_0$ is obtained from $M$ by destabilizing using $A$.
        We can swap the order of applying desingularization and the operations other than $\overline{a_i}$ since there are no bigons of types (I), (II), (III), or (IV) by Claim B. 
        Thus, there is a sequence of $3$-manifolds $M \to N_0 \to \cdots \to N_n \simeq M_n$, and $N_{j+1}$ is obtained from $N_{j}$ by the operations \thelem.1--\thelem.6 in the statement of Lemma \ref{lem:des} by Claim A.

        \vspace{5pt}
        \noindent \textbf{The case where $|V'| = 0$ and $|E'| = 0$}\\
        \indent We suppose that $|V'| = 0$ and $|E'| = 0$.
        This assumption implies that either of the following situations holds:
        \begin{itemize}
            \item There are two cell-deconpositions $\mathcal{C}_i$ and $\mathcal{C}_j$ containing type (I) bigons $B_i \subset \mathcal{C}_i$ and $B_j \subset \mathcal{C}_j$, and each of $M_{i+1}$ and $M_{j+1}$ is obtained by filling a $2$-sphere boundary containing a singular point with a $3$-ball.
            \item There is a cell-decomposition $\mathcal{C}_i$ containing a type (II) bigon $B$ and $M_{i+1}$ is obtained from $M_i$ by filling a $2$-sphere boundary containing the two singular points with a $3$-ball.
        \end{itemize}
        In either case, we can show that Lemma \ref{lem:des} holds by the same argument of the case where $V' = 1$ and $E' = 0$.

        \vspace{5pt}
        \noindent \textbf{The case where $|V'| = 0$, $|E'| = 1$, and the component of $E'$ is a circle constructed by two edges}\\
        \indent Suppose that $|V'| = 0$, $|E'| = 1$ and the component of $E'$ is a circle constructed by two edges.
        In this case, there is a cell-decomposition $\mathcal{C}_i$ containing a bigon $B$ of type (III), and $M_{i+1}$ is obtained from $M_i$ by cutting $M_i$ along a properly embedded disk corresponding to $B$.
        In addition, for each cell-decomposition $\mathcal{C}_j$ $(0 \leq j \leq n)$, there are no bigons of types (I), (II), (III), or (IV) other than $B$ by the proof of Lemma \ref{lem:V'_E'}.
        
        From the assumption, $\mathcal{T}''$ is obtained from $\mathcal{T}'$ by Operation \ref{ope:eliminate_loop2}.
        We define a similar operation for the underlying $3$-manifold $M_j$ of a cell-decomposition $\mathcal{C}_{j}$ which satisfies $|V_j| = 0$, $|E_j| = 1$, and the component of $E_j$ is a circle constructed by two edges.
        Let $v_1$ and $v_2$ denote the vertices in the circle in $E_j$, and let $e_1$ and $e_2$ denote the edges in the circle in $E_j$.
        We denote the points and the arcs in $M_{j}$ corresponding to $v_1, v_2, e_1,$ and $e_2$ by $\overline{v_1}, \overline{v_2}, \overline{e_1},$ and $\overline{e_2}$, respectively.
        We define Operation \ref{ope:eliminate_loop2}$'$ as the operation gluing the two edges $\overline{e_1}$ and $\overline{e_2}$ and then desingularizing at the two points $\overline{v_1}$ and $\overline{v_2}$.
        Then, $M''$ is obtained from $M'$ by Operation \ref{ope:eliminate_loop2}$'$.

        \begin{claim}
            We consider applying the move $\overline{a_j}$ and Operation \ref{ope:eliminate_loop2}$'$ to $M_j$.
            If $a_j$ is not an atomic move flatting a bigon of type (I), (II), (III), or (IV) in $\mathcal{C}_j$, then the resulting $3$-manifold does not depend on the order of $\overline{a_j}$ and Operation \ref{ope:eliminate_loop2}$'$.
        \end{claim}
        \noindent {\it Proof of Claim C.}
        The operation $\overline{a_j}$ is one of operations (a)--(d) because $a_j$ is not an atomic move flatting a bigon of type (I), (II), (III), or (IV) in $\mathcal{C}_j$.
        If $\overline{a_j}$ is an operation (c) or (d), then we can prove that the resulting $3$-manifold does not depend on the order of $\overline{a_j}$ and Operation \ref{ope:eliminate_loop2}$'$ in the same way of the proof of Claim B.
        Furthermore, if $\overline{a_j}$ is the operation (a) removing a $3$-sphere component of $M_j$, then we can swap the order of $\overline{a_j}$ and Operation \ref{ope:eliminate_loop2}$'$ since the $3$-sphere component has no boundaries.
        We suppose that $\overline{a_j}$ is an operation (a) removing a $3$-ball or an operation (b), i.e., the atomic move $a_j$ flats a bigon $B \subset \partial \mathcal{C}_j$ whose edges are identified.
        Let $S$ denote the $2$-sphere boundary constructed by $B$.
        If $S$ does not contain the circle in $E_j$, then we can swap the operation $\overline{a_j}$ and Operation \ref{ope:eliminate_loop2}$'$.
        If $S$ contains the circle in $E_j$, then the circle divides $S$ into two disks.
        This contradicts that $S$ is constructed by just one bigon $B$.
        Therefore, the resulting $3$-manifold does not depend on the order of $\overline{a_j}$ and Operation \ref{ope:eliminate_loop2}$'$.
        \qed
        
        Let $N_{i+1}$ denote the $3$-manifold obtained from $M_{i+1}$ by Operation \ref{ope:eliminate_loop2}$'$, and for any $j$ which is greater than $i$, we inductively define $N_{j}$ as the $3$-manifold obtained from $N_{j-1}$ by the operation $\overline{a_{j-1}}$.
        Then, $N_n$ is homeomorphic to $M''$ because the order of Operation \ref{ope:eliminate_loop2}$'$ and the operation $\overline{a_j}$ can be swapped if $j \neq i$ by Claim C.
        Furthermore, $N_{i+1}$ is homeomorphic to the $3$-manifold obtained from $M_i$ by desingularizing at each singular point as shown Figure \ref{fig:case_0_1}, and so we apply desingularization to $M_i$ instead of the two operations, $\overline{a_{i}}$ and Operation \ref{ope:eliminate_loop2}$'$.
        \begin{figure}
            \centering
            \input{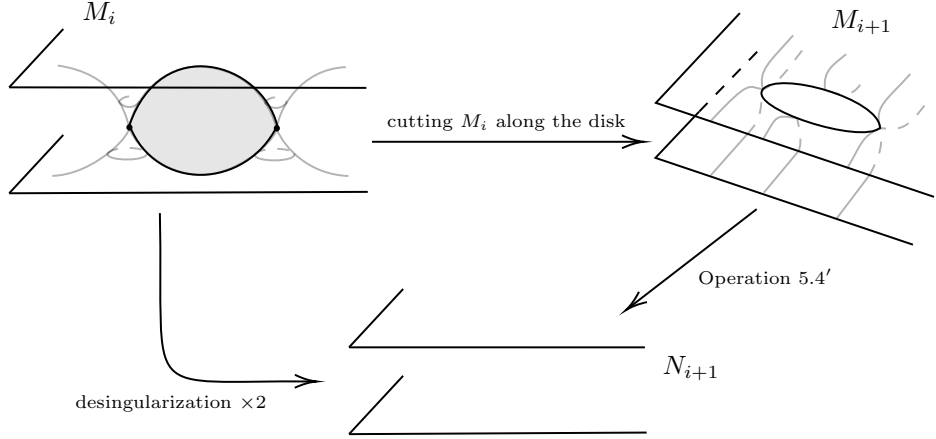}
            \caption{The case where there is a disk containing two singular points and which intersects $\mathcal{S}_i^0$ and $\mathcal{S}_i^1$}
            \label{fig:case_0_1}
        \end{figure}
        Now, we have the sequence of $3$-manifolds and singular $3$-manifolds
        \[
            M \to M_0 \to \cdots \to M_i \xrightarrow{\text{desingularizations}} N_{i+1} \to \cdots \to N_n \simeq M''.
        \]
        
        Let $N_0$ be the $3$-manifold obtained from $M_0$ by desingularizing at each point in $M_0$, and for any $j\ (j\leq i)$, we inductively define $N_{j}$ as the $3$-manifold obtained from $N_{j-1}$ by the operation $\overline{a_{j-1}}$.  
        For any $j$ other than $i$, the resulting $3$-manifold does not depend on the order of desingularization and the operation $\overline{a_j}$ by Claim B.
        Therefore, we have the following sequence of the $3$-manifolds
        \[
            M \to N_0 \to \cdots N_i \simeq N_{i+1} \to \cdots \to N_n \simeq M''.
        \]
        If $j \neq 0$, then $N_j$ is obtained from $N_{j-1}$ by one of the operations \thelem.1--\thelem.5 by Claim A.
        $N_0$ is obtained from $M$ by cutting along the annulus $A$ in $M$, shrinking each copy to a point, and desingularizing at each singular point.
        Thus, $N_0$ is obtained from $M$ by destabilization using $A$.
        Therefore, there is a sequence of $3$-manifolds $M \to N_0 \to \cdots \to N_n \simeq M''$, and $N_{j+1}$ is obtained from $N_{j}$ by one of the operations \thelem.1--\thelem.6.

        \vspace{5pt}
        \noindent \textbf{The case where $|V'| = 1$, $|E'| = 2$, and each component of $E'$ is a circle constructed by one edge}\\
        \indent Suppose that $|V'| = 1$, $|E'| = 2$, and each component of $E'$ is a circle constructed by one edge.
        In this case, $\mathcal{T}''$ is obtained from $\mathcal{T}'$ by Operation \ref{ope:eliminate_loop1} at each circle in $E'$ and then desingularizing at the singular point in $\mathcal{T}'$.
        We define a similar operation for the underlying singular $3$-manifold of a cell-decomposition $\mathcal{C}_j$ which has circles constructed by one edge in $E_j$.
        Let $c$ be a circle constructed by one edge in $\mathcal{C}_j$ and $\overline{c}$ be the circle in $M_j$ corresponding to $c$.
        We call the operation shrinking $c$ to a point and then desigularizing at the point Operation \ref{ope:eliminate_loop1}$'$.
        As with Claim C, the resulting $3$-manifold does not depend on the order of Operation \ref{ope:eliminate_loop1}$'$ and the operation $\overline{a_j}$.
        
        From the assumption, there is a cell-decomposition $\mathcal{C}_i$ containing a bigon $B$ of type (IV) and $\mathcal{C}_{i+1}$ is obtained from $\mathcal{C}_i$ by flatting $B$ to an edge.
        In addition, there are no bigons of types (I), (II), (III), or (IV) in for each $\mathcal{C}_j$ $(0 \leq j \leq n)$ other than $B$ by the proof of Lemma \ref{lem:V'_E'}.
        Let $N_{i+1}$ be the $3$-manifold obtained from $M_{i+1}$ by Operation \ref{ope:eliminate_loop1}$'$ and desingularization at the singular point in $M_{i+1}$, and for any $j$ which is greater than $i$, we define $N_{j+1}$ as the $3$-manifold obtained from $N_{j}$ by the operation $\overline{a_{j}}$.
        We can swap the order of the operation $\overline{a_j}$ and Operation \ref{ope:eliminate_loop1}$'$, and we can also swap the order of $\overline{a_j}$ and desingularization by Claim B.
        Therefore, we have $M'' \simeq N_n$.
        
        Let $N_i$ denote the $3$-manifold obtained from $M_i$ by desingularizing at each singular point in $M_i$.
        \begin{claim}
            The $3$-manifold $N_{i+1}$ is obtained from $N_i$ by the operation \thelem.6 or \thelem.7 in the statement of Lemma \ref{lem:des}.
        \end{claim}
        \noindent{\it Proof of Claim \Alph{claim}.}
        We obtain the properly embedded annulus $B'$ in $N_i$ from $B$ by desingularization.
        Let $b'_0$ and $b'_1$ denote the boundary component of $B'$.
        First, we prove that $b'_0$ bounds a disk in $\partial N_i$ if and only if $b'_1$ bounds a disk in $\partial N_i$.
        For any cell-decompositions $\mathcal{C}_j$ $(0 \leq j < i)$, there are no bigons of types (I), (II), (III), or (IV), and so $N_i$ is homeomorphic to the $3$-manifold obtained from $M_0$ by desingularizing at first and then applying the operations $\overline{a_j}$ $(0 \leq j <i)$ by Claim B.
        The destabilization using $A$ is the same as the operation cutting open $M$ along $A$, shrinking each copy of $A$ to a point, and desingularizing at each singular point.
        Now, we see that $N_i$ is obtained from $M$ by the operations \thelem.1--\thelem.6 by Claim A.
        $M$ is a $3$-manifold obtained from $\mathcal{S} \times I$ removing embedded open solid tori, and so we can obtain the component of $N_i$ containing $B'$ from $F \times I$ by removing embedded open solid tori and emnedded open $3$-balls, where $F$ is an orientable surface.
        Therefore, we can obtain $F \times I$ from the component of $N_i$ containing $B'$ by gluing solid tori and $3$-balls, and we obtain the properly embedded annulus $B''$ in $F \times I$.
        $\partial B''$ intersects with $F \times \{0\}$ and $F \times \{1\}$ because the bigon $B \subset \mathcal{C}_i$ is the type (IV), i.e., $\partial B$ intersects with $\mathcal{S}^0_i$ and $\mathcal{S}^1_i$.
        Let $f:\mathbb{S}^1 \times I \to F \times I$ be the embedding map of $B''$ and $p:F \times I \to F$ be the projection map.
        Then, $p \circ f : \mathbb{S}^1 \times I \to F$ is an isotopy from $p \circ f|_{\mathbb{S}^1 \times \{0\}}$ to $p \circ f|_{\mathbb{S}^1 \times \{1\}}$.
        Hence, $B'' \cap F \times \{0\}$ bounds a disk in the boundary of $F \times I$ if and only if $B'' \cap F \times \{1\}$ also bounds a disk in the boundary of $F \times I$.
        Thereby, $b'_0$ bounds a disk in $\partial N_i$ if and only if $b'_1$ bounds a disk in $\partial N_i$.
        
        We show that if $b'_0$ and $b'_1$ bound disks in $\partial N_i$, then $N_{i+1}$ is obtained from $N_i$ by the operation \thelem.7.
        Let $D^0$ and $D^1$ denote the disks in the boundary of $N_i$ bounded by $b'_0$ and $b'_1$.
        Then, $B' \cup D^0 \cup D^1$ is a $2$-sphere, and the boundary of the regular neighborhood $\partial N(B' \cup D^0 \cup D^1)$ is a disjoint union of a $2$-sphere and an annulus.
        We denote the $2$-sphere component of $\partial N(B' \cup D^0 \cup D^1)$ by $S$.
        $S$ splits the component of $N_i$ containing $S$ into two components.
        One of the components is a $3$-manifold obtained by removing an open $3$-ball and an open regular neighborhood of $\hat{D'_1}$ from $F \times I$, where $F$ is a closed orientable surface and $\hat{D'_1}$ is a sublink of $\hat{D}$ which may be the empty set, and we denote this component by $X_1$.
        The other component is a component obtained by removing an open regular neighborhood of $\hat{D'_2}$ from a $3$-ball, where $\hat{D'_2}$ is a sublink of $\hat{D}$ which may be the empty set, and we denote this component by $X_2$.
        
        $M_{i+1}$ has two components which has a circle corresponding to a circle in $E_{i+1}$.
        Let $Y_1$ and $Y_2$ be the components in $N_{i+1}$ which are obtained from the components corresponding to the components containing a circle in $E_{i+1}$.
        One of $Y_1$ and $Y_2$ is homeomorphic to the $3$-manifold obtained from $X_1$ by filling the copy of $S$ with a $3$-ball, 
        and the other component is homeomorphic to the $3$-manifold obtained by gluing the copy of $S$ to the boundary obtained by removing a $3$-ball from $\mathbb{S}^2 \times I$ as shown in Figure \ref{fig:boundary_bounds_disk}.
        Therefore, $N_{i+1}$ is obtained from $N_i$ by the operation \thelem.7.
        \begin{figure}
            \centering
            \input{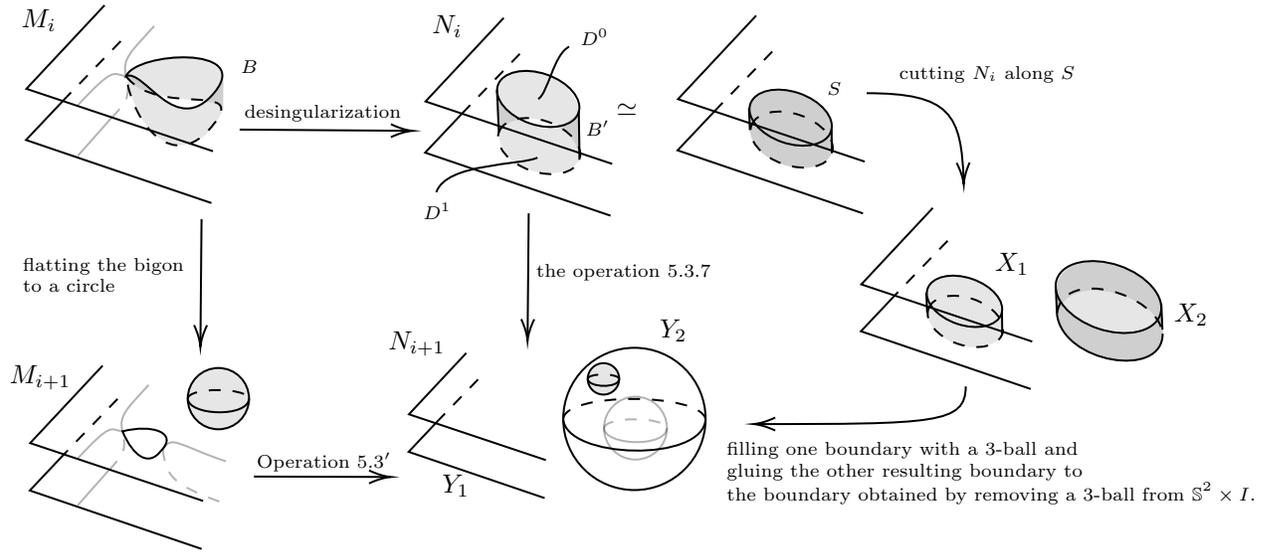}
            \caption{The case where each component of the boundary of $B'$ bounds a disk in the boundary of $N_i$}
            \label{fig:boundary_bounds_disk}
        \end{figure}
        
        Next, we prove that if $b'_0$ and $b'_1$ are not bound disks in $\partial N_i$, then $N_{i+1}$ is obtained from $N_i$ by the destabilization using $B'$, i.e., the operation \thelem.6. 
        $M_{i+1}$ is obtained from $M_i$ by flatting the bigon in $M_i$ corresponding to $B$, and $N_{i+1}$ is obtained from $M_{i+1}$ by shrinking the two circles in $M_i$ corresponding to $E_{i+1}$ to points and then desingularizing at the two singular points, i.e., Operation \ref{ope:eliminate_loop1}$'$.
        On the other hand, $M_{i+1}$ is homeomorphic to the $3$-manifold obtained from $N_i$ by flatting the annulus $B'$ to a circle, thus, $N_{i+1}$ is obtained from $N_i$ by the operation cutting $N_i$ along $B'$, shinking each copy of $B'$ to a point, and desingularizing at the two singular points.
        This operation is equal to destabilization using $B'$ as shown in Figure \ref{fig:boundary_does_not_bound_disk}.
        Therefore, $N_{i+1}$ is obtained from $N_i$ by destabilization using $B'$.
        \qed
        \begin{figure}
            \centering
            \input{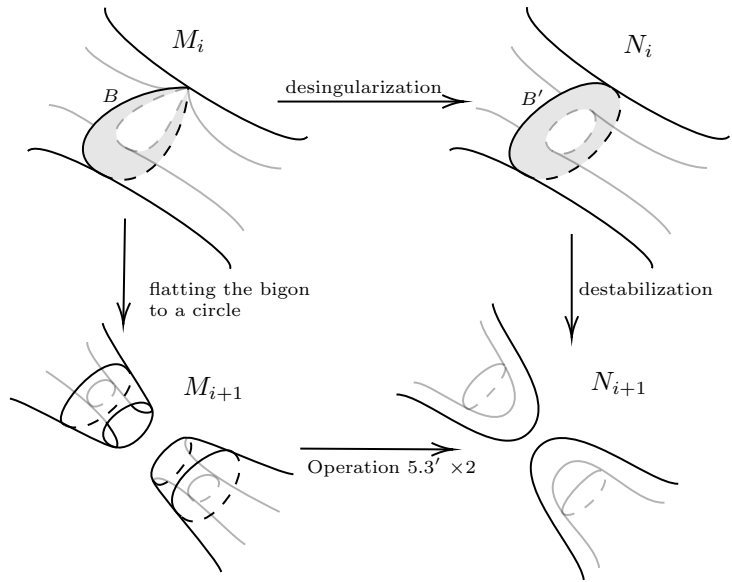}
            \caption{The case where each component of the boundary of $B'$ does not bound a disk in the boundary of $N_i$}
            \label{fig:boundary_does_not_bound_disk}
        \end{figure}

        Let $N_0$ be the $3$-manifold obtained from $M_0$ by desingularizing at the singular points, and for any $j\ (0 \leq j < i-1)$, let $N_{j+1}$ be the $3$-manifold obtained from $N_{j}$ by the operation $\overline{a_j}$.
        The resulting $3$-manifold does not depend on the order of desingularizing and the operations $\overline{a_j}$ $(0 \leq j < i-1)$ by Claim B, therefore, there is a sequence of $3$-manifolds $M \to N_0 \to \cdots \to N_n \simeq M''$ and for any $j$ $(0 \leq j < n)$, $N_{j+1}$ is obtained from $N_j$ by one of the operations \thelem.1--\thelem.7.

        \vspace{5pt}
        \noindent \textbf{The case where $|V'| = 0$, $|E'| = 2$, and each component of $E'$ is a circle constructed by one edge}\\
        \indent Suppose that $|V'| = 0$, $|E'| = 2$.
        This implies that there is a type (I) bigon $B_i \subset \mathcal{C}_i$ and a type (IV) bigon $B_j \subset \mathcal{C}_j$, and $\mathcal{C}_{i+1}$ and $\mathcal{C}_{j+1}$ is obtained from $\mathcal{C}_i$ and $\mathcal{C}_j$ by flatting the bigons in $M_i$ and $M_j$ corresponding to $B_i$ and $B_j$.
        In this case, we can show Lemma \ref{lem:des} by the same argument of the proof of the case where $V' = 1$ and $E' = 0$ and the case where $V' = 1$ and $E' = 2$.

        \vspace{5pt}
        \noindent \textbf{The case where $|V'| = 0$, $|E'| = 4$, and each component of $E'$ is a circle constructed by one edge}\\
        \indent We suppose that $|V'| = 0$, $|E'| = 4$, and each component of $E'$ is a circle constructed by one edge.
        In this case, $\mathcal{T}''$ is obtained from $\mathcal{T}'$ by Operation \ref{ope:eliminate_loop1} at each circle in $E'$.
        From the assumption, there are two cell-decompositions $\mathcal{C}_i$ and $\mathcal{C}_j$ containing a type (IV) bigon and the atomic moves $a_i$ and $a_j$ are the move flatting a type (IV) bigon to an edge.
        We suppose that $i < j$, and let $B_i$ and $B_j$ denote the bigons in $M_i$ and $M_j$ corresponding to the type (IV) bigons in $\mathcal{C}_i$ and $\mathcal{C}_j$, respectively.
        
        If there are no annuli in $\partial M_j$ whose boundary contains two points identified with a singular point and whose boundary is $\partial B_j$, then we can prove that there is a sequence of $3$-manifolds $M \to N_0 \to \cdots \to N_n \simeq M''$ and $N_{k+1}$ is obtained from $N_k$ by one of the operations \thelem.1--\thelem.7 as the same argument of the case where $|V'| = 1$, $|E'| = 2$.
        
        We suppose that there is an annulus $F$ in $M_j$ whose boundary contains two points identified with a singular point and whose boundary is $\partial B_j$.
        In this case, $F$ contains a circle in $E_j$.
        Let $c_1$ and $c_2$ denote the circles obtained from $B_i$ in $\mathcal{C}_j$, and suppose that $c_1$ is in $F$.
        If $B_j$ and $c_1$ has no intersections, then the resulting $3$-manifold does not depend on the order of the operation $\overline{a_j}$ and Operation \ref{ope:eliminate_loop1}$'$ for $c_1$ as shown Figure \ref{fig:case_0_4}.
        \begin{figure}
            \centering
            \input{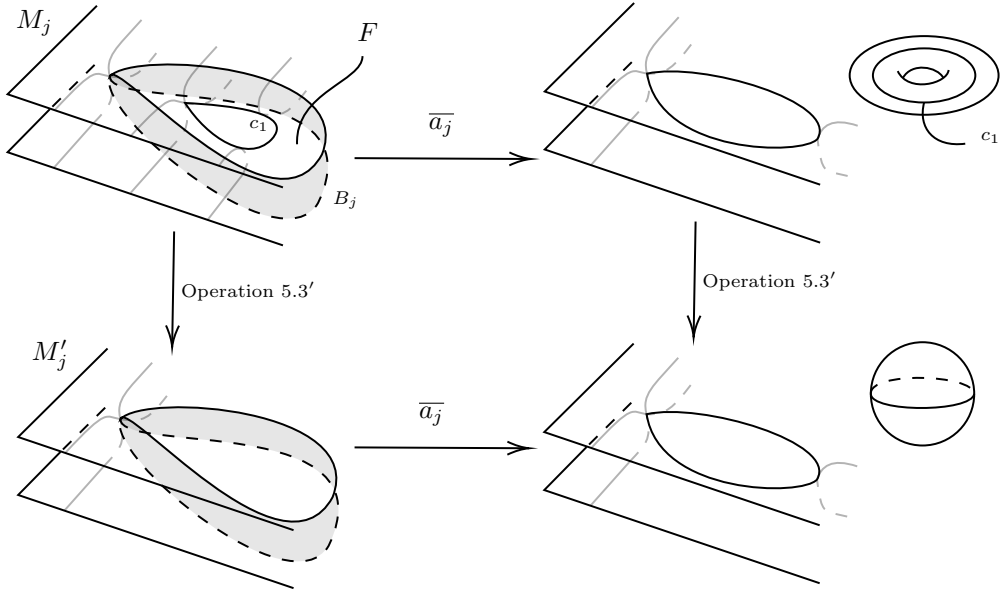}
            \caption{The $3$-manifold obtained from $M_j$ by the operation $\overline{a_j}$ and Operation 4.1'}
            \label{fig:case_0_4}
        \end{figure}
        
        We prove that $B_j$ and $c_1$ has no intersections.
        Let $s$ be the singular point contained in $B_j$.
        We see that $s \in V_j$ and $c_1 \in E_j$.
        From the definition of $V_j$ and $E_j$, $V_j \cap E_j = \emptyset$, thus, $s \cap c_1 = \emptyset$.
        If the edges of $B_j$ are identified to $c_1$, then $s \cap c_1 \neq \emptyset$, and so $B_j$ and $c_1$ has no intersections.
        Therefore, the resulting $3$-manifold does not depend on the order of the operation $\overline{a_j}$ and Operation \ref{ope:eliminate_loop1}$'$ for $c_1$.
        
        Let $M'_j$ be the singular $3$-manifold obtained from $M_j$ by applying Operation \ref{ope:eliminate_loop1}$'$ at $c_1$ and $c_2$.
        We denote the bigon obtained from $B_j$ in $M'_j$ by $B'_j$.
        Then, each boundary component of $B'_j$ bounds a disk in the boundry of $M'_j$.
        Therefore, as the same argument of the case where $|V'| = 1$, $|E'| = 2$, we can prove that there is a seaquence of $3$-manifolds $M \to N_0 \to \cdots \to N_n \simeq M''$ and for any $k$, $N_{k+1}$ is obtained from $N_k$ by one of the operations \thelem.1--\thelem.7.
\end{proof}

We can show the same result as Collorary \ref{cor:splitting_comp} for Operation \ref{ope:des}.
\begin{cor}\label{cor:des_comp}
    Suppose the same situation of Operation \ref{ope:des}.
    Let $\hat{D'}$ and $\hat{D''}$ be sublinks of $\hat{D}$.
    A component of the 3-manifold $M'$ obtained from $M$ by Operation \ref{ope:des} is one of the following:
    \begin{itemize}
        \item $\text{cl}(\mathcal{S}' \times I - N(\hat{D'})$, where $\mathcal{S}'$ is a closed orientable surface,
        \item $\mathbb{B}^3 - N(\hat{D''})$,
        \item $\mathbb{S}^3 - N(\hat{D''})$,
        \item a component removed some 3-balls from the above component.
    \end{itemize}
\end{cor}

Consider the case where $\hat{D}$ is a knot in a thickened closed orientable surface.
As the same argument of Corollary \ref{cor:splitting_comp_knot}, if the underlying $3$-manifold of the resulting triangulation of Operation \ref{ope:des} is the empty set, then $\hat{D}$ is the trivial knot.
Thus, the following corollary holds.
\begin{cor}\label{cor:des_comp_knot}
    Suppose that $\mathcal{S}$ is a closed orientable surface whose genus is not zero, $\hat{D}$ is a knot in $\mathcal{S} \times I$, and $\mathcal{T}$ is a triangulation of the exterior $M = \text{cl}(\mathcal{S} \times I - N(\hat{D}))$.
    If $F$ is a vertical essential normal annulus in $M$ with respect to $\mathcal{T}$, then the underlying $3$-manifold $M'$ of the resulting triangulation of Operation \ref{ope:splitting} is the empty set or one of follows:
    \begin{itemize}
        \item $\text{cl}(\mathcal{S}' \times I - N(\hat{D})$, where $\mathcal{S}'$ is a closed orientable surface,
        \item $\mathbb{B}^3 - N(\hat{D})$,
        \item $\mathbb{S}^3 - N(\hat{D})$.
    \end{itemize}
    Furthermore, if $M'$ is the empty set, then $\hat{D}$ is the trivial knot.
\end{cor}

In order to analyze the running time of Operation \ref{ope:des}, we consider the number of tetrahedra increased by Operation \ref{ope:des}.

\begin{lem}\label{lem:desing}
    Let $\mathcal{T}$ be a triangulation of the exterior $M$ of a link $\hat{D}$ in a thickened closed orientable surface and $A$ be a vertical essential normal surface in $M$ with respect to $\mathcal{T}$.
    Run Operation \ref{ope:des} using $A$ on $\mathcal{T}$, and let $\mathcal{T}'$ be the resulting triangulation.
    Suppose that $p$ is a singular point in $\mathcal{T}'$.
    $S$ denotes a boundary component of $\mathcal{T}'$ which contains $p$, and $|S|$ denotes the number of triangles in $S$.
    Then, the number of tetrahedra increased by the desingularization at $p$ is $\mathcal{O}(|S|)$.
\end{lem}
\begin{proof}
    Let $N$ denote the set of triangles which contain $p$ in $S$.
    For any triangle $t$ in $N$, $n_p(t)$ denotes the number of vertices of $t$ which are identified to $p$.
    Stretching $p$ to an edge, for any $2$-simplex $t$ in $N$,
    $n_p(t)$ tetrahedra are added to $\mathcal{T}'$ as shown in Figure \ref{fig:desing2}.
    $n_p(t)$ is at most three, so that this operation increases the number of tetrahedra by $\mathcal{O}(|S|)$.
\end{proof}
\begin{figure}[htbp]
    \centering
    \input{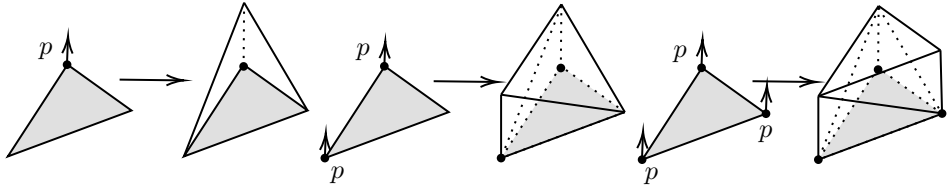}
    \caption{Desingularization at each $2$-simplex}
    \label{fig:desing2}
\end{figure}

\begin{cor}\label{cor:des_num}
    Let $\hat{D}$ be a link in a thickened closed orientable surface $\mathcal{S} \times I$ and $\mathcal{T}$ be a triangulation of the exterior $M = \text{cl}(\mathcal{S} \times I - N(\hat{D}))$.
    Let $\mathcal{S}^k$ $(k = 0 \text{ or } 1)$ denote the triangulation of $\mathcal{S}$ in the boundary of $\mathcal{T}$.
    Suppose that $|\mathcal{S}^k| \geq |\mathcal{S}^{1-k}|$, where $|\mathcal{S}^k|$ is the number of triangles in $\mathcal{S}^k$.  
    Then, the number of tetrahedra increased by Operation \ref{ope:des} is $\mathcal{O}(|\mathcal{S}^k|)$.
\end{cor}
\begin{proof}
    The crushing procedure does not increase the number of tetrahedra, therefore, $|\mathcal{T}'| \leq |\mathcal{T}|$.
    The number of triangles in a boundary component containing components in $V'$ or $E'$ is $\mathcal{O}(|\mathcal{S}^k|)$ because 
    such boundary component is obtained from $\mathcal{S}^0 \cup \mathcal{S}^1$ and the number of triangles in the boundary of $\mathcal{T}$ is not increased by the crushing procedure.
    Thus, the number of tetrahedra increased by desingularization is $\mathcal{O}(|\mathcal{S}^k|)$.
    
    Operation \ref{ope:eliminate_loop1} is the operation adding a tetrahedron and applying desingualrization once, and so Operation \ref{ope:eliminate_loop1} increases the number of tetrahedra by $\mathcal{O}(|\mathcal{S}^k|)$ by Lemma \ref{lem:desing}.
    Operation \ref{ope:eliminate_loop2} is the operation gluing two edges and applying desingualrization twice, and so Operation \ref{ope:eliminate_loop2} also increases the number of tetrahedra by $\mathcal{O}(|\mathcal{S}^k|)$ by \ref{lem:desing}.
    
    Designularization is carried out at most twice other than Operation \ref{ope:eliminate_loop1} and Operation \ref{ope:eliminate_loop2} since the number of singular points in $\mathcal{T}'$ is at most two by Lemma \ref{lem:V'_E'}.
    Similarly, Operation \ref{ope:eliminate_loop1} and Operation \ref{ope:eliminate_loop2} are carried out at most four times. 
    Therefore, the number of tetrahedra increased by Operation \ref{ope:des} is $\mathcal{O}(|\mathcal{S}^k|)$.
\end{proof}

Now, we have $\mathcal{O}(|\mathcal{S}^k|) = \mathcal{O}(n)$, where $n$ is the number of tetrahedra in a triangulation $\mathcal{T}$ of the exterior of a link in a thickened closed orientable surface $\mathcal{S} \times I$.
Thus, the number of tetrahedra is increased exponentially with the number of Operation \ref{ope:des}.
Therefore, it takes too much time to solve classical knot recognition if we simply perform Operation \ref{ope:des}.
For this reason, we consider reducing the number of triangles in the triangulations $\mathcal{S}^0$ and $\mathcal{S}^1$ before performing Operation \ref{ope:des}.
The following lemma is based on the method of constructing a one-vertex triangulation of a classical knot exterior in \cite{B_unknot}.

\begin{lem}\label{lem:one_vertex}
    Let $\mathcal{S}$ be a closed orientable connected surface whose genus is not zero, $\hat{D}$ be a link in $\mathcal{S} \times I$, and $\mathcal{T}$ be a triangulation of the exterior $M = \text{cl}(\mathcal{S} \times I - N(\hat{D}))$ of $\hat{D}$.
    Let $\mathcal{S}^k$ $(k = 0 \text{ or } 1)$ denote the triangulations of $\mathcal{S}$ in the boundary of $\mathcal{T}$.
    Suppose that $M$ is irreducible and $\partial$-irreducible.
    Then, we can reduce the vertices of $\mathcal{S}^k$ to one while keeping the topology of $M$.
    Furthermore, this operation can be carried out in time $\mathcal{O}(n^3)$, where $n$ is the number of tetrahedra in $\mathcal{T}$.
\end{lem}
\begin{proof}
    We show that the number of vertices of $\mathcal{S}^k$ is reduced to one by crushing using normal disks.
    If there are two or more vertices in $\mathcal{S}^k$, then there is an edge $e$ which connects distinct vertices and lies on $\mathcal{S}^k$.
    $\partial N(e)$ denotes the boundary of a small regular neighborhood of $e$ in $\mathcal{T}$.
    $\partial N(e)$ is an inessential disk in $\mathcal{T}$.

    Suppose that $\partial N(e)$ is a normal disk, then we can reduce the vertices of $\mathcal{S}^k$ by the crushing procedure using $\partial N(e)$.
    Since $\mathcal{T}$ is irreducible and $\partial$-irreducible, 
    we can obtain the 3-manifold represented by the crushed triangulation by adding 3-balls and 3-spheres to $E$.
    Thus, the triangulation obtained by removing 3-balls and 3-spheres from the crushed triangulation is a triangulation of $M$.

    Consider the case where $\partial N(e)$ is not a normal surface.
    In this situation, two or more edges of a tetrahedron in $\mathcal{T}$ are identified to $e$ depicted as Figure \ref{fig:makeNormal}.
    We can obtain a normal surface which is an inessential disk as the boundary of a small neighborhood of a subcomplex constructed as follows.
    \begin{figure}[htbp]
        \centering
        \tikzset{every picture/.style={line width=0.75pt}} 

\begin{tikzpicture}[x=0.75pt,y=0.75pt,yscale=-0.7,xscale=0.7]

\draw  [color={rgb, 255:red, 155; green, 155; blue, 155 }  ,draw opacity=1 ] (156.46,10) -- (156.46,177.39) -- (79.5,140.84) -- cycle ;
\draw  [color={rgb, 255:red, 155; green, 155; blue, 155 }  ,draw opacity=1 ] (156.46,10) -- (156.46,177.39) -- (252.67,121.6) -- cycle ;
\draw [color={rgb, 255:red, 155; green, 155; blue, 155 }  ,draw opacity=1 ] [dash pattern={on 0.84pt off 2.51pt}]  (79.5,140.84) -- (252.67,121.6) ;
\draw [line width=1.5]    (156.46,10) -- (252.67,121.6) ;
\draw [color={rgb, 255:red, 155; green, 155; blue, 155 }  ,draw opacity=1 ]   (139.5,133) .. controls (154.5,119) and (165.5,129) .. (170.89,134.7) ;
\draw [color={rgb, 255:red, 155; green, 155; blue, 155 }  ,draw opacity=1 ]   (141.5,98) .. controls (156.5,84) and (165.11,93.3) .. (170.5,99) ;
\draw [color={rgb, 255:red, 155; green, 155; blue, 155 }  ,draw opacity=1 ]   (208.5,112) .. controls (213.5,122) and (224.5,117) .. (231.12,115.29) ;
\draw [color={rgb, 255:red, 155; green, 155; blue, 155 }  ,draw opacity=1 ]   (196.5,96) .. controls (201.5,106) and (211.5,101) .. (218.12,99.29) ;
\draw [color={rgb, 255:red, 155; green, 155; blue, 155 }  ,draw opacity=1 ]   (147.5,26) .. controls (157.5,54) and (163.5,80) .. (176.5,72) ;
\draw  [fill={rgb, 255:red, 155; green, 155; blue, 155 }  ,fill opacity=0.5 ] (170.89,167.7) .. controls (171.65,100.01) and (162.97,46.65) .. (198.67,79.02) .. controls (208.64,88.06) and (222.06,103.93) .. (240.12,128.29) .. controls (223.5,134) and (218.5,130) .. (217.5,125) .. controls (179.5,75) and (192.5,95) .. (147.5,26) .. controls (139.5,107) and (142.5,90) .. (138.5,168) .. controls (148.5,158) and (161.5,156) .. (170.89,167.7) -- cycle ;
\draw [line width=1.5]    (156.46,10) -- (156.46,177.39) ;
\draw    (63.5,69) .. controls (68.5,87) and (119.5,94) .. (142.5,87) ;

\draw (164.8,180.76) node [anchor=north west][inner sep=0.75pt]    {$e$};
\draw (218.77,44.73) node [anchor=north west][inner sep=0.75pt]    {$e$};
\draw (45,46.4) node [anchor=north west][inner sep=0.75pt]    {$\partial N( e)$};

\end{tikzpicture}
        \caption{The case where $\partial N(e)$ is {\it not} a normal surface}
        \label{fig:makeNormal}
    \end{figure}
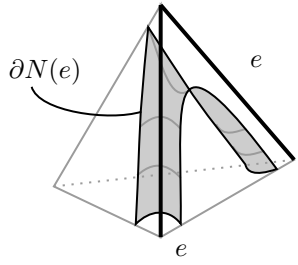

    First, we define $\mathcal{E} = \{e\}$.
    Then, we apply the following operations as much as possible:
    \begin{itemize}
        \item If there is a $2$-simplex whose two or more edges belonging to $\mathcal{E}$, then the entire of the $2$-simplex is added to $\mathcal{E}$.
        \item If there is an tetrahedron whose all faces belong to $\mathcal{E}$, then the entire of the tetrahedron is added to $\mathcal{E}$.
    \end{itemize}
    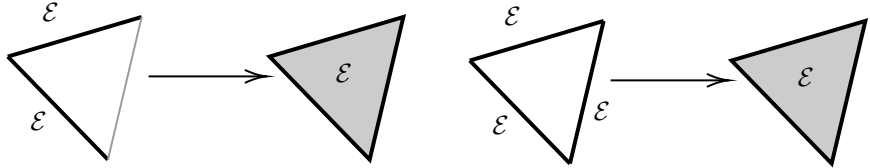
\begin{figure}[htbp]
        \centering
        \tikzset{every picture/.style={line width=0.75pt}} 

\begin{tikzpicture}[x=0.75pt,y=0.75pt,yscale=-1,xscale=1]

\draw [line width=1.5]    (21,40) -- (71.5,92) ;
\draw [line width=1.5]    (21,40) -- (88.5,20) ;
\draw [color={rgb, 255:red, 155; green, 155; blue, 155 }  ,draw opacity=1 ]   (71.5,92) -- (88.5,20) ;
\draw  [fill={rgb, 255:red, 204; green, 204; blue, 204 }  ,fill opacity=1 ][line width=1.5]  (203.5,92) -- (153,40) -- (220.5,20) -- cycle ;
\draw    (92,50) -- (149.5,50) ;
\draw [shift={(151.5,50)}, rotate = 180] [color={rgb, 255:red, 0; green, 0; blue, 0 }  ][line width=0.75]    (10.93,-3.29) .. controls (6.95,-1.4) and (3.31,-0.3) .. (0,0) .. controls (3.31,0.3) and (6.95,1.4) .. (10.93,3.29)   ;
\draw [line width=1.5]    (254,42) -- (304.5,94) ;
\draw [line width=1.5]    (254,42) -- (321.5,22) ;
\draw  [fill={rgb, 255:red, 204; green, 204; blue, 204 }  ,fill opacity=1 ][line width=1.5]  (436.5,94) -- (386,42) -- (453.5,22) -- cycle ;
\draw    (325,52) -- (382.5,52) ;
\draw [shift={(384.5,52)}, rotate = 180] [color={rgb, 255:red, 0; green, 0; blue, 0 }  ][line width=0.75]    (10.93,-3.29) .. controls (6.95,-1.4) and (3.31,-0.3) .. (0,0) .. controls (3.31,0.3) and (6.95,1.4) .. (10.93,3.29)   ;
\draw [line width=1.5]    (304.5,94) -- (321.5,22) ;

\draw (31,66.4) node [anchor=north west][inner sep=0.75pt]    {$\mathcal{E}$};
\draw (37,11.4) node [anchor=north west][inner sep=0.75pt]    {$\mathcal{E}$};
\draw (185,42.4) node [anchor=north west][inner sep=0.75pt]    {$\mathcal{E}$};
\draw (264,68.4) node [anchor=north west][inner sep=0.75pt]    {$\mathcal{E}$};
\draw (270,13.4) node [anchor=north west][inner sep=0.75pt]    {$\mathcal{E}$};
\draw (418,44.4) node [anchor=north west][inner sep=0.75pt]    {$\mathcal{E}$};
\draw (315,61.4) node [anchor=north west][inner sep=0.75pt]    {$\mathcal{E}$};

\end{tikzpicture}
        \caption{The construction of $\mathcal{E}$}
        \label{fig:makeNormal2}
    \end{figure}

    \setcounter{claim}{0}
    \begin{claim}
        The boundary of a small regular neighborhood of $\mathcal{E}$, denote $\partial N(\mathcal{E})$, consists of inessential disks and inessential $2$-spheres.
    \end{claim}
    \noindent
    {\it Proof of Claim A.} 
    If $\mathcal{E} = \{e\}$, $\partial N(\mathcal{E})$ is an inessential disk since $\mathcal{T}$ is $\partial$-irreducible and the entire of $e$ lies on $\mathcal{S}^k$.
    We observe the effect of the operations to extend $\mathcal{E}$.
    \begin{itemize}
        \item The operation to add an entire $2$-simplex if the two or more edges of the $2$-simplex belong to $\mathcal{E}$:
        \begin{itemize}
            \item If there is a $2$-simplex whose two edges belong to $\mathcal{E}$, 
                    $\partial N(\mathcal{E})$ is isotopic to the original surface (Figure \ref{fig:makeNormal3}\subref{fig:makeNormal3-1}).
            \item If there is a $2$-simplex whose three edges belong to $\mathcal{E}$, 
                    $\partial N(\mathcal{E})$ is obtained by cutting the original surface along an circle and 
                    filling the copies of the circles with disks parallel to the $2$-simplex (Figure \ref{fig:makeNormal3}\subref{fig:makeNormal3-1}).
                    This operation divides the original surface into two surfaces.
                    These surfaces are inessential because $\mathcal{T}$ is irreducible and $\partial$-irreducible.
        \end{itemize}
        \begin{figure}[htbp]
            \begin{minipage}{0.49\textwidth}
                \centering
                \tikzset{every picture/.style={line width=0.75pt}} 

\begin{tikzpicture}[x=0.75pt,y=0.75pt,yscale=-1,xscale=1]

\draw [line width=1.5]    (11,40) -- (61.5,92) ;
\draw [line width=1.5]    (11,40) -- (78.5,20) ;
\draw [color={rgb, 255:red, 155; green, 155; blue, 155 }  ,draw opacity=1 ]   (61.5,92) -- (78.5,20) ;
\draw    (82,50) -- (139.5,50) ;
\draw [shift={(141.5,50)}, rotate = 180] [color={rgb, 255:red, 0; green, 0; blue, 0 }  ][line width=0.75]    (10.93,-3.29) .. controls (6.95,-1.4) and (3.31,-0.3) .. (0,0) .. controls (3.31,0.3) and (6.95,1.4) .. (10.93,3.29)   ;
\draw    (78.5,8) .. controls (113.5,33) and (36.5,31) .. (32.5,46) .. controls (28.5,61) and (95.5,88) .. (57.5,104) ;
\draw    (7.5,48) .. controls (17.5,50) and (30.5,51) .. (32.5,46) ;
\draw  [dash pattern={on 0.84pt off 2.51pt}]  (10.5,34) .. controls (20.5,36) and (33.5,39) .. (32.5,46) ;
\draw    (42.5,95) .. controls (52.5,97) and (67.5,99) .. (68.5,94) ;
\draw  [dash pattern={on 0.84pt off 2.51pt}]  (47.5,84) .. controls (57.5,86) and (69.5,87) .. (68.5,94) ;
\draw    (49.5,21) .. controls (59.5,23) and (84.5,26) .. (86.5,21) ;
\draw  [dash pattern={on 0.84pt off 2.51pt}]  (55.5,8) .. controls (65.5,10) and (87.5,14) .. (86.5,21) ;
\draw    (224.5,8) -- (199.5,113) ;
\draw  [dash pattern={on 0.84pt off 2.51pt}]  (181,47.5) .. controls (191,49.5) and (213,53.5) .. (212,60.5) ;
\draw  [fill={rgb, 255:red, 204; green, 204; blue, 204 }  ,fill opacity=0.6 ][line width=1.5]  (193.5,91.5) -- (143,39.5) -- (210.5,19.5) -- cycle ;
\draw    (175,60.5) .. controls (185,62.5) and (210,65.5) .. (212,60.5) ;

\end{tikzpicture}
                \subcaption{The case where two edges in $\mathcal{E}$}
                \label{fig:makeNormal3-1}
            \end{minipage}
            \begin{minipage}{0.49\textwidth}
                \centering
                \tikzset{every picture/.style={line width=0.75pt}} 

\begin{tikzpicture}[x=0.75pt,y=0.75pt,yscale=-1,xscale=1]

\draw  [line width=1.5]  (60.68,87.88) -- (17.34,51.22) -- (108.59,46.39) -- cycle ;
\draw    (16.47,92) .. controls (18.1,77) and (27.07,73) .. (31.14,69) ;
\draw    (101.25,92) .. controls (100.44,80) and (92.28,77) .. (88.2,70.86) ;
\draw [color={rgb, 255:red, 155; green, 155; blue, 155 }  ,draw opacity=1 ]   (102.07,35) .. controls (102.07,42.37) and (18.1,42) .. (18.1,34.63) ;
\draw [color={rgb, 255:red, 155; green, 155; blue, 155 }  ,draw opacity=1 ] [dash pattern={on 0.84pt off 2.51pt}]  (102.07,35) .. controls (102.07,29.37) and (18.1,30) .. (18.1,34.63) ;
\draw [color={rgb, 255:red, 155; green, 155; blue, 155 }  ,draw opacity=1 ]   (37.13,56.94) .. controls (36.85,67) and (82.5,71) .. (82.5,56) ;
\draw [color={rgb, 255:red, 155; green, 155; blue, 155 }  ,draw opacity=1 ] [dash pattern={on 0.84pt off 2.51pt}]  (37.13,56.94) .. controls (37.66,54) and (81.69,50) .. (82.5,56) ;
\draw    (36.03,62) .. controls (42.55,46) and (21.36,44) .. (7.5,23) ;
\draw    (83.32,62) .. controls (75.92,45.24) and (101.25,47) .. (110.22,24) ;
\draw [color={rgb, 255:red, 155; green, 155; blue, 155 }  ,draw opacity=1 ]   (253.7,29) .. controls (249.62,36.37) and (190.92,36.37) .. (190.92,29) ;
\draw [color={rgb, 255:red, 155; green, 155; blue, 155 }  ,draw opacity=1 ] [dash pattern={on 0.84pt off 2.51pt}]  (253.7,29) .. controls (253.7,23.37) and (190.92,24.37) .. (190.92,29) ;
\draw    (262.66,17) .. controls (234.95,58) and (205.6,54) .. (182.77,18) ;
\draw [color={rgb, 255:red, 155; green, 155; blue, 155 }  ,draw opacity=1 ]   (254.51,89) .. controls (250.44,96.37) and (190.11,97.37) .. (190.11,90) ;
\draw [color={rgb, 255:red, 155; green, 155; blue, 155 }  ,draw opacity=1 ] [dash pattern={on 0.84pt off 2.51pt}]  (254.51,89) .. controls (254.51,83.37) and (190.11,85.37) .. (190.11,90) ;
\draw    (185.22,99) .. controls (195.82,70) and (243.1,68) .. (261.03,98) ;
\draw  [fill={rgb, 255:red, 204; green, 204; blue, 204 }  ,fill opacity=1 ][line width=1.5]  (222.09,88.88) -- (178.75,52.22) -- (270,47.39) -- cycle ;
\draw    (119.59,53) -- (166.1,53) ;
\draw [shift={(168.1,53)}, rotate = 180] [color={rgb, 255:red, 0; green, 0; blue, 0 }  ][line width=0.75]    (10.93,-3.29) .. controls (6.95,-1.4) and (3.31,-0.3) .. (0,0) .. controls (3.31,0.3) and (6.95,1.4) .. (10.93,3.29)   ;

\end{tikzpicture}
                \subcaption{The case where three edges in $\mathcal{E}$}
                \label{fig:makeNormal3-2}
            \end{minipage}
            \caption{The effect of extending $\mathcal{E}$}
            \label{fig:makeNormal3}
        \end{figure}
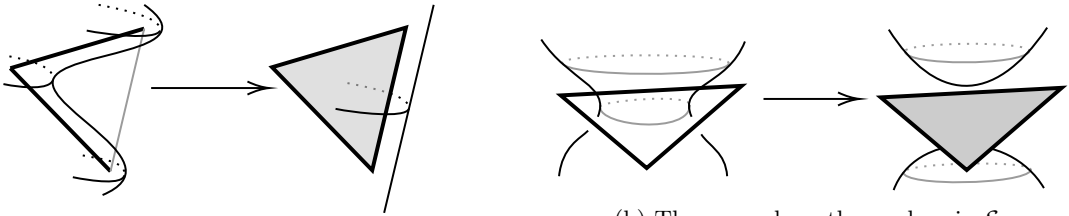

        \item The operation to add an entire tetrahedron if all faces of the tetrahedron belong to $\mathcal{E}$:
            \begin{itemize}
                \item $\partial N(\mathcal{E})$ is obtained by removing a $2$-sphere component from the original surface.
            \end{itemize}
    \end{itemize}
    
    From the above observation, the component of $\partial N(\mathcal{E})$ consists of inessential disks and inessential $2$-spheres.
    \qed
    
    \begin{claim}
        $\partial N(\mathcal{E})$ is a normal surface with respect to $\mathcal{T}$.
    \end{claim}
    \noindent
    {\it Proof of Claim B.}
    $\partial N(\mathcal{E})$ can be obtained as a normal surface by placing normal disks for each tetrahedron $\Delta$ as follows:
    \begin{itemize}
        \item If the entire of $\Delta$ belongs to $\mathcal{E}$, then no normal disks are placed.
        \item If the entire of $\Delta$ does not belong to $\mathcal{E}$ and a face of $\Delta$ belongs to $\mathcal{E}$,
                then we place the triangle normal disk which is parallel to the face as shown in the left of Figure \ref{fig:makeNormal4}.
        \item If no faces of $\Delta$ belongs to $\mathcal{E}$ and an edge of $\Delta$ belongs to $\mathcal{E}$, 
                then we place the quadrilateral normal disk which has no intersection with the edge as shown in the center of Figure \ref{fig:makeNormal4}.
        \item If a vertex of $\Delta$ belongs to $\mathcal{E}$ and any edges whose one of the end points is the vertex do not belong to $\mathcal{E}$,
                then we place the triangle normal disk surrounding the vertex as shown in the right of Figure \ref{fig:makeNormal4}.  
    \end{itemize}
    \qed
    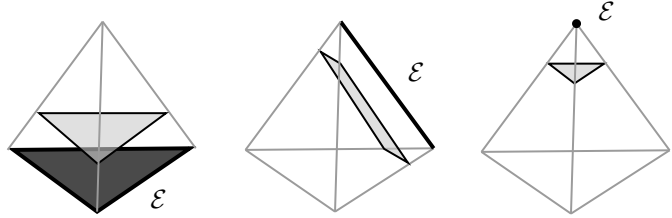
\begin{figure}[htbp]
        \centering
        \tikzset{every picture/.style={line width=0.75pt}} 

\begin{tikzpicture}[x=0.75pt,y=0.75pt,yscale=-1,xscale=1]

\draw  [fill={rgb, 255:red, 74; green, 74; blue, 74 }  ,fill opacity=1 ][line width=1.5]  (115.5,84.74) -- (68.75,116.71) -- (26.59,85.54) -- cycle ;
\draw [color={rgb, 255:red, 155; green, 155; blue, 155 }  ,draw opacity=1 ]   (71.55,20.8) -- (23.59,85.54) ;
\draw [color={rgb, 255:red, 155; green, 155; blue, 155 }  ,draw opacity=1 ]   (71.55,20.8) -- (118.5,84.74) ;
\draw  [fill={rgb, 255:red, 204; green, 204; blue, 204 }  ,fill opacity=0.6 ] (69,92.54) -- (103.5,67) -- (39.5,67) -- cycle ;
\draw [color={rgb, 255:red, 155; green, 155; blue, 155 }  ,draw opacity=1 ]   (71.55,20.8) -- (68.75,116.71) ;
\draw [color={rgb, 255:red, 155; green, 155; blue, 155 }  ,draw opacity=1 ]   (191.55,20.8) -- (238.5,84.74) ;
\draw [color={rgb, 255:red, 155; green, 155; blue, 155 }  ,draw opacity=1 ]   (143.59,85.54) -- (188.75,116.71) ;
\draw [color={rgb, 255:red, 155; green, 155; blue, 155 }  ,draw opacity=1 ]   (188.75,116.71) -- (238.5,84.74) ;
\draw [color={rgb, 255:red, 155; green, 155; blue, 155 }  ,draw opacity=1 ]   (143.59,85.54) -- (238.5,84.74) ;
\draw  [fill={rgb, 255:red, 204; green, 204; blue, 204 }  ,fill opacity=0.6 ] (191.5,42) -- (225.5,92) -- (213.5,85) -- (181.5,36) -- cycle ;
\draw [color={rgb, 255:red, 155; green, 155; blue, 155 }  ,draw opacity=1 ]   (191.55,20.8) -- (188.75,116.71) ;
\draw [color={rgb, 255:red, 0; green, 0; blue, 0 }  ,draw opacity=1 ][line width=1.5]    (191.55,20.8) -- (238.5,84.74) ;
\draw [color={rgb, 255:red, 155; green, 155; blue, 155 }  ,draw opacity=1 ]   (191.55,20.8) -- (143.59,85.54) ;
\draw [color={rgb, 255:red, 155; green, 155; blue, 155 }  ,draw opacity=1 ]   (262.59,86.54) -- (307.75,117.71) ;
\draw [color={rgb, 255:red, 155; green, 155; blue, 155 }  ,draw opacity=1 ]   (307.75,117.71) -- (357.5,85.74) ;
\draw [color={rgb, 255:red, 155; green, 155; blue, 155 }  ,draw opacity=1 ]   (262.59,86.54) -- (357.5,85.74) ;
\draw  [fill={rgb, 255:red, 204; green, 204; blue, 204 }  ,fill opacity=0.6 ] (309.5,52) -- (324.5,42) -- (296.5,42) -- cycle ;
\draw [color={rgb, 255:red, 155; green, 155; blue, 155 }  ,draw opacity=1 ]   (310.55,21.8) -- (307.75,117.71) ;
\draw [color={rgb, 255:red, 155; green, 155; blue, 155 }  ,draw opacity=1 ]   (310.55,21.8) -- (262.59,86.54) ;
\draw [color={rgb, 255:red, 155; green, 155; blue, 155 }  ,draw opacity=1 ]   (310.55,21.8) -- (357.5,85.74) ;
\draw  [fill={rgb, 255:red, 0; green, 0; blue, 0 }  ,fill opacity=1 ] (308.78,21.8) .. controls (308.78,20.83) and (309.57,20.04) .. (310.55,20.04) .. controls (311.52,20.04) and (312.31,20.83) .. (312.31,21.8) .. controls (312.31,22.77) and (311.52,23.56) .. (310.55,23.56) .. controls (309.57,23.56) and (308.78,22.77) .. (308.78,21.8) -- cycle ;

\draw (94.12,104.12) node [anchor=north west][inner sep=0.75pt]    {$\mathcal{E}$};
\draw (224.12,40.12) node [anchor=north west][inner sep=0.75pt]    {$\mathcal{E}$};
\draw (320.12,9.12) node [anchor=north west][inner sep=0.75pt]    {$\mathcal{E}$};

\end{tikzpicture}
        \caption{How to place normal disks}
        \label{fig:makeNormal4}
    \end{figure}
    

    The normal surface which is an inessential disk is obtained by removing the components which has no boundaries from $\partial N(\mathcal{E})$.
    This normal surface is inessential in $\mathcal{T}$. Therefore,
    we can reduce the number of vertices in $\mathcal{S}^k$ by the crushing procedure using this normal surface while keeping the topology of $\mathcal{S}^k$.
    We can reduce the number of vertices in $\mathcal{S}^k$ to one by repeatedly performing this operation.

    Next, we analyze the running time to construct $\mathcal{E}$.
    It takes $\mathcal{O}(n)$ time to perform the operation to check the conditions and add $i$-simplices ($i=0,1,2,3$) to $\mathcal{E}$ for each tetrahedron.
    The operation is run $\mathcal{O}(n)$ times to construct $\mathcal{E}$.
    Thus, $\mathcal{E}$ is constructed in time $\mathcal{O}(n^2)$.

    It takes $\mathcal{O}(n)$ time to obtain $\partial N(\mathcal{E})$, a small neighborhood of $\mathcal{E}$, as a normal surface by placing and gluing normal disks.
    The normal surface has at most two disks for each normal disk type for each tetrahedron.
    Thus, we can remove components which has no boundaries in time $\mathcal{O}(n)$.

    There are $\mathcal{O}(n)$ vertices in $\mathcal{S}^k$. 
    Thus, we can reduce the vertices of $\mathcal{S}^k$ to one in time $\mathcal{O}(n^3)$.
\end{proof}

From this lemma, it turns out that we can suppose that each component of $\mathcal{S}^k$ has only one vertex.
It is possible to suppress the increase in the number of tetrahedra when Operation \ref{ope:des} is performed based on this assumption.
\begin{lem}\label{lem:desNum}
    Let $\mathcal{T}$ be a triangulation of the canonical exterior $M$ of a virtual link diagram $D$ and $A \subset M$ be a vertical essential vertex annulus with respect to $\mathcal{T}$.
    There are two triangulations of $\mathcal{S}$ in the boundary of $\mathcal{T}$, and $\mathcal{S}^0$ and $\mathcal{S}^1$ denote them.
    Assume that each $\mathcal{S}^k$ has only one vertex.
    Then, the number of tetrahedra increased by Operation \ref{ope:des} is $\mathcal{O}(c)$, where $c$ is the number of the real crossings of $D$.
\end{lem}
\begin{proof}
    By Corollary \ref{cor:genus}, we have $g(\mathcal{S}^k) \in \mathcal{O}(c)$.
    $\mathcal{S}^k$ has only one vertex, and so 
    we see that $|\mathcal{S}^k| \in \mathcal{O}(c)$ from the relationship of the Euler characteristic and the genus of $\mathcal{S}^k$.
    Thus, the number of tetrahedra increased by Operation \ref{ope:des} is $\mathcal{O}(c)$ from Lemma \ref{cor:des_num}.
\end{proof}

Next, we consider the running time of Operation \ref{ope:des}.
\begin{lem}
    Let $\mathcal{T}$ be a triangulation of the canonical exterior $M$ of a virtual link diagram $D$ and $A$ be a vertical essential normal annulus in $M$ with respect to $\mathcal{T}$.
    Suppose that each triangulation of $\mathcal{S}$ in the boundary of $\mathcal{T}$ has exactly one vertex.
    If $A$ is a vertex annulus with respect to $\mathcal{T}$, then Operation \ref{ope:des} is carried out in $\mathcal{O}(n^2 + c)$ time, where $n$ is the number of tetrahedra in $\mathcal{T}$.
\end{lem}
\begin{proof}
     Let $\mathbf{x} = (x_1, \ldots, x_{7n})$ denote the vector representation of the given vertex surface $A$.
     As shown in Lemma \ref{lem:splittingTime}, we can read $\mathbf{x}$ in time $\mathcal{O}(n^2)$.
     From Theorem \ref{thm:crush}, we can carry out the step 1 in time $\mathcal{O}(n)$.
     Let $\mathcal{T}'$ be the triangulation obtained by the step 1 of Operation \ref{ope:des}.
     We can carry out desingularization in time $\mathcal{O}(n)$, and desingularization is carried out at most twice since the number of the singular points in $\mathcal{T}'$ is at most two by the from Lemma \ref{lem:V'_E'}, therefore, the step 2 is carried out in $\mathcal{O}(n)$ time.
     Operation \ref{ope:eliminate_loop1} and Operation \ref{ope:eliminate_loop2} can be carried out in $\mathcal{O}(n)$ time since we can carry out desingularization in $\mathcal{O}(n)$ time.
     Operation \ref{ope:eliminate_loop1} and Operation \ref{ope:eliminate_loop2} are carried out at most four times by Lemma \ref{lem:V'_E'}, thus, the step 3 and the step 4 take in $\mathcal{O}(n)$ time.
     Let $\mathcal{T}''$ denote the triangulation obtained from $\mathcal{T}$ by running the step (1)--(4) of Operation \ref{ope:des}.
     We can remove the components of which contains no components of $\partial N(\hat{D})$ in $\mathcal{O}(n)$ time since the size of $\mathcal{T}''$ is $\mathcal{O}(n+c)$, where $c$ is the number of the real crossings in $D$.
     Thereby, we can carry out Operation \ref{ope:des} in time $\mathcal{O}(n^2 + c)$.
\end{proof}

\subsection{The determination of a vertex surface}
In this subsection, we give a method of determining 
whether an integer vector $\mathbf{x}$ is the representation of a vertex surface
in a triangulation of the canonical exterior of a virtual link diagram.
This method is used in the proof of Theorem \ref{thm:main}.

First of all, we give a way to calculate the Euler characteristic of a normal surface.
\begin{lem}
    Suppose that $\mathcal{T}$ is an $n$-tetrahedra triangulation of a $3$-manifold $M$ and $F$ is a normal surface in $M$ with respect to $\mathcal{T}$.
    ${\bf v}_F \in \mathbb{R}^{7n}$ denotes the vector representation of $F$.
    Given the vector ${\bf v}_F$, 
    let $\chi: \mathbb{R}^{7n} \rightarrow \mathbb{R}$ be the function that outputs the Euler characteristic of $F$.
    Then $\chi$ is a linear function.
\end{lem}
\begin{proof}
    For any $i = 1, \ldots, 7n$, we define $a_i = 1 + e_i + v_i$, where:
    \begin{itemize}
        \item $e_i$ is defined as a sum over edges of the normal disk corresponding to the $i$-th element of ${\bf v}_F$:
                an edge of the normal disk contributes $-1$ to this sum if the edge lies on the boundary of $\mathcal{T}$, 
                or $-\frac{1}{2}$ if the edge is internal to $\mathcal{T}$,
        \item $v_i$ is defined as a sum over vertices of the normal disk corresponding to the $i$-th element of ${\bf v}_F$:
                a vertex of the normal disk contributes $\frac{1}{d}$, 
                where $d$ is the degree of the $1$-simplex of $\mathcal{T}$ that the vertex lies within.
    \end{itemize}
    Consider the linear function $\chi(\mathbf{v}_F) = (a_1, \ldots, a_{7n}) \cdot \mathbf{v}_F$.
    A vertex of degree $d$ in $F$ contributes a total of $\frac{1}{d} \cdot d = 1$ to $\chi(\mathbf{v}_F)$.
    An internal edge of $F$ contributes a total of $-\frac{1}{2} \cdot 2 = -1$ to $\chi(\mathbf{v}_F)$ and 
    a boundary edge of $F$ contributes $-1$ to $\chi(\mathbf{v}_F)$.
    A face of $F$ contributes $1$ to $\chi(\mathbf{v}_F)$.
    Therefore, $\chi(\mathbf{v}_F)$ gives the Euler characteristic of $F$.
\end{proof}

Next, we introduce basic properties of vertex surfaces.
\begin{defi}
    Let $\mathbf{x} = (x_1, \ldots, x_{d})$ and $\mathbf{y} = (y_1, \ldots, y_{d})$ be vectors in $\mathbb{R}^{d}$.
    We say that $\mathbf{x}$ {\it dominates} $\mathbf{y}$ if for any index $i = 1, \ldots, d$ satisfies that $y_i = 0$ if $x_i = 0$.
    Furthermore, we say that $\mathbf{x}$ {\it strictly dominates} $\mathbf{y}$ 
    if $\mathbf{x}$ dominates $\mathbf{y}$ and there is an index $i$ which satisfies $x_i \neq 0$ and $y_i = 0$.
\end{defi}

\begin{lem}[Burton \cite{B}]\label{lem:vertex_determine}
    Let $\mathcal{T}$ be an $n$-tetrahedra triangulation of a $3$-manifold $M$.
    $\mathcal{P}$ denotes the projective solution space which is defined by the matching equations of $\mathcal{T}$.
    Let $\mathbf{x}_F$ be the vector representation of a normal surface in $M$ with respect to $\mathcal{T}$.
    If there are no vertex surfaces whose vector representations dominate $\mathbf{x}_F$,
    then the projection of $\mathbf{x}_F$ onto the hyperplane $\sum_{i=0}^{7n} x_i = 1$ is a vertex solution in $\mathcal{P}$.
\end{lem}

From this lemma, we can prove the following lemma.
\begin{lem}\label{lem:S^2_check}
    Assume that $\mathcal{T}$ is an $n$-tetrahedra triangulation of the exterior $M$ of a link in a thickened closed orientable surface $\mathcal{S} \times I$ and
    $\mathbf{x} \in \mathbb{Z}^{7n}$ is an integer vector.
    Then, we can determine whether $\mathbf{x}$ is the vector representation of a vertex $2$-sphere with respect to $\mathcal{T}$ in polynomial time of $n$.
\end{lem}
\begin{proof}
    By Theorem \ref{thm:HLP}, if there is an index $i$ which does not satisfy $x_i \leq 2^{7n-1}$, 
    then $\mathbf{x}$ is not the vector representation of a vertex surface, 
    and so we assume $x_i \leq 2^{7n-1}$.

    We can verify that $\mathbf{x}$ is the vector representation of a normal surface by 
    checking $\mathbf{x} \geq \mathbf{0}$, the matching equations $A\mathbf{x} = \mathbf{0}$ and the quadrilateral condition for each tetrahedron.
    These conditions can be checked in polynomial time since $x_i \leq 2^{7n-1}$.

    Next, we check that the projection of $\mathbf{x}$ onto the hyperplane $\sum x_i = 1$ is a vertex solution in $\mathcal{P}$, 
    where $\mathcal{P}$ is the projective solution space which is defined by the matching equation of $\mathcal{T}$.
    By Lemma \ref{lem:vertex_determine}, we can check this by checking that 
    there are no normal surfaces whose vector representations are strictly dominated by $\mathbf{x}$.
    Let $NZ(\mathbf{x})$ and $Z(\mathbf{x})$ be the sets as follows:
    \begin{itemize}
        \item $NZ(\mathbf{x}) = \{i \in \{1, \ldots, 7n\}| x_i \neq 0\}$
        \item $Z(\mathbf{x}) = \{i \in \{1, \ldots, 7n\}| x_i = 0\}$
    \end{itemize}
    If an integer vector $\mathbf{y}$ which is dominated by $\mathbf{x}$ is the vector representation of a normal surface in $\mathcal{T}$,
    then $\mathbf{y}$ satisfies the following conditions:
    \begin{itemize}
        \item $\mathbf{y} \geq \mathbf{0}$,
        \item the matching equations $A\mathbf{y} = \mathbf{0}$,
        \item $x_i = 0 \Rightarrow y_i = 0$.
    \end{itemize}
    From the third condition, if $\mathbf{x}$ satisfies the quadrilateral condition for each tetrahedron, then $\mathbf{y}$ satisfies the quadrilateral condition for each tetrahedron.
    Thus, $\mathbf{y}$ is the vector representation of a normal surface if $\mathbf{y}$ satisfies the three conditions.
    Additionally, there is an index $i$ which satisfies $x_i \neq 0 \land y_i = 0$ if $\mathbf{x}$ strictly dominates $\mathbf{y}$.
    We consider to determine whether there is such a vector using linear algebra.
    
    Let $\mathbf{y}$ be an integer vector which is dominated by $\mathbf{x}$ and represents a normal surface in $\mathcal{T}$.
    $\mathbf{y}$ satisfies the following conditions:
    \begin{itemize}
        \item $\mathbf{y} \geq \mathbf{0}$,
        \item $A\mathbf{y} = \mathbf{0}$,
        \item for any $j \in Z(\mathbf{x})$, $y_j = 0$,
    \end{itemize}
    We call the set of these conditions $C$.
    Conversely, if an integer vector satisfies the conditions in $C$, then the vector is dominated by $\mathbf{x}$ and represents a normal surface in $\mathcal{T}$.
    
    Additionally, if $\mathbf{x}$ strictly dominates $\mathbf{y}$, then there is an index $i \in NZ(\mathbf{x})$ which $y_i = 0$.
    For each index $i \in NZ(\mathbf{x})$, $C_i$ denotes the conditions which is obtained by adding the condition $y_i = 0$ to $C$.
    Conversely, if an integer vector satisfies the conditions in $C_i$, then the vector is strictly dominated by $\mathbf{x}$ and represents a normal surface in $\mathcal{T}$.

    We consider how to determine if there is a vector $\mathbf{y}$ which satisfies the conditions of $C_i$.
    $\mathbf{e}_k$ denotes the $7n$-dimensional row vector whose $k$-th coordinate is $1$ and the other coordinates are $0$, and 
    $I_{Z(\mathbf{x})}$ denotes a matrix which is obtained by arranging all row vectors $\mathbf{e}_j (j \in Z(\mathbf{x}))$ vertically.
    For any element $i$ in $NZ(x)$, we define the matrix $M_i$ as follows:
    \[
        M_i = \left( 
            \begin{array}{c}
                A\\
                I_{Z(\mathbf{x})}\\
                \mathbf{e}_i
            \end{array} \right)
    \]
    Then there is a vector $\mathbf{y}$ which satisfies the conditions in $C_i$ if and only if $\text{rank}(M_i) \neq 0$ holds.

    There are no normal surfaces whose vector representations are dominated by $\mathbf{x}$ 
    if $\text{rank}(M_i) = 0$ holds for any element $i$ in $NZ(\mathbf{x})$.
    Therefore, the projection of $\mathbf{x}$ onto the hyperpalne $\sum_{i=1}^{7n} x_i = 1$ is a vertex solution of $\mathcal{P}$ by Lemma \ref{lem:vertex_determine}.
    We can verify that $\text{rank}(M_i) = 0$ holds for each element $i$ in $NZ(\mathbf{x})$ in polynomial time
    because we can calculate $\text{rank}(M_i)$ in polynomial time and $|NZ(x)| < 7n$ holds.

    By the above discussion, we can verify that $\mathbf{x}$ is the vector representation of a normal surface in $\mathcal{T}$ and
    the projection of $\mathbf{x}$ onto the hyperpalne $\sum_{i=1}^{7n} x_i = 1$ is a vertex solution in $\mathcal{P}$ in polynomial time.
    From the definition of a vertex surface, 
    the normal surface $F$ represented by $\mathbf{x}$ is a vertex surface if $F$ is connected and two-sided in $\mathcal{T}$.

    If the Euler characteristic of $F$ is not two, then $F$ is not a $2$-sphere, and
    if $F$ has boundaries, then $F$ is not a $2$-sphere.
    We can check these conditions in polynomial time, and so 
    we suppose that $F$ is a closed surface whose Euler characteristic is two.
    We show that $F$ is a vertex surface if and only if $\text{GCD}(x_1, \ldots, x_{7n}) = 1$, where $\text{GCD}(x_1, \ldots, x_{7n})$ is the greatest common divisor of all elements which satisfies $x_i \neq 0$.
    
    Assume that $F$ is a vertex surface which is a $2$-sphere.
    $\text{GCD}(x_1, \ldots, x_{7n}) \leq 2$ holds because $F$ is a vertex surface.
    If $\text{GCD}(x_1, \ldots, x_{7n}) = 2$, then $\frac{1}{2}\mathbf{x}$ is the vector representation of a normal surface which is a properly embedded projective plane.
    This contradicts that a projective plane can not be embedded in the  exterior of a link in a thickened closed orientable surface.
    Therefore, $\text{GCD}(x_1, \ldots, x_{7n}) = 1$.
    
    Assume that $\text{GCD}(x_1, \ldots, x_{7n}) = 1$.
    $F$ is a connected 2-sphere since $\text{GCD}(x_1, \ldots, x_{7n}) = 1$ and the Euler characteristic of $F$ is two.
    $\mathcal{T}$ and $F$ are orientable, and so $F$ is two-sided in $\mathcal{T}$. Therefore, $F$ is a vertex surface.
\end{proof}

\begin{lem}\label{lem:Ann_check}
    Assume that $\mathcal{T}$ is an $n$-tetrahedra triangulation of the exterior $M$ of a link in a thickened closed orientable surface $\mathcal{S} \times I$ and
    $\mathbf{x} \in \mathbb{Z}^{7n}$ is an integer vector.
    Then we can determine whether $\mathbf{x}$ is the vector representation of a vertex annulus $F$ in $M$ with respect to $\mathcal{T}$ such that $\partial F$ is in $\mathcal{S} \times \{0\} \cup \mathcal{S} \times \{1\}$ in polynomial time of $n$, where $\mathcal{S} \times \{k\}$ is a copy of supporting surface of $D$ in the boundary of $M$.
\end{lem}
\begin{proof}
    In the same argument as Lemma \ref{lem:S^2_check},
    we can verify that $\mathbf{x}$ is the vector representation of a normal surface with respect to  $\mathcal{T}$ and
    the projection of $\mathbf{x}$ onto the hyperpalne $\sum_{i=1}^{7n} x_i = 1$ is a vertex solution in $\mathcal{P}$ in polynomial time, where $\mathcal{P}$ is the projective solution space defined by the matching equations of $\mathcal{T}$.
    
    Let $F$ denote the normal surface represented by $\mathbf{x}$.
    If $F$ is an annulus whose boudary is in $\mathcal{S} \times \{0\} \cup \mathcal{S} \times \{1\}$, then the Euler characteristic of $F$ is zero and $F \cap (\mathcal{S} \times \{0\} \cup \mathcal{S} \times \{1\}) \neq \emptyset$.
    These conditions can be verified in polynomial time.
    We show that the normal surface $F$ which satisfies these conditions is a vertex annulus whose boundary is in $\mathcal{S} \times \{0\} \cup \mathcal{S} \times \{1\}$ if and only if $\text{GCD}(x_1, \ldots, x_{7n}) = 1$.
    
    Assume that $F$ is a vertex annulus whose boundary is in $\mathcal{S} \times \{0\} \cup \mathcal{S} \times \{1\}$.
    Since $F$ is a vertex surface, $\text{GCD}(x_1, \ldots, x_{7n}) \leq 2$.
    If $\text{GCD}(x_1, \ldots, x_{7n}) = 2$, then $\frac{1}{2}\mathbf{x}$ is the vector representation of a normal Moebius band $F'$ in $M$.
    However, the Moebius band can not be properly embedded in $\mathcal{S} \times I$.
    This is contradiction, and so $\text{GCD}(x_1, \ldots, x_{7n}) = 1$.
    
    Assume that $\text{GCD}(x_1, \ldots, x_{7n}) = 1$.
    $F$ is an annulus or a Moeibus band since $\text{GCD}(x_1, \ldots, x_{7n}) = 1$, $F$ has the Euler characteristic zero, and $F$ has boundaries.
    Since $\partial F$ is in $\mathcal{S} \times \{0\} \cup \mathcal{S} \times \{1\}$, $F$ is an annulus, for otherwise $F$ would be a properly embedded Moebius band in $\mathcal{S} \times I$.
    Since $M$ and $F$ are orientable, $F$ is two-sided in $M$.
    Therefore, $F$ is a vertex surface.
\end{proof}

\begin{cor}\label{cor:vertical_ann_check}
    Assume that $\mathcal{T}$ is an $n$-tetrahedra triangulation of the exterior $M$ of a link in a thickened closed orientable surface $\mathcal{S} \times I$ and
    $\mathbf{x} \in \mathbb{Z}^{7n}$ is an integer vector.
    Then we can determine whether $\mathbf{x}$ is the vector representation of a vertical vertex annulus $F$ in $M$ with respect to $\mathcal{T}$ in polynomial time of $n$.
\end{cor}
\begin{proof}
    By Lemma \ref{lem:Ann_check}, we can determine whether $\mathbf{x}$ is the vector representation of a vertex annulus $F$ in $M$ with respect to $\mathcal{T}$ such that $\partial F$ is in $\mathcal{S} \times \{0\} \cup \mathcal{S} \times \{1\}$ in polynomial time of $n$.
    $F$ is vertical if and only if $\partial F \cap \mathcal{S} \times \{0\} \neq \emptyset$ and $\partial F \cap \mathcal{S} \times \{1\} \neq \emptyset$.
    This condition can be checked in polynomial time of $n$.
\end{proof}


Consider the case where $\hat{D}$ is a knot in a thickened closed orientable surface $\mathcal{S} \times I$.
Recall that a properly embedded annulus $A$ in the exterior $M = \text{cl}(\mathcal{S} \times I - N(D))$ is calssicalization annulus if $\partial A \subset \mathcal{S} \times \{k\} (k = 0\text{ or } 1)$ and $A$ separetes $\partial N(D)$ and $\mathcal{S} \times \{1-k\}$.
We give an algorithm determining whether a given annulus is a classicalization annulus.
\begin{alg}\label{alg:classicalizarion_annulus}
    Assume that $\mathcal{T}$ is an $n$-tetrahedra triangulation of the exterior $M$ of a knot $\hat{D}$ in a thickened closed orientable surface $\mathcal{S} \times I$.
    Given a normal annulus $A$ in $M$ with respect to $\mathcal{T}$, do the following steps:
    \begin{enumerate}
        \item if $\partial A \cap \mathcal{S} \times \{0\} \neq \emptyset$ and $\partial A \cap \mathcal{S} \times \{1\} \neq \emptyset$, output ``no'', otherwise, run the following steps,
        \item assume that $\partial A \cap \mathcal{S} \times \{k\} = \emptyset$ $(k = 0 \text{ or } 1)$, and choose a path $P$ in the $1$-skeleton of $\mathcal{T}$ which connects $\mathcal{S} \times \{k\}$ and $\partial N(\hat{D})$,
        \item if $|A \cap P|$ is odd, then output ``yes'', otherwise, output ``no''.
    \end{enumerate}
\end{alg}
A normal surface with respect to $\mathcal{T}$ and a path in the $1$-skeleton intersect transversally.
Therefore, a normal annulus $A$ whose boudary is in $\mathcal{S} \times \{k\}$ separates $\partial N(\hat{D})$ and $\mathcal{S} \times \{1-k\}$ if and only if there is a path $P$ such that $|P \cap A|$ is odd and $P$ connects $\partial N(\hat{D})$ and $\mathcal{S} \times \{1-k\}$.
On the other hand, if $\partial A \subset \mathcal{S} \times \{k\}$, then for each pair of paths $P_1$ and $P_2$ which connect $\partial N(\hat{D})$ and $\mathcal{S} \times \{1-k\}$, $P_1 \cap A \equiv P_2 \cap A \pmod 2$.
Thus, Algorithm \ref{alg:classicalizarion_annulus} outputs ``yes'' if and only if $A$ is a classicalization annulus.

Next, we show that if $A$ is a vertex surface, then Algorithm \ref{alg:classicalizarion_annulus} runs in polynomial time.

\begin{lem}\label{lem:classicalization_alg_time}
    Suppose the same situation of Algorithm \ref{alg:classicalizarion_annulus} and $A$ is a vertex surface with respect to $\mathcal{T}$.
    Then, Algorithm \ref{alg:classicalizarion_annulus} runs in $\mathcal{O}(n^2)$ time, where $n$ is the number of tetrahedra in $\mathcal{T}$.
\end{lem}
\begin{proof}
    Let $\mathbf{x} = (x_1, \ldots, x_{7n})$ denote the vector representation of $A$.
    Since $A$ is a vertex surface, $\mathbf{x}$ can be read in $\mathcal{O}(n^2)$ time.
    The step 1 can be run in time $\mathcal{O}(n)$.
    A path $P$ is obtained by the following operation: choose a spanning tree of the $1$-skeleton of $\mathcal{T}$ and choose a path connecting $\partial N(D)$ and $\mathcal{S} \times\{k\}$ in the spanning tree, where $\mathcal{S} \times\{k\}$ is the supporting surface of $D$ which does not contain $\partial A$ in $\partial M$.
    Since this operation can be carried out in $\mathcal{O}(n)$ time, the step 2 runs in $\mathcal{O}(n)$, and the length of $P$ is $\mathcal{O}(n)$.
    We can calculate $|A \cap P|$ in $\mathcal{O}(n^2)$ time since the length of $P$ is $\mathcal{O}(n)$ and $x_i \leq 2^{7n-1}$ by Theorem \ref{thm:HLP}.
    Therefore, Algorithm \ref{alg:classicalizarion_annulus} runs in $\mathcal{O}(n^2)$ time.
\end{proof}

\begin{cor}\label{cor:classicalization_ann_check}
    Assume that $\mathcal{T}$ is an $n$-tetrahedra triangulation of the exterior $M$ of a knot in a thickened closed orientable surface $\mathcal{S} \times I$ and
    $\mathbf{x} \in \mathbb{Z}^{7n}$ is an integer vector.
    Then we can determine whether $\mathbf{x}$ is the vector representation of a classicalization vertex annulus $F$ in $M$ with respect to $\mathcal{T}$ in polynomial time of $n$.
\end{cor}
\begin{proof}
    By Lemma \ref{lem:Ann_check}, we can determine whether $\mathbf{x}$ is the vector representation of a vertex annulus $F$ in $M$ with respect to $\mathcal{T}$ such that $\partial F$ is in $\mathcal{S} \times \{0\} \cup \mathcal{S} \times \{1\}$ in polynomial time.
    By Lemma \ref{lem:classicalization_alg_time}, the vertex annulus $F$ can be determined whether $F$ is a classicalization annulus in polynomial time of $n$.
    Therefore, we can determine whether $\mathbf{x}$ is the vector representation of a classicalization vertex annulus $F$ in polynomial time of $n$.
\end{proof}

\subsection{The proof of Theorem \ref{thm:main}}
We say that a decision problem is in NP if there is a non-deterministic Turing machine which solves the problem in polynomial time.
It is possible to prove the main theorem using this definition, but 
we use an alternative definition using a string called a {\it witness}.
\begin{thm}
    The following propositions are equivalent.
    \begin{itemize}
        \item An decision problem is in NP.
        \item There is a non-deterministic Turing machine which solves the problem in polynomial time.
        \item There is a deterministic Turing machine $M$ which satisfies the following conditions:
    \begin{itemize}
        \item for any input $s$ whose answer is ``yes'', there is a witness $w$ whose length is polynomial length of $|s|$, $M$ accepts $(s, w)$ in polynomial time, where $|s|$ denotes the length of $s$.
        \item for any input $s$ whose answer is ``no'' and any string $w$, $M$ rejects $(s, w)$ in polynomial time.
    \end{itemize}
    \end{itemize}

\end{thm}

$\langle D \rangle$ denotes the oriented Gauss code of a virtual knot diagram $D$.
In order to prove the main theorem, we show that there is a deterministic Turing machine $M$ which satisfies the following conditions:
\begin{itemize}
    \item for any diagram $D$ which represents a classical knot, there is a witness $w$ which is polynomial length of $|\langle D \rangle|$, and $M$ accepts $(\langle D \rangle, w)$ in polynomial time,
    \item for any virtual knot diagram $D$ and any string $w$, if $D$ does not represent any classical knot, then $M$ rejects $(\langle D \rangle, w)$ in polynomial time.
\end{itemize}

\setcounter{section}{1}
\setcounter{thm}{0}
\begin{thm}
    Classical knot recognition is in NP.
\end{thm}
\setcounter{section}{5}
\begin{proof}
    Let $K$ be a virtual knot and $D$ be a diagram of $K$ with $c$ real crossings.
    Note that $g(F)$ represents the genus of a surface $F$, and we define $g(\emptyset) = 0$, where $\emptyset$ is the empty set.
    Let $M$ be the deterministic Turing machine defined as follows:
    \begin{enumerate}
        \item Check whether $w$ is a code of a sequence (possibly empty) of integer vectors .
        In that case, let $(\mathbf{x}_0, \ldots, \mathbf{x}_{m-1})$ be the sequence of integer vectors represented by $w$. Otherwise reject.
        
        \item Construct a triangulation $\mathcal{T}_0$ of the canonical exterior of $D$ from $\langle D \rangle$.
        Choose a triangulation of the supporting surface of $D$ in the boundary of $\mathcal{T}_0$, and $\mathcal{S}_0$ denotes it.
        
        \item For any integer $i = 0, \ldots, m-1$, do the following steps:
        \begin{itemize}
            \item If $\mathbf{x_i}$ represents a vertex $2$-sphere with respect to $\mathcal{T}_i$, then perform Operation \ref{ope:splitting}.
            Let $\mathcal{T}_{i+1}$ be the resulting triangulation and $\mathcal{S}_{i+1}$ be the boundary component of $\mathcal{T}_{i+1}$ which is obtained from $\mathcal{S}_i$.
            \item If $\mathbf{x}_i$ does not represents a vertex $2$-sphere, then reduce the vertices of $\mathcal{S}_i$ to one (if $\mathcal{S}_i$ has two or more vertices), and we denote the resulting triangulations obtained from $\mathcal{T}_i$ and $\mathcal{S}_i$ by $\mathcal{T}'_i$ and $\mathcal{S}'_i$, respectively.
            \begin{itemize}
                \item If $\mathbf{x_i}$ represents a clasicalization vertex annulus with respect to $\mathcal{T}'_i$, then accept.
                \item If $\mathbf{x}_i$ represents a vertical vertex annulus $A$ with respect to $\mathcal{T}'_i$, then perform Operation \ref{ope:des} using $A$, and let $\mathcal{T}_{i+1}$ be the resulting triangulation and $\mathcal{S}_{i+1}$ be the boundary component of $\mathcal{T}_{i+1}$ which is obtained from $\mathcal{S}'_i$.
                \item If $\mathbf{x}_i$ does not represent a classicalization vertex annulus or a vertical vertex annulus with respect to $\mathcal{T}'_i$, then reject.
            \end{itemize}
        \end{itemize}
        \item If $g(\mathcal{S}_m) = 0$, then accept.
        Otherwise reject. 
    \end{enumerate}

    \setcounter{claim}{0}
    \begin{claim}
        $K$ is a classical knot if and only if there is a witness $w$ such that $M$ accepts $(\langle D \rangle, w)$ and $w$ can be encoded with polynomial length of $c$.
    \end{claim}
    \noindent
    {\it Proof of Claim A.}
    Suppose that $K$ is a classical knot, and we show that there is a witness $w$ such that $M$ accepts $(\langle D \rangle, w)$ and $w$ can be encoded with polynomial length of $c$.
    Let $M_0$ denote the canonical exterior of $D$, $\mathcal{T}_0$ be the triangulation of $M_0$ constructed by the argument of Corollary \ref{cor:comp}, and $\mathcal{S}_0$ be a triangulation of the supporting surface of $D$ in the boundary of $\mathcal{T}_0$. 
    
    In the case where $g(\mathcal{S}_0)$ is zero, if $w$ is a code of the empty sequence, then $M$ accepts $(\langle D \rangle, w)$.
    Suppose that $g(\mathcal{S}_0)$ is not zero.
    In this case, we define a witness $w$ as follows:
    \begin{enumerate}
        \item Let $X$ be the empty sequence of integers vectors.
        \item For any integer $i$, add an integers vector $\mathbf{x}_i$ to $X$ inductively as follows:
        \begin{itemize}
            \item In the case where $g(\mathcal{S}_i)$ is zero, go step 3.
            \item In the case  where $g(\mathcal{S}_i)$ is not zero:
            \begin{itemize}
                \item If $M_i$ is reducible, then there is an essential vertex $2$-sphere $F_i$ in $M_i$ with respect to $\mathcal{T}_i$ by Theorem \ref{thm:essS2}.
                Then, add the vector representation $\mathbf{x}_i$ of $F_i$ to $X$, and go to the step 3.
                \item If $M_i$ is irreducible, then reduce the vertices of $\mathcal{S}_i$ to one (if $\mathcal{S}_i$ has two or more vertices), and we denote the resulting triangulations obtained from $\mathcal{T}_i$ and $\mathcal{S}_i$ by $\mathcal{T}'_i$ and $\mathcal{S}'_i$, respectively.
                By Theorem \ref{thm:K}, there is a vertical essential annulus in $M_i$ since $K$ is classical.
                Therefore, there is a vertical essential vertex annulus or a classicalization vertex annulus in $M_i$ with respect to $\mathcal{T}'_i$ by Theorem \ref{thm:essAn}.
                \begin{itemize}
                    \item If there is a classicalization vertex annulus $F_i$, then add the vector representation $\mathbf{x}_i$ of $F_i$ to $X$, and go to the step 3.
                    \item If there are no classicalization vertex annuli and there is a vertical essential vertex annulus $F_i$, then add the vector representation $\mathbf{x}_i$ of $F_i$ to $X$, run Operation \ref{ope:des} on $\mathcal{T}'_i$ using $F_i$, and reduce the vertices of the triangulated surface obtained from $\mathcal{S}_i$ to one by the argument used in Lemma \ref{lem:one_vertex}.
                    Let $\mathcal{T}_{i+1}$ be the resulting triangulation, $M_{i+1}$ be the underlying $3$-manifold, and $\mathcal{S}_{i+1}$ be the boundary surface in $\mathcal{T}_{i+1}$ obtained from $\mathcal{S}'_i$.
                    Then, repeat the step 2 for the integer $i+1$.
                \end{itemize}
            \end{itemize}
        \end{itemize}
        \item Encode $X$.
    \end{enumerate}
    For each $\mathcal{S}_i$, $g(\mathcal{S}_{i+1})$ is less than $g(\mathcal{S}_i)$.
    This implies that the length of $X$ is at most $g(\mathcal{S}_0)$, and so the length of $X$ is $\mathcal{O}(c)$ by Corollary \ref{cor:genus}, in particular, $X$ is a finite sequence.
    
    We show that $M$ accepts $(\langle D \rangle, w)$.
    Suppose that $X = (\mathbf{x}_0, \ldots, \mathbf{x}_{m-1})$.
    By the construction of $X$, the normal surfaces represented by $\mathbf{x}_0, \ldots, \mathbf{x}_{m-2}$ are vertical annuli.
    Let $F_{m-1}$ denote the normal surface represented by $\mathbf{x}_{m-1}$.
    
    First, suppose that $F_{m-1}$ is a $2$-sphere, then the underlying $3$-manifold of $\mathcal{T}_m$ is $\text{cl}(\mathbb{S}^3 - N(\hat{D}))$ or the empty set by Corollary \ref{cor:splitting_comp_knot}.
    In either case, $\mathcal{S}_m$ is the empty set, and so $M$ accepts $(\langle D \rangle, w)$ since we define $g(\emptyset) = 0$.
    Next, suppose that $F_{m-1}$ is a vertical annulus.
    This implies that $g(\mathcal{S}_m) = 0$, thus, $M$ accepts $(\langle D \rangle, w)$.
    Finally, if $F_{m-1}$ is a classicalization annulus, then $M$ accepts $(\langle D \rangle, w)$.

    Consider the length of the witness $w$ obtained by the above operation.
    By Theorem \ref{thm:HLP}, each $\mathbf{x}_i$ can be encoded with $\mathcal{O}(|\mathcal{T}_i|^2)$ length, where $|\mathcal{T}_i|$ is the number of tetrahedra in $\mathcal{T}_i$.
    By the construction of $\mathcal{S}'_i$, each $\mathcal{S}'_i$ has exactly one vertex, thus $|\mathcal{T}_{i+1}|$ is $|\mathcal{T}_{i}| + \mathcal{O}(c)$ by Lemma \ref{lem:desNum}.
    Since $|\mathcal{T}_0|$ is $\mathcal{O}(c)$ and the length of $X$ is $\mathcal{O}(c)$, each $|\mathcal{T}_i|$ is $\mathcal{O}(c^2)$.
    Therefore, each $\mathbf{x}_i$ is encoded with $\mathcal{O}(c^4)$ length, and the witness $w$ obtained by the above operation is encoded with $\mathcal{O}(c^5)$ length.

    Next, assume that $M$ accepts $(\langle D \rangle, w)$, and we show that $K$ is classical.
    The Turing machine $M$ accepts $(\langle D \rangle, w)$ if and only if $g(\mathcal{S}_m) = 0$ or there is a classicalization annulus $F_i$, where $F_i$ is the normal surface represented by $\mathbf{x}_i$.
    First, if there is a classicalization annulus $F_i$, then $K$ is a classical knot by Lemma \ref{lem:classicalization_annulus}.
    Next, suppose that there are no classicalization annuli and $g(\mathcal{S}_m)$ is zero.
    Let $M_m$ denote the underlying $3$-manifold of $\mathcal{T}_m$.
    From Corollary \ref{cor:splitting_comp_knot}, Corollary \ref{cor:des_comp_knot} and $g(\mathcal{S}_m) = 0$, 
    $M_m$ is one of the followings:
    \begin{itemize}
        \item $\text{cl}(\mathbb{S}^2 \times I - N(\hat{D}))$
        \item $\mathbb{B}^3 - N(\hat{D})$,
        \item $\mathbb{S}^3 - N(\hat{D})$,
        \item the empty set.
    \end{itemize}
    If $M_m$ is not the empty set, then $K$ is a classical knot, and if $M_m$ is the empty set, then $K$ is the trivial knot by Corollary \ref{cor:splitting_comp_knot} and Corollary \ref{cor:des_comp_knot}.
    \qed

    By Claim A, if $L$ is classical, there is a witness $w$ such that $M$ accepts $(\langle D \rangle, w)$.
    If $M$ accepts $(\langle D \rangle, w)$, $D$ is a diagram of a classical knot. Considering the contraposition of this proposition, 
    for any virtual knot diagram $D$ and any string $w$, if $D$ dose not represent any classical knot, then $M$ rejects $(\langle D \rangle, w)$.

    Lastly, we analyze the running time of the deterministic Turing machine $M$. 
    \begin{claim}
        $M$ runs in polynomial time of $|(\langle D \rangle, w)|$.
    \end{claim}
    \noindent
    {\it Proof of Claim B.}
    It is clear that the step 1 runs in polynomial time of $|(\langle D \rangle, w)|$.
    By Corollary \ref{cor:comp}, the step 2 runs in $\mathcal{O}(c)$ time, where $c$ is the number of real crossings in $D$.
    For each $\mathcal{T}_i$, $|\mathcal{T}_i|$ is $\mathcal{O}(c^2)$ by the discussion in the proof of Claim A.
    \begin{itemize}
        \item The vector $\mathbf{x}_i$ can be determined whether $\mathbf{x}_i$ represents a vertex $2$-sphere with respect to $\mathcal{T}_i$ in polynomial time of $|\mathcal{T}_i|$ by Lemma \ref{lem:S^2_check}.
        \begin{itemize}
            \item In the case where $\mathbf{x}_i$ represents a vertex $2$-sphere with respect to $\mathcal{T}_i$, Operation \ref{ope:splitting} runs in polynomial time of $|\mathcal{T}_i|$.
            \item In the case where $\mathbf{x}_i$ does not represent a vertex $2$-sphere with respect to $\mathcal{T}_i$, we can reduce the vertices of $\mathcal{S}_i$ to one in $\mathcal{O}(|\mathcal{T}_i|^3) = \mathcal{O}(c^6)$ time by Lemma \ref{lem:one_vertex}.
            Note that $|\mathcal{T}'_i|$ is less than $|\mathcal{T}_i|$.
            \begin{itemize}
                \item The vector $\mathbf{x}_i$ can be determined whether $\mathbf{x}_i$ represents a classicalization vertex annulus with respect to $\mathcal{T}'_i$ in polynomial time of $|\mathcal{T}'_i|$ by Corollary \ref{cor:classicalization_ann_check}.
                \item The vector $\mathbf{x}_i$ can be determined whether $\mathbf{x}_i$ represents a vertical vertex annulus with respect to $\mathcal{T}'_i$ in polynomial time of $|\mathcal{T}'_i|$ by Corollary \ref{cor:vertical_ann_check}, and Operation \ref{ope:des} also runs in polynomial time of $|\mathcal{T}_i|$.
            \end{itemize}
        \end{itemize}

    \end{itemize}
    The step 3 is repeated at most $\mathcal{O}(c)$ times since $g(\mathcal{S}_0)$ is $\mathcal{O}(c)$, therefore, the step 3 runs in polynomial time of $c$.
    The genus of $\mathcal{S}_m$ is calculated in $\mathcal{O}(|\mathcal{T}_m|) = \mathcal{O}(c^2)$ time.
    
    Since the steps 2--4 run in polynomial time of $c$ and $c$ is $\mathcal{O}(|(\langle D \rangle, w)|)$, each step runs in polynomial time of $|(\langle D \rangle, w)|$.
    Therefore, the running time of $M$ is polynomial time of $|(\langle D \rangle, w)|$.
\end{proof}

\section{Algorithm for classical knot recognition}
In this section, we give an algorithm for classical knot recognition, prove the correctness of the algorithm, and analyze the running time of the algorithm.
Let $D$ be a virtual knot diagram and $M = \text{cl}(\mathcal{S} \times I - N(\hat{D}))$ be the canonical exterior of $D$ with a triangulation $\mathcal{T}$.
We can solve classical knot recognition by finding an essential $2$-sphere, a classicalization annulus, or an essential vertical annulus in $M$ and performing Operation \ref{ope:splitting} or Operation \ref{ope:des}.
If $g(\mathcal{S}) = 0$, then $D$ is a classical knot diagram, thus suppose that $g(\mathcal{S}) \neq 0$.
If $M$ is reducible, then there is an essential vertex $2$-sphere in $M$ with respect to $\mathcal{S}$ by Theorem \ref{thm:essS2}, and if $M$ is irreducible and there is an vertical essential annulus, then there is a vertical essential vertex annulus or a classicalization vertex annulus in $M$ with respect to $\mathcal{T}$ by Theorem \ref{thm:essAn}.
Thus, we can find these surfaces by enumerating the vertex surfaces in $M$ with respect to $\mathcal{T}$ if there are these surfaces.

\subsection{Algorithm and its correctness}
Note that $g(\emptyset)$ is defined as $0$, where $\emptyset$ is the empty set.
We describe an algorithm for classical knot recognition.

\begin{alg}\label{alg:classical}
    Let $K$ be a virtual knot and $D$ be a diagram of $K$.
    \begin{enumerate}
        \item Construct a triangulation $\mathcal{T}$ of the canonical exterior of $D$. There are two triangulations of the supporting surface of $D$ in $\partial \mathcal{T}$, and denote one of them by $\mathcal{S}$.
        \item Output ``yes'' if $g(\mathcal{S}) = 0$, otherwise run the following steps.
        \item Enumerate all vertex solutions in the projective solution space defined by the matching equations of $\mathcal{T}$, and $V_{\mathcal{T}}$ denotes the list of all vertex solutions.
        \item For each $\mathbf{x}_i \in V_\mathcal{T}$, calculate the minimum positive integer $k_i$ so that $k_i\mathbf{x}_i$ is an integers vector.
        \item Output ``yes'' if there is a vector $\mathbf{x}_i \in V_\mathcal{T}$ which satisfies the following conditions:
        \begin{itemize}
            \item The vector $k_i\mathbf{x}_i$ represents a vertex $2$-sphere $F$ with respect to $\mathcal{T}$.
            \item The triangulation $\mathcal{S}'$, the triangulation obtained from $\mathcal{S}$ by Operation \ref{ope:splitting} using $F$, is the empty set, i.e., the resulting triangulation of Operation \ref{ope:splitting} is $\text{cl}(\mathbb{S}^3 - N(\hat{D}))$.
        \end{itemize}
        
        \item Reduce the vertices of $\mathcal{S}$ to one, and $\mathcal{T}'$ and $\mathcal{S}'$ denote the triangulation obtained from $\mathcal{T}$ and $\mathcal{S}$, respectively.
        \item Enumerate all vertex solutions in the projective solution space defined by the matching equations of $\mathcal{T}'$, and $V_{\mathcal{T}'}$ denotes the list of all vertex solutions.
        \item For each $\mathbf{x}'_i \in V_{\mathcal{T}'}$, calculate the minimum positive integer $k'_i$ so that $k'_i\mathbf{x}'_i$ is an integers vector.
        
        \item Output ``yes'' if there is a vector $\mathbf{x}'_i$ such that $k'_i\mathbf{x}'_i$ represents a classicalization vertex annulus with respect to $\mathcal{T}'$.
        \item Check whether there is a vector $\mathbf{x}'_i \in V_{\mathcal{T}'}$ which satisfies the following conditions:
        \begin{itemize}
            \item The vector $k'_i\mathbf{x}'_i$ represents a vertical vertex annulus $A$ with respect to $\mathcal{T}'$.
            \item $g(\mathcal{S}'') < g(\mathcal{S}')$, where $\mathcal{S}''$ is the triangulation obtained from $\mathcal{S}'$ by Operation \ref{ope:des} using $A$.
        \end{itemize}
        \begin{enumerate}
            \item If there is a vector $\mathbf{x}'_i$ which satisfies the above conditions, then run Operation \ref{ope:des} using $A$ on $\mathcal{T}'$, where $A$ is the vertical annulus represented by $k'_i\mathbf{x}'_i$.
            Redefine $\mathcal{T}$ as the resulting triangulation of Operation \ref{ope:des} and $\mathcal{S}$ as the boundary of the resulting triangulation $\mathcal{T}$ obtained from $\mathcal{S}'$, and then go to the step 2.
            \item Otherwise, output ``no''.
        \end{enumerate}
    \end{enumerate}
\end{alg}

\begin{thm}\label{thm:correctness}
    Algorithm \ref{alg:classical} outputs ``yes'' if and only if $K$ is a classical knot.
\end{thm}
\begin{proof}
    If Algorithm \ref{alg:classical} outputs ``yes'', then $g(\mathcal{S}) = 0$, there is an essential $2$-sphere, or there is a classicalization annulus in the underlying $3$-manifold of $\mathcal{T}$.
    In any case, we see that $K$ is a classical knot.
    
    Suppose that $K$ is a classical knot.
    Let $\mathcal{T}$ be the triangulation of the canonical exterior constructed in the step 1 of Algorithm \ref{alg:classical} and $\mathcal{S}$ be the triangulation of the supporting surface of a given diagram $D$ in the boundary of $\mathcal{T}$.
    If $g(\mathcal{S}) = 0$, then Algorithm \ref{alg:classical} outputs ``yes'' in the step 2, and so we consider the case where $g(\mathcal{S}) \neq 0$.
    Let $M$ denote the underlying $3$-manifold of $\mathcal{T}$.
    If $M$ is reducible, then there is an essential vertex $2$-sphere by Theorem \ref{thm:essS2}, and so Algorithm \ref{alg:classical} outputs ``yes'' in the step 5.
    Suppose that $M$ is irreducible.
    Let $\mathcal{T}'$ denotes the triangulation of $M$ constructed in the step 6.
    If $M$ is irreducible, then there is a vertical essential annulus in $M$ by Theorem \ref{thm:K}.
    Therefore, there is a vertex surface which is an essential vertical annulus or a classicalization annulus in $M$ with respect to $\mathcal{T}'$ by Theorem \ref{thm:essAn}. 
    If there is a classicalization vertex annulus, then Algorithm \ref{alg:classical} outputs ``yes'' in the step 6.
    Otherwise, there is a vertical essential vertex annulus $A$.
    In this case, Operation \ref{ope:des} using $A$ on $\mathcal{T}'$ reduces the genus of $\mathcal{S}$.
    Since $K$ is classical, $g(\mathcal{S})$ is reduced to $0$ (or an essential $2$-sphere or a classicalization annulus in $M$ is found in the middle of reducing $g(\mathcal{S})$).
    Thus, Algorithm \ref{alg:classical} outputs ``yes''.
\end{proof}

\subsection{Double description method}
In order to enumerate vertex solutions in a projective solution space, we use {\it double description method} (\cite{B_double}).
In this subsection, assume that $A$ is a $(s,t)$-matrix, where $s$ and $t$ are natural numbers, and 
$\mathcal{P} = \{\mathbf{x} \in \mathbb{R}^t \ |\  \mathbf{x} \geq \mathbf{0}, \sum x_i = 1, A\mathbf{x} = \mathbf{0}\}$ is the projective solution space defined by $A$.
We denote the $i$-th row of $A$ by $\mathbf{m}_i$.
Then $\mathcal{P}$ is the intersection of the following subspaces of $\mathbb{R}^t$:
\begin{itemize}
    \item $O = \{\mathbf{x} \in \mathbb{R}^t \ |\  \mathbf{x} \geq \mathbf{0}\}$
    \item the hyperplane $J = \{\mathbf{x} \in \mathbb{R}^t \ | \ \sum x_i = 1\}$
    \item the hyperpalnes $H_1, H_2,\ldots, H_s$, where $H_i$ is the solution space of the equation $\mathbf{m}_i \cdot \mathbf{x} = 0$.
\end{itemize}

\begin{alg}[Double description method]\label{alg:double}
    Suppose that a $(s,t)$-matrix $A$ is given.
    Let $\mathcal{P}$ denote the projective solution space defined by $A$.
    Then we can obtain all vertex solutions in $\mathcal{P}$ as follows.

    We define $P_0$ as $O \cap J$, and 
    for each $i = 1,\ldots,s$, 
    we define $P_i$ as $O \cap J \cap H_1 \cap \cdots \cap H_i$.
    Let $V_i$ be the set of all vertex solutions of $P_i$ for each $i = 0,\ldots,s$.
    In particular, $P_s$ is $\mathcal{P}$, and $V_s$ is the set of all vertex solutions in $\mathcal{P} = P_s$.
    Then $V_i$ can be obtained from $V_{i-1}$ by the following operations:
    \begin{enumerate}
        \item Let $V_0 = \{\mathbf{e}_i | 1 \leq i \leq t\}$, where $\mathbf{e}_i \in \mathbb{R}^t$ is the vector whose coordinate $x_j$ is one if $j = i$, otherwise $x_j$ is zero.
        \item For each $i = 1, \ldots, s$, $V_i$ is defined as follows:
        \begin{enumerate}
            \item Split $V_{i-1}$ to $S_0$, $S_+$ and $S_-$, where these sets are defined as follows: 
            \begin{itemize}
                \item $S_0 = \{\mathbf{v} \in V_{i-1}\ |\ \mathbf{m}_i\cdot \mathbf{v} = 0\}$
                \item $S_+ = \{\mathbf{v} \in V_{i-1}\ |\ \mathbf{m}_i\cdot \mathbf{v} > 0\}$
                \item $S_- = \{\mathbf{v} \in V_{i-1}\ |\ \mathbf{m}_i\cdot \mathbf{v} < 0\}$
            \end{itemize}
            In other words, $S_0$, $S_+$ and $S_-$ contain the element of $V_{i-1}$ that lies in, above and below the hyperplane $H_i$, respectively.
            \item Add all elements of $S_0$ to $V_i$.
            \item For each pair $(\mathbf{u}, \mathbf{w}) \in S_+ \times S_-$, 
                    if $\mathbf{u}$ and $\mathbf{w}$ are adjacent in $P_{i-1}$, 
                    then add the intersection point of $H_i$ and the line segment connecting $\mathbf{u}$ and $\mathbf{w}$ to $V_i$.
        \end{enumerate}
    \end{enumerate}
\end{alg}

While we do not describe in detail how to determine whether two points $\mathbf{u}$ and $\mathbf{w}$ are adjacent in $P_{i-1}$ in the step 2-(c) of Algorithm \ref{alg:double},
one of the methods is described below.

Fukuda and Prodon proved the following lemma in more general case, and
Burton adapt it to the projective solution space.
Note that $Z(\mathbf{x}) = \{k \in \mathbb{N}\ |\ x_k = 0\}$.
\begin{lem}[Fukuda and Prodon \cite{double}, Burton \cite{B_double}]
    Let $P_{i-1}$, $\mathbf{u}$ and $\mathbf{w}$ be the subspace of $\mathbb{R}^t$ and vertex solutions of $P_{i-1}$ in the step 4-(c) of Algorithm \ref{alg:double}.
    Then $\mathbf{u}$ and $\mathbf{w}$ are adjacent in $P_{i-1}$ if and only if 
    the dimension of the intersection space of $H_1 \cap \cdots \cap H_{i-1}$ and the hyperpalne $ \bigcap_{k \in Z(\mathbf{u}) \cap Z(\mathbf{w})} \{\mathbf{x} \in \mathbb{R}^t | x_k = 0\}$ is two.
\end{lem}

We describe how to check that 
the dimension of the intersection space of $H_1 \cap \cdots \cap H_{i-1}$ with $ \bigcap_{k \in Z(\mathbf{u}) \cap Z(\mathbf{w})} \{\mathbf{x} \in \mathbb{R}^t | x_k = 0\}$ is two.
$A_{i-1}$ denotes the submatrix obtained by deleting all rows below the $i$-th row from $(s,t)$-matrix $A$.
Let $\mathbf{e}_k \in \mathbb{R}^t$ be the unit vector whose $k$-th coordinate is $1$ and the other coordinates are $0$, and 
let $I_{Z(\mathbf{u}) \cap Z(\mathbf{w})}$ be the matrix obtained by deleting $k$-th rows which satisfies $k \notin Z(\mathbf{u}) \cap Z(\mathbf{w})$ from the $t \times t$ identity matrix.
Then we consider the matrix
\[
M = \left(
\begin{array}{cc}
        A_{i-1}\\
        I_{Z(\mathbf{u}) \cap Z(\mathbf{w})}
    \end{array}
    \right).
\]
The dimension of the intersection of $H_1 \cap \cdots \cap H_{i-1}$ and $\bigcap_{k \in Z(\mathbf{u}) \cap Z(\mathbf{w})} \{\mathbf{x} \in \mathbb{R}^t | x_k = 0\}$
is obtained by calculating the dimension of the kernel of M; $\text{dim}(\text{ker}(M))$ denotes it, where $\text{ker}(M)$ is the kernel of $M$.
Since the equation $\text{dim}(\text{ker}(M)) = t - \text{rank(M)}$ holds,
we can check whether two vectors $\mathbf{u}$ and $\mathbf{w}$ are adjacent in $P_{i-1}$ by calculating the rank of $M$.

Next, we analyze the running time of Algorithm \ref{alg:double}.
The rank of $M$ can be calculated in time $\mathcal{O}((s+t)^3)$ because the number of rows of $M$ is at most $s+t$.
We can analyze the running time of Algorithm \ref{alg:double} if the size of $V_i$ is known.
We can show lemma \ref{lem:ver_sol} in the same way as the discussion of the number of vertex surfaces in \cite{B_vertex}.
\begin{lem}\label{lem:ver_sol}
    Let $A$ be an input matrix of Algorithm \ref{alg:double}, and suppose that $A$ is a $(s,t)$-matrix.
    For each $i = 0, \ldots ,s$, $|V_i| \in \phi^{\mathcal{O}(t)}$ holds, where $\phi = \frac{1 + \sqrt{5}}{2}$.
\end{lem}

\begin{lem}
    Let $A$ be an input matrix of Algorithm \ref{alg:double}, and suppose that $A$ is a $(s,t)$-matrix.
    Then, Algorithm \ref{alg:double} runs in time $\mathcal{O}(s(s+t)^3)\phi^{\mathcal{O}(t^2)}$.
\end{lem}
\begin{proof}
    We analyze the time to construct $V_i$ from $V_{i-1}$.
    By Lemma \ref{lem:ver_sol}, $|V_{i-1}| \in \phi^{\mathcal{O}(t)}$ holds, so that $|S_0|, |S_+|$ and $|S_-| \in \phi^{\mathcal{O}(t)}$.
    Therefore, it takes $\phi^{\mathcal{O}(t)}$ time to add $|S_0|$ to $V_i$.
    For each pair $(\mathbf{u}, \mathbf{w}) \in S_+ \times S_-$, 
    we can determine whether two vectors $\mathbf{u}$ and $\mathbf{w}$ are adjacent in $P_{i-1}$ in time $\mathcal{O}((s+t)^3)$.
    Furthermore, we can calculate the intersection point of $H_i$ and the line segment connecting $\mathbf{u}$ and $\mathbf{w}$ in time $\mathcal{O}(t)$.
    There are $\phi^{\mathcal{O}(t^2)}$ pairs in $S_+ \times S_-$, so that 
    we can construct $V_i$ in time $\mathcal{O}((s+t)^3)\phi^{\mathcal{O}(t^2)}$.
    Because $i \leq s$, Algorithm \ref{alg:double} runs in time $\mathcal{O}(s(s+t)^3)\phi^{\mathcal{O}(t^2)}$ .
\end{proof}

To enumerate vertex surfaces in a triangulation $\mathcal{T}$, 
we consider enumerating vertex solutions in the projective solution space defined by the matching equations of $\mathcal{T}$.
Suppose that $\mathcal{T}$ has $n$ tetrahedra and $A$ is the matching equations of $\mathcal{T}$.
Since the number of rows of $A$ is $7n$, and the number of columns is at most $6n$, the following corollary holds.
\begin{cor}\label{cor:ver_sol_time}
    Let $\mathcal{T}$ be an $n$-tetrahedra triangulation, and let $A$ be the matrix defined by the matching equations of $\mathcal{T}$.
    Then, we can enumerate all vertex solutions in the projective solution defined by $A$ in time $\phi^{\mathcal{O}(n^2)}$.
    Futhermore, the number of the vertex solutions is $\phi^{\mathcal{O}(n)}$.
\end{cor}

\subsection{The running time of Algorithm \ref{alg:classical}}
Next, we give a method to calculate the minimum positive integer $k$ so that $k\mathbf{x}$ is an integer vector, where $\mathbf{x}$ is a vertex solution of the matching equations of a triangulation $\mathcal{T}$ of a $3$-manifold.
We can calculate the integer $k$ in polynomial time of the number of tetrahedra in $\mathcal{T}$.
\begin{lem}\label{lem:GCD}
    Let $\mathcal{T}$ be a triangulation which has $n$ tetrahedra.
    Given a vertex solution $\mathbf{x}$ in the projective solution space defined by the matching equations of $\mathcal{T}$,
    we can calculate the minimum positive integer $k$ so that $k\mathbf{x}$ an integer vector in time $\mathcal{O}(n^3 \log n)$.
\end{lem}
\begin{proof}
    Let $NZ(\mathbf{x}) = \{i \in \{1, \ldots, 7n\} | x_i \neq 0\}$, and 
    let $m$ be the size of $NZ(\mathbf{x})$.
    $x_i$ is a non-zero rational number if $i \in NZ(\mathbf{x})$, so that let $x_i = \frac{a_i}{b_i}$, where $a_i, b_i \neq 0$ are natural numbers.
    We can assume that $a_i$ and $b_i$ are relatively prime by reducing $x_i$ when we calculate $\mathbf{x}$.
    Now, $k$ is the least common multiple of $b_{i_j}$, where $i_j \in NZ(\mathbf{x})$. We denote this by $\text{LCM}(b_1, \ldots, b_{7n})$.
    We can calculate $\text{LCM}(b_1, \ldots, b_{7n})$ by calculating the least common multiple of two natural numbers $\mathcal{O}(\log n)$ times.

    Let $\alpha$ and $\beta$ be natural numbers which are not $0$, and we assume that $\alpha \geq \beta$.
    In general, $\text{LCM}(\alpha, \beta)$ can be obtained by dividing the product of $\alpha$ and $\beta$ by the greatest common divisor of $\alpha$ and $\beta$, denoted by $\text{GCD}(\alpha, \beta)$.
    We can calculate $\text{GCD}(\alpha, \beta)$ by modulo operation $\mathcal{O}(\log \alpha)$ times by using Euclidean Algorithm.
    Multiplication, division, and modulo operation of $\alpha$ and $\beta$ can be carried out in time $\mathcal{O}((\log \alpha)^2)$, and so 
    we can calculate $\text{LCM}(\alpha, \beta)$ in time $\mathcal{O}((\log \alpha)^3)$.

    Next, we analyze the size of $b_i$.
    We have $x_i = \frac{k x_i}{\sum k x_j}$ because $\mathbf{x}$ is the projection of $k\mathbf{x}$ onto the hyperpalne $\sum x_i = 1$.
    By Theorem \ref{thm:HLP}, for each integer $j = 1, \ldots , 7n$, $k x_j \leq 2^{7n-1}$ holds, and so we have $\sum k x_j \leq 7n2^{7n-1}$.
    Now, we see that $a_i$ and $b_i$ are relatively prime and $x_i = \frac{a_i}{b_i} = \frac{k x_i}{\sum k x_j}$, and so 
    $b_i \leq 7n2^{7n-1}$ holds.
    On the other hand, for any integers $i_j$ and $i_{j'}$ in $NZ(\mathbf{x})$,   
    we have $\text{LCM}(b_{i_j}, \ldots, b_{i_{j'}}) \leq  \text{LCM}(b_1, \ldots, b_{7n}) = k = \sum kx_j \leq 7n2^{7n-1}$.
    Therefore, we can calculate $k = \text{LCM}(b_1, \ldots, b_{7n})$ in time $\mathcal{O}((\log n)(\log(7n2^{7n-1}))^3) = \mathcal{O}(n^3 \log n)$.
\end{proof}

\begin{thm}\label{thm:time}
    Algorithm \ref{alg:classical} runs in time $\phi^{\mathcal{O}(c^4)}$, 
    where $c$ is the number of real crossings of a given diagram and $\phi = \frac{1+\sqrt{5}}{2}$.
\end{thm}
\begin{proof}
    The step 1 runs in time $\mathcal{O}(c)$ by Lemma \ref{cor:comp}.
    Furthermore, $|\mathcal{T}|$ is $\mathcal{O}(c)$, where $\mathcal{T}$ is the triangulation obtained by the argument of Lemma \ref{cor:comp}.
    
    We consider the number of tetrahedra in $\mathcal{T}$.
    Recall that the genus of a virtual knot diagram $D$, denoted by $sg(D)$, is the genus of the supporting surface of $D$.
    The steps 2--10 are repeated at most $sg(D)$ times since the step 10-(a) must reduce the genus of $\mathcal{S}$.
    Since $sg(D)$ is $\mathcal{O}(c)$ by Corollary \ref{cor:genus}, the steps 2--10 run $\mathcal{O}(c)$ times.
    The triangulation $\mathcal{T}$ may be changed in the step 6 and the step 10-(a).
    The step 6 does not increase the number of tetrahedra, and the step 10-(a) increases the number of tetrahedra by $\mathcal{O}(c)$ by Lemma \ref{lem:desNum}.
    Thus, the number of tetrahedra of $\mathcal{T}$ is $\mathcal{O}(c^2)$.
    
    The step 2 runs in time $\mathcal{O}(c^2)$ since $g(\mathcal{S})$ is calculated in time $\mathcal{O}(|\mathcal{S}|) = \mathcal{O}(|\mathcal{T}|) = \mathcal{O}(c^2)$.
    By Corollary \ref{cor:ver_sol_time}, the list of all vertex solutions is calculated in time  $\phi^{\mathcal{O}(|\mathcal{T}|^2)} = \phi^{\mathcal{O}(c^4)}$, and so the step 3 and the step 7 run in time $\phi^{\mathcal{O}(c^4)}$.
    Let $V_\mathcal{T}$ denotes the list of all vertex solutions in the projective solution space defined by the matching equation of $\mathcal{T}$.
    For each $\mathbf{x} \in V_\mathcal{T}$, we can calculate the minimum positive integer $k$ so that $k\mathbf{x}$ is an integers vector in time $\mathcal{O}(|\mathcal{T}|^3 \log |\mathcal{T}|) = \mathcal{O}(c^6 \log c^2)$ time by Lemma \ref{lem:GCD}.
    Thus, the step 4 and the step 8 run in time $\phi^{\mathcal{O}(|\mathcal{T}|)} = \phi^{\mathcal{O}(c^2)}$ time.
    
    Consider the running time of the step 5.
    For each $k\mathbf{x}$, we can determine whether $k\mathbf{x}$ represents a vertex $2$-sphere with respect to $\mathcal{T}$ in polynomial time by Lemma \ref{lem:S^2_check}, and Operation \ref{ope:splitting} runs in time $\mathcal{O}(|\mathcal{T}|^2) = \mathcal{O}(c^4)$.
    Since the length of $V_\mathcal{T}$ is $\phi^{\mathcal{O}(c^2)}$, the step 5 runs in time $\phi^{\mathcal{O}(c^2)}$.
    In the same argument, the step 9 and the step 10 run in time $\phi^{\mathcal{O}(c^2)}$ by Lemma \ref{lem:Ann_check} and Lemma \ref{lem:classicalization_alg_time}.
    
    We see that the running time of each step of Algorithm \ref{alg:classical} is less than $\phi^{\mathcal{O}(c^4)}$, and the steps 2--10 runs $\mathcal{O}(c)$ times.
    Therefore, Algorithm \ref{alg:classical} runs in time $\phi^{\mathcal{O}(c^4)}$.
\end{proof}

\section{Experimental performance}
By Theorem \ref{thm:time}, Algorithm \ref{alg:classical} runs in time $\phi^{\mathcal{O}(c^4)}$.
However, Burton verified that Lemma \ref{lem:ver_sol} is not a sharp bound of the number of vertex solutions by the computational experiment in \cite{B_vertex}.
Thus, it is expected that we can get a better bound of the running time of Algorithm \ref{thm:time}.
We describe the result of testing of Algorithm \ref{alg:classical} for some virtual knot diagrams and estimate a better bound of the running time of the algorithm.

\subsection{Implementation}
Algorithm \ref{alg:classical} has been implemented in C++ with {\it Regina}.
Regina is a software package to calculate for low-dimensional topology and is developed by Burton et al.(\cite{B_regina, regina}).
Regina specializes in calculation regarding triangulations and normal surfaces, and so 
it is suitable for implementation of algorithms which uses essential surfaces in triangulations.

Let $\mathcal{T}$ be the triangulation of the canonical exterior of an input diagram which is constructed by the argument of Corollary \ref{cor:comp}.
If Algorithm \ref{alg:classical} is implemented simply, it takes too much time to enumerate vertex solutions because $\mathcal{T}$ has many tetrahedra.
In the experiment, the number of tetrahedra in $\mathcal{T}$ is reduced heuristically when $\mathcal{T}$ is constructed and when $\mathcal{T}$ is changed.
The detail of the operation to reduce the number of tetrahedra is written in \cite{burton2013computational}.

\subsection{Input data}
The input data are the diagrams in the table of virtual knots made by Green (\cite{G}).
We use the oriented Gauss codes of virtual knot diagrams which have five or less real crossings in the table of virtual knots.
The diagrams in the table are encoded by a different way from oriented Gauss code.
However, it is easy to convert them.

\subsection{Experimental method}
In Algorithm \ref{alg:classical}, the part which takes the longest time is the enumeration of the vertex solutions.
The running time of this part increases exponentially with increasing the number of tetrahedra.
Thus, it is not possible to run the algorithm in realistic time if a triangulation of the canonical exterior of an input diagram has many tetrahedra.
Therefore, we conduct the experiment as follows:
\begin{enumerate}
    \item construct all triangulations of the canonical exteriors of input diagrams by the argument of Corollary \ref{cor:comp},
    \item reduce the number of tetrahedra in the triangulations heuristically, and 
    \item run the program of Algorithm \ref{alg:classical} in ascending order of the size of the triangulations.
\end{enumerate}

In this experiment, we measure the running time for the diagrams whose canonical exterior triangulations have 27 or less tetrahedra.

\subsection{Experimental result}
The experiment result is shown in Figure \ref{fig:result}.
In Figure \ref{fig:result}\subref{fig:result1}, the horizontal axis shows the number of real crossings in the diagrams, and the vertical axis shows the running time.
In Figure \ref{fig:result}\subref{fig:result2}, the horizontal axis shows the number of tetrahedra in the trinagulations obtained by the diagrams, and the vertical axis shows the running time.

We estimate a linear relationship of the pairs $\{(c, \log_2 t)\}$ using the least squares method, where $c$ is the number of real crossings of an input diagram $D$ and $t$ is the running time when $D$ is an input diagram.
As a result, the linear function $\log_2 t = 2.250c - 0.887$ is obtained.
The function $t = 2^{2.250c - 0.887}$ is shown in Figure \ref{fig:result}\subref{fig:result1}.
We also estimate a linear relationship of the pairs $\{(n, \log_2 t)\}$, where $n$ is the number of tetrahedra in a triangulation, and the linear function $\log_2 t = 1.674n - 31.696$ is obtained.
The function $t = 2^{1.674n - 31.696}$ is shown in Figure \ref{fig:result}\subref{fig:result2}.

\begin{figure}[htbp]
    \begin{minipage}{0.49\textwidth}
        \centering
        \includegraphics[width=\textwidth]{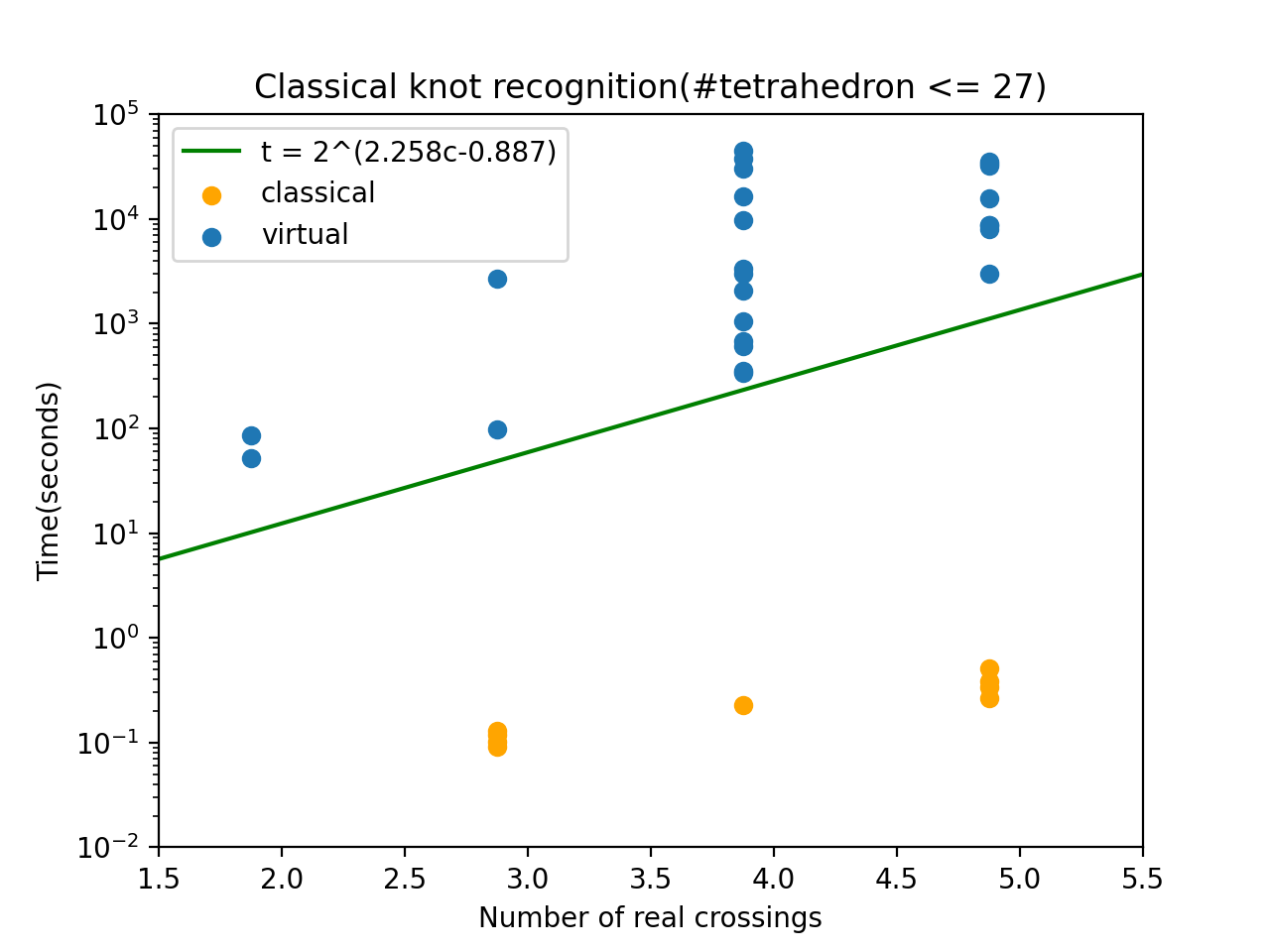}
        \subcaption{The relationship of the number of real crossings and the running time}
        \label{fig:result1}
    \end{minipage}
    \begin{minipage}{0.49\textwidth}
        \centering
        \includegraphics[width=\textwidth]{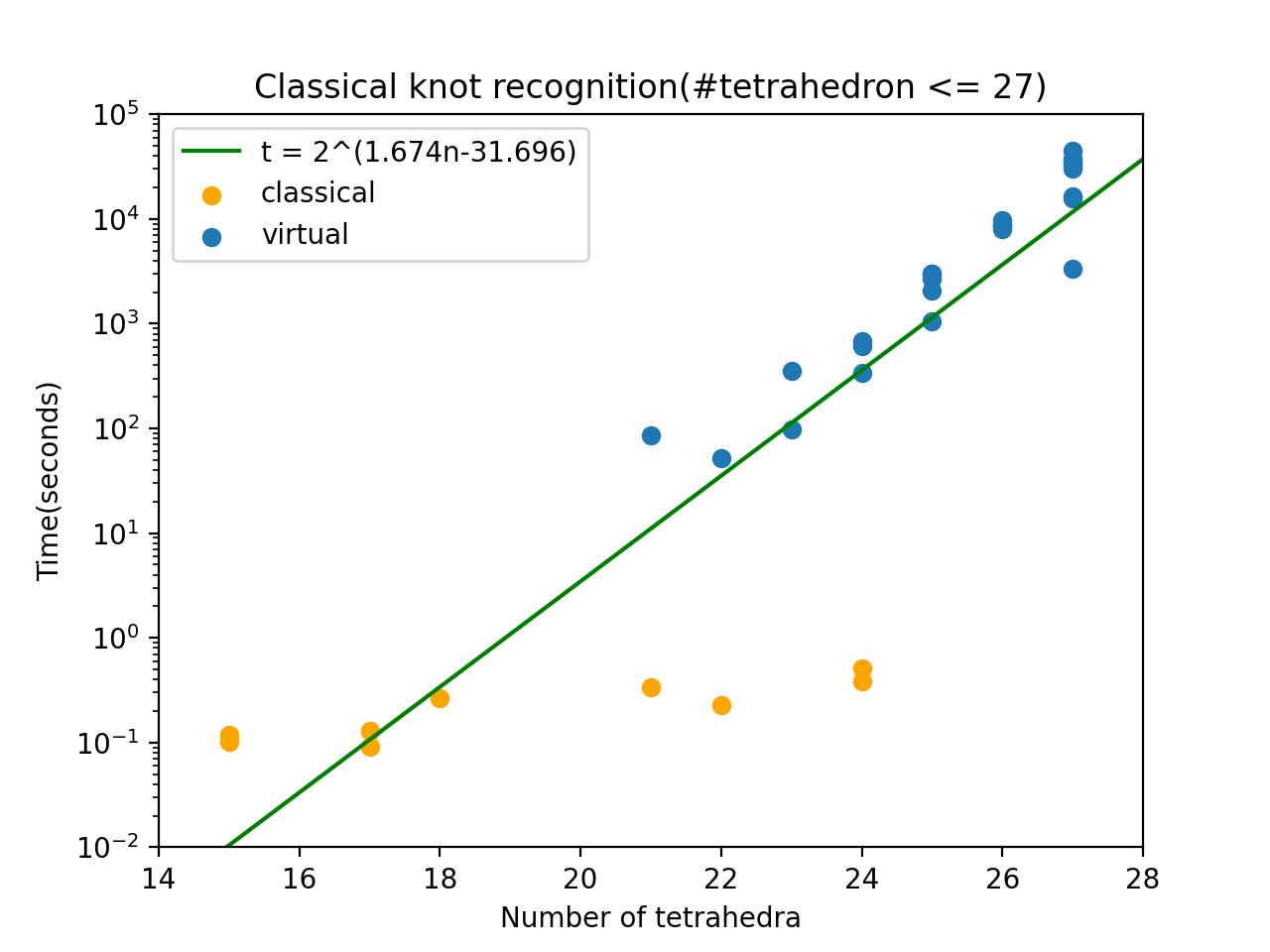}
        \subcaption{The relationship of the number of tetrahedra in a triangulation and the running time}
        \label{fig:result2}
    \end{minipage}
    \caption{The running time of Algorithm \ref{alg:classical}}
    \label{fig:result}
\end{figure}

Furthermore, Figure \ref{fig:tetNum} shows the relationship of the number of real crossings in an input diagram and the number of tetrahedra in the triangulation obtained from the diagram in the second step.
In Figure \ref{fig:tetNum}, the horizontal axis shows the number of real crossings, and the vertical axis shows the number of tetrahedra.
In addition, the results of diagrams whose supporting genera are zero, one, two, and tree are shown in violet, blue, green, and yellow, respectively.
\begin{figure}[htbp]
    \centering
    \includegraphics[width=0.6\textwidth]{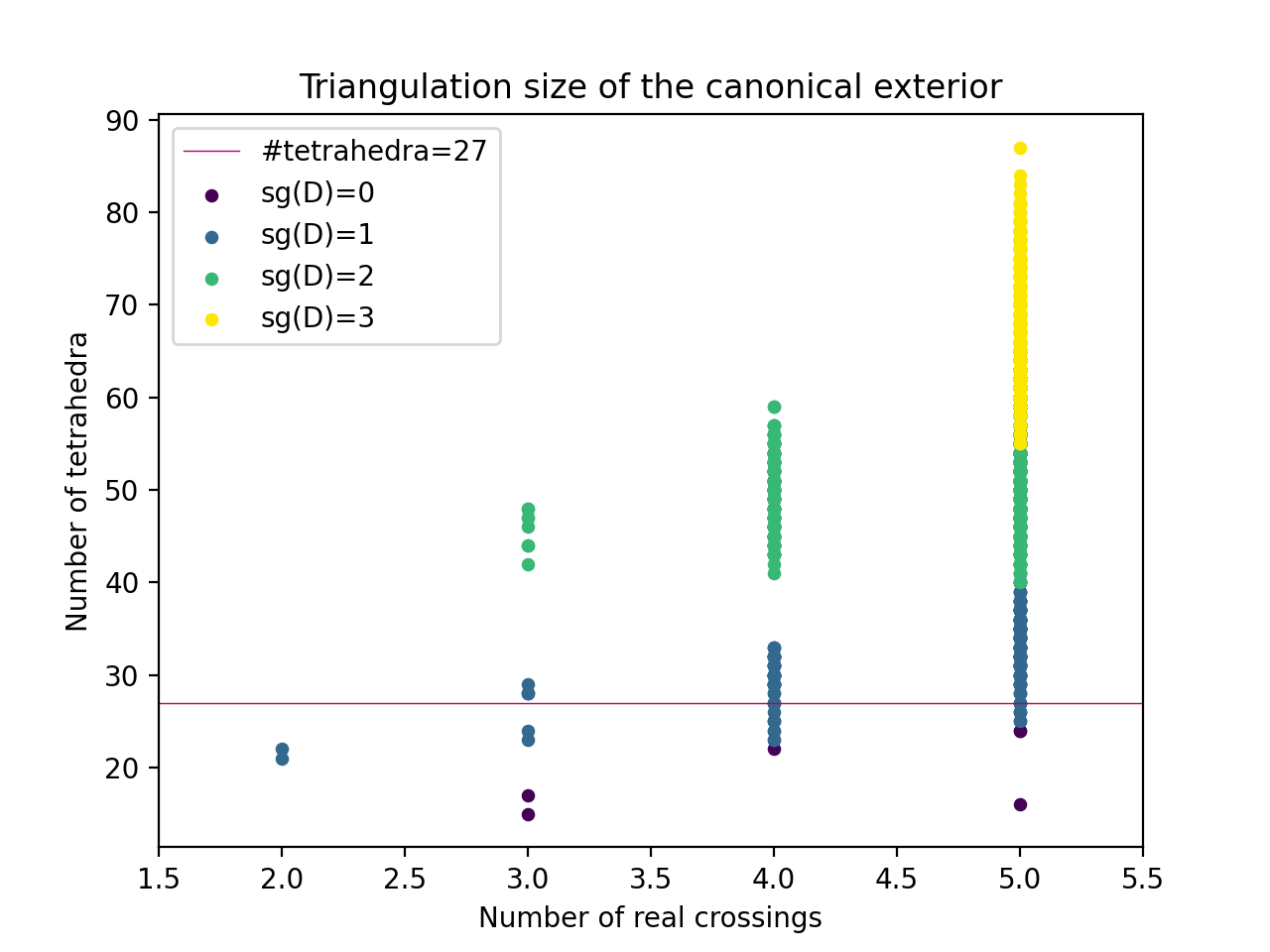}
    \caption{The relationship of the number of tetrahedra and the number of real crossings}
    \label{fig:tetNum}
\end{figure}

\subsection{Consideration}
As shown in Figure \ref{fig:result}\subref{fig:result2}, it seems that the average of the running time is bounded by $2^{\mathcal{O}(n)}$, where $n$ is the number of tetrahedra in the triangulation constructed in the second step of the experiment.
From Figure \ref{fig:tetNum}, we see that the number of tetrahedra in the triangulation in the second step increases linearly with increasing the number of real crossings.
Therefore, it is expected that the average of the running times is bounded by $2^{\mathcal{O}(c)}$, where $c$ is the number of real crossings in an input diagram.
This result implies that the time complexity of Algorithm \ref{alg:classical} is also bounded by $2^{\mathcal{O}(c)}$.

Figure \ref{fig:tetNum} shows that a triangulation which has many tetrahedra is obtained from a diagram whose supporting genus is large.
In this experiment, the running time is measured for diagrams whose canonical exterior triangulations constructed in the second step have 27 or less tetrahedra.
As shown in Figure \ref{fig:tetNum}, each of such diagrams has small supporting genus.
Therefore, in order to solve classical knot recognition for a diagram with larger supporting genus, we need to develop a faster algorithm.


\section{Discussion}
In this study, we showed that classical knot recognition is in NP, gave the exponential time algorithm for classical knot recognition, and tested Algorithm \ref{alg:classical} for the diagrams whose triangulations of the canonical exteriors have few tetrahedra.

We have two future tasks.
The first one is to show that classical \textit{link} recognition is in NP.
In the case where an input is a virtual knot diagram, we showed that if there is a vertical essential annulus in the canonical exterior of the input diagram, then there is a vertex surface which is a vertical essential annulus or a classicalization annulus with respect to a triangulation of the exterior.
Note that a classicalization annulus is defined for only the exterior of a \textit{knot} in a thickened closed orientable surface.
Since a vertex surface is encoded with polynomial length, a vertical essential annulus or a classicalization annulus is used for a polynomial length witness if and only if the input is a diagram of a classical knot.
On the other hand, in the case where an input is a diagram of a virtual link which has two or more components, it is not known whether there is a vertical essential annulus which is a vertex surface with respect to a triangulation of the canonical exterior of the virtual link diagram even though there is a vertical essential annulus in the exterior.
For this reason, it is not known that there is a polynomial length witness for an input of classical link recognition.

The other future task is to give a faster algorithm for classical knot recognition.
The experimental result shows that the running time increases with increasing the number of tetrahedra exponentially.
The reason why it takes much time to run the algorithm is that the enumeration of vertex surfaces takes much time.
However, the experiment in \cite{B_vertex} suggests that the number of vertex surfaces increases exponentially with increasing the size of a triangulation.
Thus, as long as we enumerate vertex surfaces, it seems that a faster algorithm can not be constructed.
Therefore, we consider to find a vertex surface by solving an integer programming problem as with the algorithm for unknot recognition in \cite{B_unknot}.
In \cite{B_unknot}, Burton and Ozlen proposed a fast algorithm for unknot recognition using an integer programming problem and the branch and bound method.
It is expected that we can solve classical knot recognition quickly by using this method.
\section*{Acknowledgement}
The authors thanks to Shin Satoh for useful conversations.
\bibliographystyle{plain}
\bibliography{cite}
\end{document}